\documentclass[12pt,a4paper]{article}
\usepackage{amsmath,amstext,amssymb,amscd, color}

\newcommand{\Xcomment}[1]{}

\newtheorem{theorem}{Theorem}[section]
\newtheorem{lemma}[theorem]{Lemma}
\newtheorem{corollary}[theorem]{Corollary}

\newtheorem{prop}[theorem]{Proposition}

\makeatletter \@addtoreset{equation}{section} \makeatother

\newenvironment{proof}{\noindent{\bf Proof}~}%
{\hfill$\qed$\medskip}

\def\qed{ \ \vrule width.1cm height.3cm depth0cm}

\newenvironment{numitem1}{\refstepcounter{equation}\begin{enumerate}%
\item[(\thesection.\arabic{equation})]}{\end{enumerate}}

\newcommand{\refeq}[1]{(\ref{eq:#1})}  

 \makeatletter
\renewcommand{\section}{\@startsection{section}{1}{0pt}%
{-3.5ex plus -1ex minus -.2ex}{2.3ex plus .2ex}%
{\normalfont\Large}}
 \makeatother

 \makeatletter
\renewcommand{\subsection}{\@startsection{subsection}{2}{0pt}%
{-3.0ex plus -1ex minus -.2ex}{1.5ex plus .2ex}%
{\normalfont\normalsize\bf}}
 \makeatother

\def\Rset{{\mathbb R}}

\def\Cscr{{\cal C}}
\def\Dscr{{\cal D}}
\def\Escr{{\cal E}}
\def\Fscr{{\cal F}}

\def\Iscr{{\cal I}}
\def\Lscr{{\cal L}}

\def\Pscr{{\cal P}}

\def\Rscr{{\cal R}}

\def\Zscr{{\cal Z}}

\def\tilde{\widetilde}
\def\hat{\widehat}
\def\bar{\overline}
\def\eps{\epsilon}

\def\lessdoteq{\,{\underline{\lessdot}}\,}

\def\wsep{$\hbox{\unitlength=1mm\begin{picture}(8,4)%
\put(3.5,1){\oval(6,3)}\put(1.2,0){{\rm ws}}\end{picture}}$}

\begin{document}

 \begin{center}
{\large\bf On maximal weakly separated set-systems}
 \end{center}

 \begin{center}
{\sc Vladimir~I.~Danilov}\footnote[2] {Central Institute of Economics and
Mathematics of the RAS, 47, Nakhimovskii Prospect, 117418 Moscow, Russia;
emails: danilov@cemi.rssi.ru (V.I.~Danilov); koshevoy@cemi.rssi.ru
(G.A.~Koshevoy).},
{\sc Alexander~V.~Karzanov}\footnote[3]{Institute for System Analysis of the
RAS, 9, Prospect 60 Let Oktyabrya, 117312 Moscow, Russia; email:
sasha@cs.isa.ru.},
{\sc Gleb~A.~Koshevoy}$^2$
\end{center}


 \begin{quote}
 {\bf Abstract.} \small
For a permutation $\omega\in S_n$, Leclerc and Zelevinsky~\cite{LZ} introduced
a concept of $\omega$-{\em chamber weakly separated collection} of subsets of
$\{1,2,\ldots,n\}$ and conjectured that all inclusion-wise maximal collections
of this sort have the same cardinality $\ell(\omega)+n+1$, where $\ell(\omega)$
is the length of $\omega$. We answer affirmatively this conjecture and present
a generalization and additional results.


  \end{quote}

\parskip=3pt

\section{Introduction}  \label{sec:intr}

For a positive integer $n$, let $[n]$ denote the ordered set of elements
$1,2,\ldots,n$. We deal with two binary relations on subsets of $[n]$.
  \begin{numitem1}
For $A,B\subseteq[n]$, we write:
  \begin{itemize}
\item[(i)] $A\lessdot B$ if $B-A$ is
nonempty and if $i<j$ holds for any $i\in A-B$ and $j\in B-A$ (where $A'-B'$
stands for the set difference $\{i'\colon A'\ni i'\not\in B'\})$;
\item[(ii)] $A\rhd B$ if both $A-B$ and $B-A$
are nonempty and if the set $B-A$ can be (uniquely) expressed as a disjoint
union $B'\sqcup B''$ of nonempty subsets so that $B'\lessdot A-B\lessdot B''$.
  \end{itemize}
  \label{eq:2relat}
   \end{numitem1}
Note that these relations need not be transitive in general. For example,
$13\lessdot 23\lessdot 24$ but $13\,\lefteqn{\not}\!\lessdot 24$; similarly,
$346\rhd 256\rhd 157$ but $346\,\lefteqn{\not}\!\rhd 157$, where for brevity we
write $i\ldots j$ instead of $\{i\}\cup\ldots\cup\{j\}$. \medskip

 \noindent \textbf{Definition 1}
~Sets $A,B\subseteq[n]$ are called \emph{weakly separated} (from
each other) if either $A\lessdot B$, or $B\lessdot A$, or $A\rhd
B$ and $|A|\ge B$, or $B\rhd A$ and $|B|\ge|A|$, or $A=B$. A
collection $\Cscr\subseteq 2^{[n]}$ is called weakly separated if
any two of its members are weakly separated. \medskip

We will usually abbreviate the term ``weakly separated collection'' to
``ws-collection''. When a set $X$ is weakly separated from a set $Y$ (from all
sets in a collection $\Cscr$), we write $X\wsep Y$ (resp. $X\wsep \Cscr$).
\medskip

 \noindent \textbf{Definition 2}
~Let $\omega$ be a permutation on $[n]$. A subset $X\subset [n]$ is called an
$\omega$-{\em chamber set} if $i\in X$, $j<i$ and $\omega(j)<\omega(i)$ imply
$j\in X$. A ws-collection $\Cscr\subseteq 2^{[n]}$ is called an {\em
$\omega$-chamber ws-collection} if all members of $\Cscr$ are $\omega$-chamber
sets.
\medskip

These notions were introduced by Leclerc and Zelevinsky in~\cite{LZ} where
their importance is demonstrated, in particular, in connection with the problem
of characterizing quasicommuting quantum flag minors of a generic $q$-matrix.
(Note that~\cite{LZ} deals with a relation $\prec$ which is somewhat different
from $\lessdot$; nevertheless, Definition~1 is consistent with the
corresponding definition in~\cite{LZ}. The term ``$\omega$-chamber'' for a set
$X$ is motivated by the fact that such an $X$ corresponds to a face, or
\emph{chamber}, in the pseudo-line arrangement related to some reduced word for
$\omega$; see~\cite{BFZ}.)

Let $\ell(\omega)$ denote the \emph{length} of $\omega$, i.e. the number of
pairs $i<j$ such that $\omega(j)<\omega(i)$ (\emph{inversions}). It is shown
in~\cite{LZ} that the cardinality $|\Cscr|$ of any $\omega$-chamber
ws-collection $\Cscr\subseteq 2^{[n]}$ does not exceed $\ell(\omega)+n+1$ and
is conjectured (Conjecture~1.5 there) that this bound is achieved by \emph{any}
(inclusion-wise) maximal collection among these:

 \begin{itemize}
\item[(C)] For any permutation $\omega$ on $[n]$, ~$|\Cscr|=\ell(\omega)+n+1$
holds for all maximal $\omega$-chamber ws-collections $\Cscr\subseteq 2^{[n]}$.
  \end{itemize}
The main purpose of this paper is to answer this conjecture.
\medskip

\noindent\textbf{Theorem A} ~\emph{Statement~(C) is valid.} \medskip

The \emph{longest} permutation $\omega_0$ on $[n]$ (defined by $i\mapsto
n-i+1$) is of especial interest, many results in~\cite{LZ} are devoted just to
this case, and~(C) with $\omega=\omega_0$ has been open so far as well. Since
$\omega_0(j)>\omega_0(i)$ for any $j<i$, no ``chamber conditions'' are imposed
in this case in essence, i.e. the set of $\omega_0$-chamber ws-collections
consists of all ws-collections. The length of $\omega_0$ is equal to
$\binom{n}{2}$, so the above upper bound turns into $\binom{n+1}{2}+1$. Then
the assertion in the above theorem is specified as follows. \medskip

\noindent\textbf{Theorem B} \emph{~All maximal ws-collections in $2^{[n]}$ have
the same cardinality $\binom{n+1}{2}+1$.} \medskip

We refer to a ws-collection of this cardinality as a \emph{largest} one and
denote the set of these collections by ${\bf W}_n$. An important instance is
the set $\Iscr_n$ of all intervals $[p..q]:=\{p,p+1,\ldots,q\}$ in $[n]$,
including the ``empty interval'' $\emptyset$. One can see that for a
ws-collection $\Cscr$, the collection $\{[n]-X\colon X\in\Cscr\}$ is weakly
separated as well; it is called the \emph{complementary ws-collection} of
$\Cscr$ and denoted by co-$\Cscr$. Therefore, co-$\Iscr_n$, the set of
\emph{co-intervals} in $[n]$, is also a largest ws-collection. In~\cite{LZ} it
is shown that ${\bf W}_n$ is preserved under so-called \emph{weak raising
flips} (which transform one collection into another) and is conjectured
(Conjecture~1.8 there) that in the poset structure on ${\bf W}_n$ induced by
such flips, $\Iscr_n$ and co-$\Iscr_n$ are the unique minimal and unique
maximal elements, respectively. That conjecture was affirmatively answered
in~\cite{DKK-09}.

We will show that Theorem~A can be obtained relatively easily from Theorem~B.
The proof of the latter theorem is more intricate and takes the most part of
this paper. The breakthrough step on this way consists in showing the following
lattice property for a largest ws-collection $\Cscr$: the partial order
$\prec^\ast$ on $\Cscr$ given by $A\prec^\ast B\Longleftrightarrow A\lessdot
B\; \&\; |A|\le |B|$ forms a lattice. To prove this and some other intermediate
results we will essentially use results and constructions from our previous
work~\cite{DKK-09}.

The main result in~\cite{DKK-09} shows the coincidence of four classes of
collections: (i) the set of \emph{semi-normal bases} of tropical Pl\"ucker
functions on $2^{[n]}$; (ii) the set of \emph{spectra} of certain collections
of $n$ curves on a disc in the plane, called \emph{proper wirings} (which
generalize commutation classes of pseudo-line arrangements); (iii) the set
${\bf ST}_n$ of \emph{spectra} of so-called \emph{generalized tilings} on an
$n$-zonogon (a $2n$-gon representable as the Minkowsky sum of $n$ generic
line-segments in the plane); and (iv) the set ${\bf W}_n$. Objects mentioned
in~(i) and~(ii) are beyond our consideration in this paper (for definitions,
see~\cite{DKK-09}), but we will extensively use the generalized tiling model
and rely on the equality ${\bf ST}_n={\bf W}_n$. Our goal is to show the
following property: \medskip

\emph{Any ws-collection can be extended to the spectrum of some generalized
tiling}, \medskip

\noindent whence Theorem~B will immediately follow. Due to this property,
generalized tilings give a geometric model for ws-collections.

Roughly speaking, a generalized tiling, briefly called a \emph{g-tiling}, is a
certain generalization of the notion of a \emph{rhombus tiling}. While the
latter is a subdivision of an $n$-zonogon $Z$ into rhombi, the former is a
cover of $Z$ with rhombi that may overlap in a certain way. (It should be noted
that rhombus tilings have been well studied, and one of important properties of
these is that their ``spectra'' turn out to be exactly the maximal strongly
separated collections, where $\Cscr\subseteq 2^{[n]}$ is called \emph{strongly
separated} if any two members of $\Cscr$ obey relation $\lessdot$ as
in~\refeq{2relat}(i). Leclerc and Zelevinsky explored such collections
in~\cite{LZ} in parallel with ws-collections. In particular, they established a
counterpart of Theorem~A saying that the cardinality of any maximal strongly
separated $\omega$-chamber collection is exactly $\ell(\omega)+n+1$. For a
wider discussion and related topics, see
also~\cite{BFZ,DKK-08,El,Fan,Kn,Stem}.)

This paper is organized as follows. Section~\ref{sec:omega} explains how to
reduce Theorem~A to Theorem~B. Then we start proving the latter theorem; the
whole proof lasts throughout Sections~\ref{sec:tiling}--\ref{sec:2posets}.
Section~\ref{sec:tiling} recalls the definitions of a g-tiling and its spectrum
and gives a review of properties of these objects established in~\cite{DKK-09}
and important for us. Sections~\ref{sec:proof} proves Theorem~B in the
assumption of validity of the above-mentioned lattice property of largest
ws-collections. Then there begins a rather long way of proving the latter
property stated in Theorem~\ref{tm:prec-lat}. In Section~\ref{sec:aux_graph} we
associate to a g-tiling $T$ a certain acyclic directed graph $\Gamma=\Gamma_T$
whose vertex set is the spectrum $\mathfrak{S}_T$ of $T$ (forming a largest
ws-collection) and show that the natural partial order $\prec_\Gamma$ induced
by $\Gamma$ forms a lattice, which is not difficult. Section~\ref{sec:2posets}
is devoted to proving the crucial property that $\prec_\Gamma$ coincides with
the partial order $\prec^\ast$ on $\mathfrak{S}_T$, thus yielding
Theorem~\ref{tm:prec-lat} and completing the proof of Theorem~B; this is
apparently the most sophisticated part of the paper where especial
combinatorial techniques of handling g-tilings are elaborated. The concluding
Section~\ref{sec:concl} presents additional results and finishes with a
generalization of Theorem~A. This generalization (Theorem~A$'$) deals with two
permutations $\omega',\omega$ on $[n]$ such that each inversion of $\omega'$ is
an inversion of $\omega$ (in this case the pair $(\omega',\omega)$ is said to
obey the \emph{weak Bruhat relation}). It asserts that all maximal
ws-collections whose members $X$ are $\omega$-chamber sets and simultaneously
satisfy the condition: $i\in X\; \&\; j>i\; \&\; \omega'(j)<\omega'(i)
\Longrightarrow j\in X$, have the same cardinality, namely,
$\ell(\omega)-\ell(\omega')+n+1$. When $\omega'$ is the identical permutation
$i\mapsto i$, this turns into Theorem~A.
\smallskip

Later on, for a set $X\subset[n]$, distinct elements $i,\ldots,j\in[n]-X$ and
an element $k\in X$, we usually abbreviate $X\cup\{i\}\cup\ldots\cup\{j\}$ as
$Xi\ldots j$, and $X-\{k\}$ as $X-k$. \smallskip

{\bf Acknowledgments}. We thank the referees for useful remarks and
suggestions. This work is partially supported by the Russian Foundation of
Basic Research and the Ministry of National Education of France (project RFBR
10-01-9311-CNRSL-a). A part of this research was done while the third author
was visiting RIMS, Kyoto University and IHES, Bures-sur-Yvette and he thanks
these institutes for hospitality.

\section{Maximal $\omega$-chamber ws-collections}  \label{sec:omega}

In this section we explain how to derive Theorem~A from Theorem~B. The proof
given here is direct and relatively short, though rather technical. Another
proof, which is more geometric and appeals to properties of tilings, will be
seen from a discussion in Section~\ref{sec:concl}. Let $\omega$ be a
permutation on $[n]$.

For $k=0,\ldots,n$, let $I_\omega^k$ denote the set
$\omega^{-1}[k]=\{i\colon \omega(i)\in[k]\}$, called $k$-th
\emph{ideal} for $\omega$ (it is an ideal of the linear order on
$[n]$ given by: $i\prec j$ if $\omega(i)<\omega(j)$). We will use
the following auxiliary collection
   \begin{equation} \label{eq:C0}
\Cscr^0=\Cscr^0_\omega:=\{I_\omega^k\cap[j..n]\colon 1\le j\le
\omega^{-1}(k),\; 0\le k\le n\},
  \end{equation}
where possible repeated sets are ignored and where
$I_\omega^0:=\emptyset$. The role of this collection is emphasized by
the following

  \begin{theorem} \label{tm:checker}
Let $X\subset[n]$ and $X\not\in\Cscr^0$. The following properties are
equivalent:

{\rm(i)} $X\wsep \Cscr^0$ (i.e. $X$ is weakly separated from all sets in
$\Cscr^0$);

{\rm(ii)} $X$ is an $\omega$-chamber set.
  \end{theorem}

Due to this property, we call $\Cscr^0$ the (canonical) $\omega$-{\em checker}.
(We shall explain in Section~\ref{sec:concl} that $\Cscr^0$ is chosen to be the
spectrum of a special tiling and that there are other tilings whose spectra can
be taken as a checker in place of $\Cscr^0$ in Theorem~\ref{tm:checker}; see
Corollary~\ref{cor:chamb-til}.) It is easy to verify that: (a) $\Cscr^0$ is a
ws-collection; (b) its subcollection
   \begin{equation} \label{eq:Yomega}
   \Iscr_\omega:=\{I_\omega^0,I_\omega^1,\ldots,I_\omega^n\}
   \end{equation}
consists of $\omega$-chamber sets; and (c) any member of
$\Cscr^0-\Iscr_\omega$ is not an $\omega$-chamber set.

Relying on Theorems~B and~\ref{tm:checker}, we can prove Theorem~A as follows.
Given an $\omega$-chamber ws-collection $\Cscr\subset 2^{[n]}$, consider
$\Cscr':=\Cscr\cup \Cscr^0$. By Theorem~\ref{tm:checker}, $\Cscr'$ is a
ws-collection. Also $\Cscr\cap \Cscr^0\subseteq \Iscr_\omega$, in view of~(c)
above. Extend $\Cscr'$ to a largest ws-collection $\Dscr$, which is possible by
Theorem~B. Let $\Dscr':=(\Dscr-\Cscr_0)\cup\Iscr_\omega$. Then $\Dscr'$
includes $\Cscr$ and is an $\omega$-chamber ws-collection by
Theorem~\ref{tm:checker}. Since the cardinality of $\Dscr'$ is always the same
(as it is equal to $\binom{n+1}{2}+1-|\Cscr^0-\Iscr_\omega|$), ~$\Dscr'$ is a
largest $\omega$-chamber ws-collection, and now Theorem~A follows from the fact
that the upper size bound $\ell(\omega)+n+1$ is achieved by some
$\omega$-chamber weakly (or even strongly) separated collection.
\smallskip

The rest of this section is devoted to proving Theorem~\ref{tm:checker}. The
proof of implication (i)$\Rightarrow$(ii) falls into three lemmas. Let
$X\subset[n]$ be such that $X\not\in \Cscr^0$ and $X\wsep\Cscr^0$, and let
$k:=|X|$ and $Y:=I_\omega^k$.
  \begin{lemma} \label{lm:YprecX}
Neither $Y\lessdot X$ nor $Y\rhd X$ can take place.
  \end{lemma}
  \begin{proof}
Suppose $Y\lessdot X$ or $Y\rhd X$. Let $k'$ be the maximum number such that
either $Y'\lessdot X$ or $Y'\rhd X$, where $Y':=I_\omega^{k'}$. Then $k\le k'$
and $Y\subseteq Y'$. Define
   $$
\Delta:=Y'-X \quad \mbox{and} \quad \Delta':=\{i\in X-Y'\colon \Delta\lessdot
\{i\}\}.
   $$
Then $\Delta,\Delta'\ne\emptyset$ and $|\Delta|\ge|\Delta'|$. The maximality of
$k'$ implies that $|\Delta'|=1$ and that the unique element of $\Delta'$, say,
$a$, is exactly $\omega^{-1}(k'+1)$ (in all other cases either $k'+1$ fits as
well, or $X$ is not weakly separated from $I_\omega^{k'+1}$).

Let $b$ be the \emph{maximal} element in $\Delta$. Then $a>b$ and
the element $\tilde k:=\omega(b)$ is at most $k'$. We assert that
there is no $d\in X$ such that $d<b$. To see this, consider the
sets $I_\omega^{\tilde k}$ and $Z:=I_\omega^{\tilde k}\cap
[b..n]$. Then $Z\in\Cscr^0$ and $Z\subseteq I_\omega^{\tilde
k}\subseteq Y'$. Therefore, $a\in X-Z$. Also $b\in Z-X$. Moreover,
$Z-X=\{b\}$, by the maximality of $b$. Now if $X$ contains an
element $d<b$, then we have $|X|>|Z|$ (in view of $a,d\in X-Z$ and
$|Z-X|=1$) and $Z\rhd X$ (in view of $d<b<a$), which contradicts
$X\wsep\Cscr^0$.

Thus, all elements of $X$ are greater than $b$. This and $X-Y'=\{a\}$ imply
that the set $U:=Y'\cap[b..n]$ satisfies $X-U=\{a\}$ and $U-X=\{b\}$. Then $X$
coincides with the set $I_\omega^{k'+1}\cap[b+1..n]$. But the latter set
belongs to $\Cscr^0$ (since $\omega^{-1}(k'+1)=a>b$). So $X$ is a member of
$\Cscr^0$; a contradiction.
  \end{proof}

  \begin{lemma} \label{lm:XrhdY}
$X\rhd Y$ cannot take place.
  \end{lemma}
  \begin{proof}
Suppose $X\rhd Y$. Take the maximal $k'$ such that the set $Y':=I_\omega^{k'}$
satisfies $X-Y'\ne \emptyset$. Then $k'\ge k$ and $|X-Y'|=1$. Since $k\le k'$
implies $Y\subseteq Y'$, we have $X\rhd Y'$. Also $|Y'-X|\ge |Y-X|\ge 2$. Then
$|X|<|Y'|$, contradicting $X\wsep Y'$.
  \end{proof}

In view of $X\wsep Y$, Lemmas~\ref{lm:YprecX} and~\ref{lm:XrhdY} imply that
only the case $X\lessdot Y$ is possible.
  \begin{lemma} \label{lm:XprecY}
Let $X\lessdot Y$. Then $X$ is an $\omega$-chamber set.
  \end{lemma}
  \begin{proof}
Suppose that there exist $i\in X$ and $j\not\in X$ such that $j<i$ and
$\omega(j)<\omega(i)$. Consider possible cases. \smallskip

\emph{Case 1}: $i\in Y$. Then $\omega(j)<\omega(i)$ implies that $j\in Y$. Take
$d\in X-Y$ (existing since $|X|=|Y|$). Then $d<j$ (since $X\lessdot Y$ and
$j\in Y-X$). Let $Y':=I_\omega^{\omega(j)}$. We have $j\in Y'$, $i\not\in Y'$
and $d\not\in Y'$. This together with $X\wsep Y'$ and $d<j<i$ implies $Y'\rhd
X$. But $|X|=|Y|>|Y'|$ (in view of $\omega(j)<\omega(i)\le k$); a
contradiction.
\smallskip

\emph{Case 2}: $i,j\not\in Y$. Then $\omega(i),\omega(j)>k$. Take $a\in Y-X$.
Since $X\lessdot Y$, we have $a>i$. Also $\omega(a)\le k$. Let
$Y':=I_\omega^{\omega(j)}$. Then $|Y'|>|Y|$ (in view of $\omega(j)>k$). Also
$a,j\in Y'-X$ and $i\in X-Y'$. Therefore, $X\rhd Y'$, contradicting
$|X|=|Y|<|Y'|$.
\smallskip

Finally, the case with $i\not\in Y$ and $j\in Y$ is impossible since $i>j$ and
$X\lessdot Y$.
  \end{proof}

Thus, (i)$\Rightarrow$(ii) in Theorem~\ref{tm:checker} is proven. Now we prove
the other direction.
    \begin{lemma} \label{lm:Xchamb}
Let $X\subset[n]$ be an $\omega$-chamber set. Then $X\wsep\Cscr^0$.
  \end{lemma}
  \begin{proof}
Consider an arbitrary set $Y=I_\omega^k\cap[j..n]$ in $\Cscr^0$ (where $j\le
\omega^{-1}(k)$). One may assume that both $X-Y$ and $Y-X$ are nonempty. Let
$a\in Y-X$ and $b\in X-Y$. We assert that $a>b$ (whence $X\lessdot Y$ follows).

Indeed, $a\in Y$ implies $\omega(a)\le k$. If $b\not\in I_\omega^k$, then
$\omega(b)>k\ge\omega(a)$. In case $a<b$ we would have $a\in X$, by the
$\omega$-chamberness of $X$. Therefore, $a>b$, as required.

Now suppose $b\in I_\omega^k$. Then $b\in I_\omega^k-[j..n]$, and therefore,
$b<j$. Since $j\le a$, we again obtain $a>b$.
  \end{proof}

This completes the proof of Theorem~\ref{tm:checker}, reducing
Theorem~A to Theorem~B.
\medskip


\section{Generalized tilings and their properties}  \label{sec:tiling}

As mentioned in the Introduction, the proof of Theorem~B will essentially rely
on results on generalized tilings from~\cite{DKK-09}. This section starts with
definitions of such objects and their spectra. Then we review properties of
generalized tilings that will be important for us later:
Subsection~\ref{ssec:Tprop} describes rather easy consequences from the
defining axioms and Subsection~\ref{ssec:s-c-e} is devoted to less evident
properties.

\subsection{Generalized tilings}  \label{ssec:tiling}

Tiling diagrams that we deal with live within a zonogon, which is defined as
follows.

In the upper half-plane $\Rset\times \Rset_+$, take $n$ non-colinear vectors
$\xi_1,\ldots,\xi_n$ so that:
  \begin{numitem1} (i) $\xi_1,\ldots,\xi_n$ follow in this order clockwise around
$(0,0)$, and \\
 (ii) all integer combinations of these vectors are different.
  \label{eq:xi}
  \end{numitem1}
Then the set
  $$
Z=Z_n:=\{\lambda_1\xi_1+\ldots+ \lambda_n\xi_n\colon \lambda_i\in\Rset,\;
0\le\lambda_i\le 1,\; i=1,\ldots,n\}
  $$
is a $2n$-gon. Moreover, $Z$ is a {\em zonogon}, as it is the sum of $n$
line-segments $\{\lambda\xi_i\colon 1\le \lambda\le 1\}$, $i=1,\ldots,n$. Also
it is the image by a linear projection $\pi$ of the solid cube $conv(2^{[n]})$
into the plane $\Rset^2$, defined by $\pi(x):=x_1\xi_1+\ldots +x_n\xi_n.$ The
boundary $bd(Z)$ of $Z$ consists of two parts: the {\em left boundary} $lbd(Z)$
formed by the points (vertices) $z^\ell_i:=\xi_1+\ldots+\xi_i$ ($i=0,\ldots,n$)
connected by the line-segments $z^\ell_{i-1}z^\ell_i:=z^\ell_{i-1}+\{\lambda
\xi_i\colon 0\le\lambda\le 1\}$, and the {\em right boundary} $rbd(Z)$ formed
by the points $z^r_i:=\xi_{i+1}+\ldots+\xi_n$ ($i=0,\ldots,n$) connected by the
line-segments $z^r_iz^r_{i-1}$. So $z^\ell_0=z^r_n$ is the minimal vertex of
$Z$, denoted as $z_0$, and $z^\ell_n=z^r_0$ is the maximal vertex, denoted as
$z_n$. We direct each segment $z^\ell_{i-1}z^\ell_i$ from $z^\ell_{i-1}$ to
$z^\ell_i$ and direct each segment $z^r_iz^r_{i-1}$ from $z^r_i$ to
$z^r_{i-1}$.

When it is not confusing, a subset $X\subseteq [n]$ is identified with the
corresponding vertex of the $n$-cube and with the point $\sum_{i\in X}\xi_i$ in
the zonogon $Z$ (and we will usually use capital letters when we are going to
emphasize that a vertex (or a point) is considered as a set). Due
to~\refeq{xi}(ii), all such points in $Z$ are different.

By a {\em (pure) tiling diagram} we mean a subdivision $T$ of $Z$ into
\emph{tiles}, each being a parallelogram of the form
$X+\{\lambda\xi_i+\lambda'\xi_j\colon 0\le \lambda,\lambda'\le 1\}$ for some
$i<j$ and some subset $X\subset[n]$ (regarded as a point in $Z$); so the tiles
are pairwise non-overlapping (have no common interior points) and their union
is $Z$. A tile $\tau$ determined by $X,i,j$ is called an $ij$-{\em tile} at $X$
and denoted by $\tau(X;i,j)$. According to a natural visualization of $\tau$,
its vertices $X,Xi,Xj,Xij$ are called the {\em bottom, left, right, top}
vertices of $\tau$ and denoted by $b(\tau)$, $\ell(\tau)$, $r(\tau)$,
$t(\tau)$, respectively. The edge from $b(\tau)$ to $\ell(\tau)$ is denoted by
$b\ell(\tau)$, and the other three edges of $\tau$ are denoted as
$br(\tau),\ell t(\tau),rt(\tau)$ in a similar way.

In fact, it is not important for our purposes which set of base vectors $\xi_i$
is chosen, subject to~\refeq{xi}. (In works on a similar subsect, it is most
often when the $\xi_i$ are assumed to have equal Euclidean norms; in this case
each tile forms a rhombus and $T$ is usually referred to as a \emph{rhombus
tiling}.) However, to simplify technical details and visualization, it will be
convenient for us to assume that these vectors always have {\em unit height},
i.e. each $\xi_i$ is of the form $(a_i,1)$. Then each tile becomes a
parallelogram of height 2. Accordingly, we say that: a point (subset)
$Y\subseteq[n]$ is of {\em height} $|Y|$; the set of vertices of tiles in $T$
having height $h$ forms $h$-th {\em level}; and a point $Y$ {\em lies on the
right} from a point $Y'$ if $|Y|=|Y'|$ and $\sum_{i\in Y}\xi_i\ge \sum_{i\in
Y'}\xi_i$.

In a {\em generalized tiling}, or a \emph{g-tiling}, some tiles may overlap. It
is a collection $T$ of tiles $\tau(X;i,j)$ which is partitioned into two
subcollections $T^w$ and $T^b$, of {\em white} and {\em black\/} tiles,
respectively, obeying axioms (T1)--(T4) below. When $T^b=\emptyset$, ~$T$
becomes a pure tiling.

We associate to $T$ the directed graph $G_T=(V_T,E_T)$ whose vertices and edges
are, respectively, the points and line-segments occurring as vertices and sides
in the tiles of $T$ (not counting multiplicities). An edge connecting vertices
$X$ and $Xi$ is directed from the former to the latter; such an edge (parallel
to $\xi_i$) is called an edge with \emph{label} $i$, or an $i$-\emph{edge}
(\cite{DKK-09} uses the term ``color'' rather than ``label''). For a vertex
$v\in V_T$, the set of edges incident with $v$ is denoted by $E_T(v)$, and the
set of tiles having a vertex at $v$ is denoted by $F_T(v)$.
  \begin{itemize}
\item[(T1)] Each boundary edge of $Z$ belongs to
exactly one tile. Each edge in $E_T$ not contained in $bd(Z)$ belongs to
exactly two tiles. All tiles in $T$ are different, in the sense that no two
coincide in the plane.
  \end{itemize}
  \begin{itemize}
\item[(T2)] Any two white tiles having a common edge do not overlap, i.e.
they have no common interior point. If a white tile and a black tile
share an edge, then these tiles do overlap. No two black tiles share an
edge.
  \end{itemize}
See the picture; here all edges are directed up and the black tiles are drawn
in bold.
 \begin{center}
  \unitlength=1mm
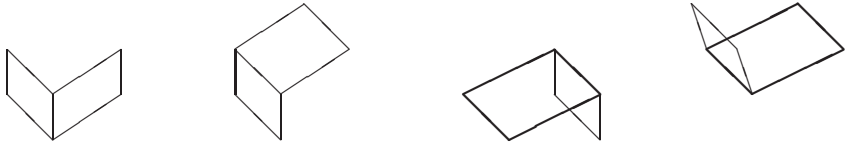
  \begin{picture}(120,18)
  \put(0,6){\line(0,1){6}}
  \put(6,0){\line(0,1){6}}
  \put(15,6){\line(0,1){6}}
  \put(0,6){\line(1,-1){6}}
  \put(0,12){\line(1,-1){6}}
  \put(6,0){\line(3,2){9}}
  \put(6,6){\line(3,2){9}}
  \put(30,6){\line(0,1){6}}
  \put(36,0){\line(0,1){6}}
  \put(30,6){\line(1,-1){6}}
  \put(30,12){\line(1,-1){6}}
  \put(39,18){\line(1,-1){6}}
  \put(30,12){\line(3,2){9}}
  \put(36,6){\line(3,2){9}}
  \put(72,6){\line(0,1){6}}
  \put(78,0){\line(0,1){6}}
  \put(72,6){\line(1,-1){6}}
{\thicklines
  \put(60,6){\line(2,1){12}}
  \put(66,0){\line(2,1){12}}
  \put(60,6){\line(1,-1){6}}
  \put(72,12){\line(1,-1){6}}
}
  \put(90,18){\line(1,-3){2}}
  \put(96,12){\line(1,-3){2}}
  \put(90,18){\line(1,-1){6}}
{\thicklines
  \put(92,12){\line(2,1){12}}
  \put(98,6){\line(2,1){12}}
  \put(92,12){\line(1,-1){6}}
  \put(104,18){\line(1,-1){6}}
}
  \end{picture}
   \end{center}

  \begin{itemize}
\item[(T3)] Let $\tau$ be a black tile. None of $b(\tau),t(\tau)$ is a vertex of
another black tile. All edges in $E_T(b(\tau))$ {\em leave} $b(\tau)$, i.e.
they are directed from $b(\tau)$. All edges in $E_T(t(\tau))$ {\em enter}
$t(\tau)$, i.e. they are directed to $t(\tau)$.
  \end{itemize}

We refer to a vertex $v\in V_T$ as {\em terminal} if $v$ is the bottom or top
vertex of some black tile. A nonterminal vertex $v$ is called {\em ordinary} if
all tiles in $F_T(v)$ are white, and {\em mixed} otherwise (i.e. $v$ is the
left or right vertex of some black tile). Note that a mixed vertex may belong,
as the left or right vertex, to several black tiles.

Each tile $\tau\in T$ corresponds to a square in the solid cube
$conv(2^{[n]})$, denoted by $\sigma(\tau)$: if $\tau=\tau(X;i,j)$ then
$\sigma(\tau)$ is the convex hull of the points $X,Xi,Xj,Xij$ in the
cube (so $\pi(\sigma(\tau))=\tau$). (T1) implies that the interiors of
these squares are pairwise disjoint and that $\cup(\sigma(\tau)\colon
\tau\in T)$ forms a 2-dimensional surface, denoted by $D_T$, whose
boundary is the preimage by $\pi$ of the boundary of $Z$. The last
axiom is:
  \begin{itemize}
\item[(T4)] $D_T$ is a disc, in the sense that it is homeomorphic to
 $\{x\in\Rset^2\colon x_1^2+x_2^2\le 1\}$.
  \end{itemize}

The {\em reversed} g-tiling $T^{rev}$ of a g-tiling $T$ is formed by replacing
each tile $\tau(X;i,j)$ of $T$ by the tile $\tau([n]-Xij;i,j)$ (or, roughly
speaking, by changing the orientation of all edges in $E_T$, in particular, in
$bd(Z)$). Clearly (T1)--(T4) remain valid for $T^{rev}$.

The {\em spectrum} of a g-tiling $T$ is the collection $\mathfrak{S}_T$
of (the subsets of $[n]$ represented by) {\em nonterminal} vertices in
$G_T$. Figure~\ref{fig:GT} illustrates an example of g-tilings; here
the unique black tile is drawn by thick lines and the terminal vertices
are indicated by black rhombi.
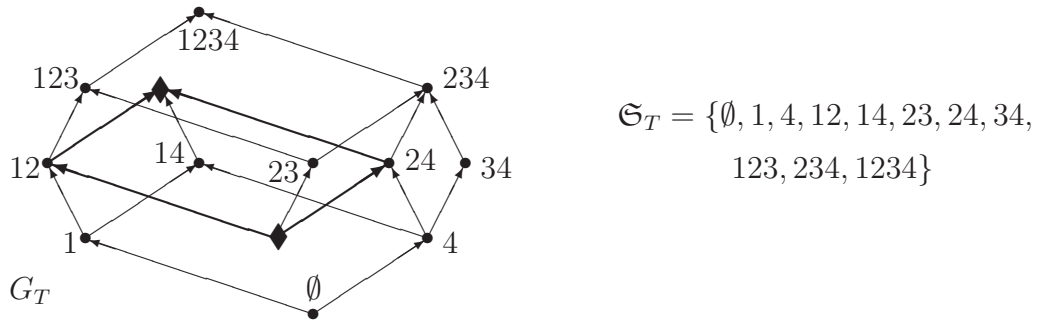
\begin{figure}[htb]
 \begin{center}
  \unitlength=1mm
  \begin{picture}(130,40)
   \put(40,0){\begin{picture}(50,40)
  \put(0,0){\vector(-3,1){29.5}}
  \put(-30,10){\vector(-1,2){4.7}}
  \put(-35,20){\vector(1,2){4.7}}
  \put(-30,30){\vector(3,2){14.7}}
  \put(15,30){\vector(-3,1){29.5}}
  \put(20,20){\vector(-1,2){4.7}}
  \put(15,10){\vector(1,2){4.7}}
  \put(0,0){\vector(3,2){14.7}}
  \put(0,0){\circle*{1.5}}
  \put(-30,10){\circle*{1.5}}
  \put(-6,9){$\blacklozenge$}
  \put(15,10){\circle*{1.5}}
  \put(-35,20){\circle*{1.5}}
  \put(-15,20){\circle*{1.5}}
  \put(0,20){\circle*{1.5}}
  \put(10,20){\circle*{1.5}}
  \put(20,20){\circle*{1.5}}
  \put(-30,30){\circle*{1.5}}
  \put(-21.5,28.5){$\blacklozenge$}
  \put(15,30){\circle*{1.5}}
  \put(-15,40){\circle*{1.5}}
  \put(0,20){\vector(-3,1){29.5}}
  \put(15,10){\vector(-3,1){29.5}}
  \put(-15,20){\vector(-1,2){4.5}}
  \put(15,10){\vector(-1,2){4.7}}
  \put(-5,10){\vector(1,2){4.7}}
  \put(10,20){\vector(1,2){4.7}}
  \put(-30,10){\vector(3,2){14.7}}
  \put(0,20){\vector(3,2){14.7}}
\thicklines{
  \put(-5,10){\vector(-3,1){29.5}}
  \put(10,20){\vector(-3,1){29.5}}
  \put(-35,20){\vector(3,2){14.7}}
  \put(-5,10){\vector(3,2){14.7}}
}
  \put(-1,2){$\emptyset$}
  \put(-33,8){1}
  \put(17,8){4}
  \put(-40,18){12}
  \put(-21,20){14}
  \put(-6,17.5){23}
  \put(12,19){24}
  \put(22,18){34}
  \put(-37,30){123}
  \put(17,30){234}
  \put(-18,35){1234}
  \put(-40,2){$G_T$}
    \end{picture}}
%
   \put(80,0){\begin{picture}(50,30)
  \put(0,25){$\mathfrak{S}_T=\{\emptyset,1,4,12,14,23,24,34,$}
  \put(15,18){$123,234,1234\}$}
    \end{picture}}
  \end{picture}
   \end{center}
\caption{A g-tiling instance for $n=4$} \label{fig:GT}
  \end{figure}

The following result on g-tilings is of most importance for us.
 \begin{theorem} \label{tm:til-ws}
{\rm \cite{DKK-09}} The spectrum $\mathfrak{S}_T$ of any generalized tiling $T$
forms a largest ws-collection. Conversely, for any largest ws-collection
$\Cscr\subseteq 2^{[n]}$, there exists a generalized tiling $T$ on $Z_n$ such
that $\mathfrak{S}_T=\Cscr$. (Moreover, such a $T$ is unique and there is an
efficient procedure to construct $T$ from $\Cscr$.)
 \end{theorem}

In what follows, when it is not confusing, we may speak of a vertex or edge of
$G_T$ as a vertex or edge of $T$. The map $\sigma$ of the tiles in $T$ to
squares in $conv(2^{[n]})$ is extended, in a natural way, to the vertices,
edges, subgraphs or other objects in $G_T$. Note that the embedding of
$\sigma(G_T)$ in the disc $D_T$ is \emph{planar} (unlike $G_T$ and $Z$, in
general), i.e. any two edges of $\sigma(G_T)$ can intersect only at their end
points. It is convenient to assume that the clockwise orientations on $Z$ and
$D_T$ are agreeable, in the sense that the image by $\sigma$ of the boundary
cycle $(z_0,z^\ell_1,\ldots,z^\ell_n,z^r_1,\ldots,z^r_n=z_0)$ is oriented
clockwise around the interior of $D_T$. Then the orientations on a tile
$\tau\in T$ and on the square $\sigma(\tau)$ are consistent when $\tau$ is
white, and different when $\tau$ is black.

\subsection{Elementary properties of generalized tilings}  \label{ssec:Tprop}

The properties of g-tilings reviewed in this subsection can be obtained rather
easily from the above axioms; see~\cite{DKK-09} for more explanations. Let $T$
be a g-tiling on $Z=Z_n$.
\medskip

\noindent \textbf{1.} Let us say that the edges of $T$ occurring in black tiles
(as side edges) are {\em black}, and the other edges of $T$ are {\em white}.
For a vertex $v$ and two edges $e,e'\in E_T(v)$, let $\Theta(e,e')$ denote the
cone (with angle $<\pi$) in the plane pointed at $v$ and generated by these
edges (ignoring their directions). When another edge $e''\in E_T(v)$ (a tile
$\tau\in F_T(v)$) is contained in $\Theta(e,e')$, we say that $e''$ (resp.
$\tau$) lies \emph{between} $e$ and $e''$.  When these $e,e'$ are edges of a
tile $\tau$, we also write $\Theta(\tau;v)$ for $\Theta(e,e')$ (the conic hull
of $\tau$ at $v$), and denote by $\theta(\tau,v)$ the angle of this cone taken
with sign $+$ if $\tau$ is white, and sign $-$ if $\tau$ is black. The sum
$\sum(\theta(\tau,v)\colon \tau\in F_T(v))$ is denoted by $\rho(v)$ and called
the \emph{full angle} at $v$. Terminal vertices of $T$ behave as follows.
  \begin{corollary} \label{cor:termv}
Let $v$ be a terminal vertex belonging to a black ~$ij$-tile $\tau$. Then:

{\rm(i)} $v$ is not connected by edge with any other terminal vertex of $T$ (in
particular, $E_T(v)$ contains exactly two black edges, namely, those belonging
to $\tau$);

{\rm(ii)} $E_T(v)$ contains at least one white edge and all such edges $e$, as
well as all tiles in $F_T(v)$, lie between the two black edges in $E_T(v)$ (so
$e$ is a $q$-edge with $i<q<j$);

{\rm(iii)} $\rho(v)=0$;

{\rm(iv)} $v$ does not belong to the boundary of $Z$ (so each boundary edge $e$
of $Z$, as well as the tile containing $e$, is white).
  \end{corollary}

Note that~(ii) implies that
   \begin{numitem1}
if a black tile $\tau$ and a white tile $\tau'$ share an edge and if $v$ is
their common nonterminal vertex (which is either left or right in both
$\tau,\tau'$), then $\tau$ is contained in the cone $\Theta(\tau';v)$.
   \label{eq:nontermv}
   \end{numitem1}
Using this and applying Euler's formula to the planar graph $\sigma(G_T)$ on
$D_T$, one can specify the full angles at nonterminal vertices.
  \begin{corollary} \label{cor:rotation}
Let $v\in V_T$ be a nonterminal vertex.

{\rm(i)} If $v$ belongs to $bd(Z)$, then $\rho(v)$ is equal to the (positive)
angle between the boundary edges incident to $v$.

{\rm(ii)} If $v$ is inner (i.e. not in $bd(Z)$), then $\rho(v)=2\pi$.
  \end{corollary}

 \noindent \textbf{2.}
Using~\refeq{nontermv} and Corollary~\ref{cor:rotation}, one can obtain the
following useful (though rather lengthy) description of the local structure of
edges and tiles at nonterminal vertices.
  \begin{corollary} \label{cor:ord-mixed}
Let $v$ be a nonterminal (ordinary or mixed) vertex of $T$ different from
$z_0,z_n$. Let $e_1,\ldots,e_p$ be the sequence of edges leaving $v$ and
ordered clockwise around $v$ (or by increasing their labels), and
$e'_1,\ldots,e_{p'}$ the sequence of edges entering $v$ and ordered
counterclockwise around $v$ (or by decreasing their labels). Then there are
integers $r,r'\ge 0$ such that:
  \begin{itemize}
\item[\rm(i)] $r+r'<\min\{p,p'\}$, the edges $e_{r+1},\ldots,e_{p-r'}$ and
$e'_{r+1},\ldots,e'_{p'-r'}$ are white, the other edges in $E_T(v)$ are black,
$r=0$ if $v\in \ell bd(Z)$, and $r'=0$ if $v\in rbd(Z)$;
\item[\rm(ii)] for $q=r+1,\ldots,p-r'-1$, the edges $e_q,e_{q+1}$ are spanned by a
white tile (so such tiles have the bottom at $v$ and lie between $e_{r+1}$ and
$e_{p-r'}$);
\item[\rm(iii)] for $q=r+1,\ldots,p'-r'-1$, the edges $e'_q,e'_{q+1}$ are spanned by a
white tile $\tau$ (so such tiles have the top at $v$ and lie between $e'_{r+1}$
and $e'_{p'-r'}$);
\item[\rm(iv)] unless $v\in \ell bd(Z)$, each of the pairs $\{e_1,e'_{r+1}\},
\{e_2,e'_r\},\ldots,\{e_{r+1},e'_1\}$ is spanned by a white tile, and each of
the pairs $\{e_1,e'_{r}\},\{e_2,e'_{r-1}\},\ldots,\{e_{r},e'_1\}$ is spanned by
a black tile (all tiles have the right vertex at $v$);
\item[\rm(v)] unless $v\in rbd(Z)$, each of the pairs $\{e_p,e'_{p'-r'}\},
\{e_{p-1},e'_{p'-r'+1}\},\ldots,\{e_{p-r'},e'_{p'}\}$ is spanned by a white
tile, and each of the pairs $\{e_p,e'_{p'-r'+1}\}$, $\{e_{p-1},e'_{p'-r'+2}\},
\ldots,\{e_{p-r'+1},e'_{p'}\}$ is spanned by a black tile (all tiles have the
left vertex at $v$).
  \end{itemize}
In particular, (a) there is at least one white edge leaving $v$ and at
least one white edge entering $v$; (b) the tiles in (ii)--(v) give a
full list of tiles in $F_T(v)$; and (c) any two tiles $\tau,\tau'\in
F_T(v)$ with $r(\tau)=\ell(\tau')=v$ do not overlap (have no common
interior point).

Also: for $v=z_0,z_n$, all edges in $E_T(v)$ are white and the pairs of
consecutive edges are spanned by white tiles.
  \end{corollary}
(When $v$ is ordinary, we have $r=r'=0$.) The case with $p=4$, $p'=5$,
$r=2$, $r'=1$ is illustrated in the picture; here the black edges are
drawn in bold and the thin (bold) arcs indicate the pairs of edges
spanned by white (resp. black) tiles.

 \begin{center}
  \unitlength=1mm
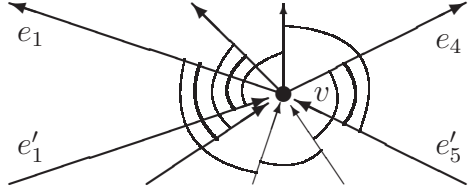
  \begin{picture}(80,25)
 \put(40,12){\circle*{2}}
  \put(36,0){\vector(1,3){3.6}}
  \put(48,0){\vector(-2,3){7}}
  \put(40,12){\vector(0,1){12}}
{\thicklines
  \put(4,0){\vector(3,1){34}}
  \put(22,0){\vector(3,2){16}}
  \put(64,0){\vector(-2,1){22.5}}
  \put(40,12){\vector(2,1){24}}
  \put(40,12){\vector(-1,1){12}}
  \put(40,12){\vector(-3,1){36}}
 }
  \put(44,11){$v$}
  \put(5,4){$e'_1$}
  \put(60,4){$e'_{5}$}
  \put(5,19){$e_1$}
  \put(60,18){$e_4$}
 \qbezier(40,17.5)(33,15)(35,10.5)
 \qbezier(33.5,18.5)(28,12)(33,7.5)
 \qbezier(27,16.5)(24,4)(36.5,1.5)
 \qbezier(37,3)(42,1.5)(45,4.5)
 \qbezier(44,6)(49,10)(46,15)
 \qbezier(50,7)(55,20)(40,21)
{\thicklines
 \qbezier(35,17)(31,13)(33.5,10)
 \qbezier(29,16)(27,10)(31.25,6.0)
 \qbezier(48,8)(51,12)(48,16)
  }
   \end{picture}
   \end{center}

\noindent \textbf{3.} In view of~\refeq{xi}(ii), the graph $G_T=(V_T,E_T)$ is
{\em graded} for each label $i\in[n]$, which means that for any closed path $P$
in $G_T$, the amounts of forward $i$-edges and backward $i$-edges in $P$ are
equal. In particular, this easily implies that
  \begin{numitem1}
if four vertices and four edges of $G_T$ form a (non-directed) cycle, then they
are the vertices and edges of a tile (not necessarily contained in $T$).
  \label{eq:4edges}
  \end{numitem1}

Hereinafter, a path in a directed graph is meant to be a sequence $P=(\tilde
v_0,\tilde e_1,\tilde v_1,\ldots,\tilde e_r,\tilde v_r)$ in which each $\tilde
e_p$ is an edge connecting vertices $\tilde v_{p-1},\tilde v_p$; an edge
$\tilde e_p$ is called {\em forward} if it is directed from $\tilde v_{p-1}$ to
$\tilde v_p$ (denoted as $\tilde e_p=(\tilde v_{p-1},\tilde v_p)$), and {\em
backward} otherwise (when $\tilde e_p=(\tilde v_p,\tilde v_{p-1})$). When
$v_0=v_r$ and $r>0$, ~$P$ is a \emph{closed path}, or a \emph{cycle}. The path
$P$ is called {\em directed} if all its edges are forward, and {\em simple} if
all vertices $v_0,\ldots,v_r$ are different. $P^{rev}$ denotes the reversed
path $(\tilde v_r,\tilde e_r,\tilde v_{r-1},\ldots,\tilde e_1,\tilde v_0)$.
Sometimes we will denote a path without explicitly indicating its edges:
$P=(\tilde v_0,\tilde v_1,\ldots,\tilde v_r)$. A directed graph is called
\emph{acyclic} if it has no directed cycles.

\subsection{Strips, contractions, expansions, and others} \label{ssec:s-c-e}

In this subsection we describe additional, more involved, results and
constructions concerning g-tilings that are given in~\cite{DKK-09} and
will be needed to us to prove Theorem~B. They are exposed in
Propositions~\ref{pr:strips}--\ref{pr:4edges} below. \medskip

\noindent \textbf{A.} \emph{Strips in $T$.} \smallskip

\noindent \textbf{Definition 3} ~Let $i\in[n]$. An $i$-{\em strip} (or a {\em
dual $i$-path}) in $T$ is a maximal alternating sequence
$Q=(e_0,\tau_1,e_1,\ldots,\tau_r,e_r)$ of edges and tiles of $T$ such that: (a)
$\tau_1,\ldots,\tau_r$ are different tiles, each being an $ij$- or $ji$-tile
for some $j$, and (b) for $p=1,\ldots,r$, ~$e_{p-1}$ and $e_p$ are the opposite
$i$-edges of $\tau_p$. \medskip

(Recall that speaking of an $i'j'$-tile, we assume that $i'<j'$.) In view of
axiom~(T1), $Q$ is determined uniquely (up to reversing it, and up to shifting
cyclically when $e_0=e_r$) by any of its edges or tiles. For $p=1,\ldots,r$,
let $e_p=(v_p,v'_p)$. Define the {\em right boundary} of $Q$ to be the (not
necessarily directed) path $R_Q=(v_0,a_1,v_1,\ldots,a_r,v_r)$, where $a_p$ is
the edge of $\tau_p$ connecting $v_{p-1}$ and $v_p$. Similarly, the {\em left
boundary} of $Q$ is the path $L_Q=(v'_0,a'_1,v'_1,\ldots,a'_r,v'_r)$, where
$a'_p$ is the edge of $\tau_p$ connecting $v'_{p-1}$ and $v'_p$. Then two edges
$a_p,a'_p$ have the same label. Considering $R_Q$ and using the fact that $G_T$
is graded, one shows that
  \begin{numitem1}
$Q$ cannot be cyclic, i.e. the edges $e_0$ and $e_r$ are different.
  \label{eq:noncycl}
  \end{numitem1}
In view of the maximality of $Q$, ~\refeq{noncycl} implies that one of
$e_0,e_r$ belongs to the left boundary, and the other to the right boundary of
the zonogon $Z$; we assume that $Q$ is directed so that $e_0\in \ell bd(Z)$
(justifying the assignment of the right and left boundaries for $Q$).
Properties of strips are exposed in the following

  \begin{prop} \label{pr:strips}
~For each $i\in[n]$, there is exactly one $i$-strip, $Q_i$ say. It contains all
$i$-edges of $T$, begins with the edge $z^\ell_{i-1}z^\ell_i$ of $\ell bd(Z)$
and ends with the edge $z^r_{i}z^r_{i-1}$ of $rbd(Z)$. Furthermore, each of
$R_{Q_i}$ and $L_{Q_i}$ is a simple path, $L_{Q_i}$ is disjoint from $R_{Q_i}$
and is obtained by shifting $R_{Q_i}$ by the vector $\xi_i$. An edge of
$R_{Q_i}$ is forward if and only if it belongs to either a white $i\ast$-tile
or a black $\ast i$-tile in $Q_i$, and similarly for the edges of $L_{Q_i}$.
  \end{prop}

\noindent \textbf{B.} \emph{Strip contractions.} \smallskip

Let $i\in[n]$.  Partition $T$ into three subsets $T^0_i, T^-_i,T^+_i$, where
$T^0_i$ consists of all $i\ast$- and $\ast i$-tiles, $T^-_i$ consists of the
tiles $\tau(X;i',j')$ with $i',j'\ne i$ and $i\not\in X$, and $T^+_i$ consists
of the tiles $\tau(X;i',j')$ with $i',j'\ne i$ and $i\in X$. Then $T_i^0$ is
the set of tiles occurring in the $i$-strip $Q_i$, and the tiles in $T_i^-$ are
vertex disjoint from those in $T_i^+$. \medskip

 \noindent \textbf{Definition 4}
~The $i$-{\em contraction} of $T$ is the collection $T/i$ of tiles
obtained by removing $T_i^0$, keeping the members of $T_i^-$, and
replacing each $\tau(X;i',j')\in T_i^+$ by $\tau(X-i;i',j')$. The
black/white coloring of tiles in $T/i$ is inherited from $T$.
\medskip

So the tiles of $T/i$ live within the zonogon generated by the vectors
$\xi_q$ for $q\in[n]-i$. Clearly if we remove from the disc $D_T$ the
interiors of the edges and squares in $\sigma(Q_i)$, then we obtain two
closed simply connected regions, one containing the squares
$\sigma(\tau)$ for all $\tau\in T_i^-$, denoted as $D_{T_i^-}$, and the
other containing $\sigma(\tau)$ for all $\tau\in T_i^+$, denoted as
$D_{T_i^+}$. Then $D_{T/i}$ is the union of $D_{T_i^-}$ and
$D_{T_i^+}-\eps_i$, where $\eps_i$ is $i$-th unit base vector in
$\Rset^{[n]}$. In other words, $D_{T_i^+}$ is shifted by $-\eps_i$ and
the path $\sigma(L_{Q_i})$ in it (the left boundary of $\sigma(Q_i)$)
merges with the path $\sigma(R_{Q_i})$ in $D_{T_i^-}$. In general,
$D_{T_i^-}$ and $D_{T_i^+}-\eps_i$ may intersect at some other points,
and for this reason, $D_{T/i}$ need not be a disc. Nevertheless,
$D_{T/i}$ is shown to be a disc in two important special cases: $i=n$
and $i=1$; moreover, the following property takes place.
  \begin{prop}  \label{pr:contract}
~The $n$-contraction $T/n$ of $T$ is a (feasible) g-tiling on the zonogon
$Z_{n-1}$ generated by the vectors $\xi_1,\ldots,\xi_{n-1}$. Similarly, the
$1$-contraction $T/1$ is a g-tiling on the $(n-1)$-zonogon generated by the
vectors $\xi_2,\ldots,\xi_n$.
  \end{prop}
(If wished, labels $2,\ldots,n$ for $T/1$ can be renamed as
$1',\ldots,(n-1)'$.) We will use the $n$- and 1-contraction operations in
Sections~\ref{sec:proof} and~\ref{sec:2posets}.\medskip

\noindent \textbf{C.} \emph{Legal paths and strip expansions.} \smallskip

Next we describe the $n$-expansion and 1-expansion operations; they are
converse, in a sense, to the $n$-contraction and 1-contraction ones,
respectively. We start with introducing the operation for $n$.

The $n$-{\em expansion} operation applies to a g-tiling $T$ on the
zonogon $Z=Z_{n-1}$ generated by $\xi_1,\ldots,\xi_{n-1}$ and to a
simple (not necessarily directed) path $P$ in the graph $G_{T}$
beginning at the minimal vertex $z_0$ and ending at the maximal vertex
$z^\ell_{n-1}$ of $Z$. Then $\sigma(P)$ subdivides the disc $D_T$ into
two simply connected closed regions $D',D''$ such that: $D'\cup
D''=D_T$, ~$D'\cap D''=\sigma(P)$, ~$D'$ contains $\sigma(\ell bd(Z))$,
and $D''$ contains $\sigma(rbd(Z))$. Let $T':=\{\tau\in T\colon
\sigma(\tau)\subset D'\}$ and $T'':=T-T'$. We disconnect $D',D''$ along
$\sigma(P)$ by shifting $D''$ by the vector $\eps_n$, and then connect
them by adding the corresponding strip of $\ast n$-tiles.

More precisely, we construct a collection $\tilde T$ of tiles on the zonogon
$Z_n$, called the \emph{$n$-expansion of $T$ along $P$}, as follows. The tiles
of $T'$ are kept and each tile $\tau(X;i,j)\in T''$ is replaced by
$\tau(Xn;i,j)$; the white/black coloring on these tiles is inherited. For each
edge $e=(X,Xi)$ of $P$, we add tile $\tau(X;i,n)$, making it white if $e$ is
forward, and black if $e$ is backward in $P$. The resulting $\tilde T$ need not
be a g-tiling in general; for this reason, we impose additional conditions on
$P$.
\medskip

\noindent \textbf{Definition 5} ~ $P$ as above is called an
$n$-{\em legal path} if it satisfies the following three
conditions:

(i) all vertices of $P$ are nonterminal;

(ii) $P$ contains no pair of consecutive backward edges;

(iii) for an $i$-edge $e$ and a $j$-edge $e'$ such that $e,e'$ are consecutive
edges occurring in this order in $P$: if $e$ is forward and $e'$ is backward in
$P$, then $i>j$, and if $e$ is backward and $e'$ is forward, then $i<j$.
\medskip

In view of~(ii), $P$ is represented as the concatenation of
$P_1,\ldots,P_{n-1}$, where $P_h$ is the maximal subpath of $P$ whose edges
connect levels $h-1$ and $h$ (i.e. are of the form $(X,Xi)$ with $|X|=h-1$). In
view of~(iii), each path $P_h$ has planar embedding in $Z$; it either contains
only one edge, or is viewed as a \emph{zigzag} path going \emph{from left to
right}. The first and last vertices of these subpaths are called {\em critical}
vertices of $P$. The importance of legal paths is seen from the following
  \begin{prop} \label{pr:legal}
~The $n$-expansion of $T$ along $P$ is a (feasible) g-tiling on the zonogon
$Z_n$ if and only if $P$ is an $n$-legal path.
  \end{prop}

Under the $n$-expansion operation, the path $P$ generates the $n$-strip
$Q_n$ of the resulting g-tiling $\tilde T$; more precisely, the right
boundary of $Q_n$ is the reversed path $P^{rev}$ to $P$, and the left
boundary of $Q_n$ is obtained by shifting $P^{rev}$ by $\xi_n$. A
possible fragment of $P$ consisting of three consecutive edges
$e,e',e''$ forming a zigzag and the corresponding fragment in $Q_n$
(with two white tiles created from $e,e''$ and one black tile created
from $e'$) are illustrated in the picture; here the shifted $e,e',e''$
are indicated with tildes.

 \begin{center}
  \unitlength=1mm
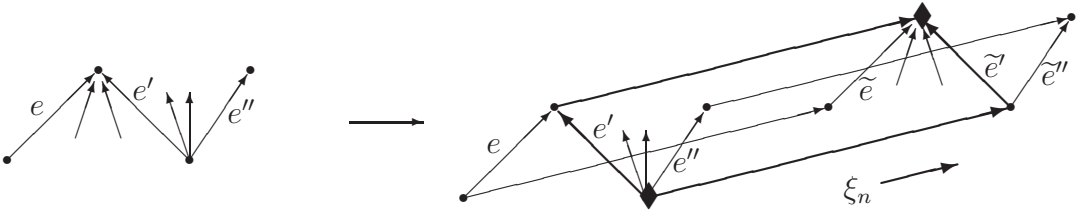
  \begin{picture}(145,25)
   \put(5,0){\begin{picture}(60,20)
  \put(0,5){\circle*{1}}
  \put(12,17){\circle*{1}}
  \put(24,5){\circle*{1}}
  \put(32,17){\circle*{1}}
  \put(0,5){\vector(1,1){11.5}}
  \put(24,5){\vector(-1,1){11.5}}
  \put(24,5){\vector(2,3){7.6}}
  \put(9,8){\vector(1,3){2.5}}
  \put(15,8){\vector(-1,3){2.5}}
  \put(24,5){\vector(-1,3){3}}
  \put(24,5){\vector(0,1){9}}
  \put(3,11){$e$}
  \put(16.5,13){$e'$}
  \put(29,10){$e''$}
  \put(45,10){\vector(1,0){9.7}}
    \end{picture}}
%
   \put(65,0){\begin{picture}(80,25)
  \put(0,0){\circle*{1}}
  \put(12,12){\circle*{1}}
  \put(24,0){\circle*{1}}
  \put(32,12){\circle*{1}}
  \put(48,12){\circle*{1}}
  \put(60,24){\circle*{1}}
  \put(72,12){\circle*{1}}
  \put(80,24){\circle*{1}}
  \put(0,0){\vector(1,1){11.5}}
  \put(24,0){\vector(2,3){7.6}}
  \put(48,12){\vector(1,1){11}}
  \put(72,12){\vector(2,3){7.6}}
  \put(0,0){\vector(4,1){47.5}}
  \put(32,12){\vector(4,1){47.5}}
  \put(57,15){\vector(1,3){2.5}}
  \put(63,15){\vector(-1,3){2.5}}
  \put(24,0){\vector(-1,3){3}}
  \put(24,0){\vector(0,1){9}}

\thicklines{
  \put(24,0){\vector(-1,1){11.5}}
  \put(72,12){\vector(-1,1){11.5}}
  \put(12,12){\vector(4,1){47.5}}
  \put(24,0){\vector(4,1){47.5}}
 }
  \put(23,-1){$\blacklozenge$}
  \put(59,23){$\blacklozenge$}
  \put(3,6){$e$}
  \put(17,8){$e'$}
  \put(27.5,3.5){$e''$}
  \put(52,13){$\tilde e$}
  \put(68.5,16){$\tilde e'$}
  \put(76,15){$\tilde e''$}
  \put(50,0){$\xi_n$}
  \put(55,2){\vector(4,1){10}}
    \end{picture}}
  \end{picture}
   \end{center}

The $n$-contraction operation applied to $\tilde T$ returns the initial $T$. A
relationship between $n$-contractions and $n$-expansions is described in the
following
  \begin{prop}  \label{pr:TP-T}
~The correspondence $(T,P)\mapsto \tilde T$, where $T$ is a g-tiling on
$Z_{n-1}$, ~$P$ is an $n$-legal path for $T$, and $\tilde T$ is the
$n$-expansion of $T$ along $P$, gives a bijection between the set of
such pairs $(T,P)$ and the set of g-tilings on $Z_n$.
  \end{prop}

In its turn, the 1-{\em expansion operation} applies to a g-tiling $T$ on the
zonogon $Z$ generated by the vectors $\xi_2,\ldots,\xi_n$ (so we deal with
labels $2,\ldots,n$) and to a simple path $P$ in $G_T$ from the minimal vertex
to the maximal vertex of $Z$; it produces a g-tiling $\tilde T$ on $Z_n$. This
is equivalent to applying the $n$-expansion operation in the mirror-reflected
situation: when label $i$ is renamed as label $n-i+1$ (and accordingly a tile
$\tau(X;i,j)$ in $T$ is replaced by the tile $\tau(\{k\colon n-k+1\in
X\};n-j+1,n-i+1)$, preserving the basic vectors $\xi_1,\ldots,\xi_n$). The
corresponding ``1-analogues'' of the above results on $n$-expansions are as
follows.
  \begin{prop} \label{pr:1expan}
~{\rm (i)} The 1-expansion $\tilde T$ of $T$ along $P$ is a g-tiling on
$Z_n$ if and only if $P$ is a \emph{1-legal path}, which is defined as
in Definition~5 with the only difference that each subpath $P_h$ of $P$
(formed by the edges connecting levels $h-1$ and $h$) either contains
only one edge, or is a zigzag path going \emph{from right to left}.

{\rm (ii)} The 1-contraction operation applied to $\tilde T$ returns the
initial $T$.

{\rm(iii)} The correspondence $(T,P)\mapsto \tilde T$, where $T$ is a g-tiling
on the zonogon generated by $\xi_2,\ldots,\xi_n$, ~$P$ is a 1-legal path for
$T$, and $\tilde T$ is the 1-expansion of $T$ along $P$, gives a bijection
between the set of such pairs $(T,P)$ and the set of g-tilings on $Z_n$.
  \end{prop}

\noindent \textbf{D.} \emph{Principal trees.} \smallskip

Let $T$ be a g-tiling on $Z=Z_n$. We distinguish between two sorts of
white edges $e$ of $G_T$ by saying that $e$ is \emph{fully white} if
both of its end vertices are nonterminal, and \emph{semi-white} if one
end vertex is terminal. (Recall that an edge $e$ of $G_T$ is called
white if no black tile contains $e$ (as a side edge); the case when
both ends of $e$ are terminal is impossible, cf.
Corollary~\ref{cor:termv}(i).) In particular, all boundary edges of $Z$
are fully white.

The following result on structural features of the set of white edges is
obtained by using Corollary~\ref{cor:ord-mixed}.
   \begin{prop} \label{pr:Hh}
~For $h=1,\ldots,n$, let $H_h$ denote the subgraph of $G_T$ induced by the set
of white edges connecting levels $h-1$ and $h$ (i.e. of the form $(X,Xi)$ with
$|X|=h-1$). Then $H_h$ is a forest. Furthermore:

{\rm (i)} there exists a component (a maximal tree) $K_h$ of $H_h$ that
contains all fully white edges of $H_p$ (in particular, the boundary edges
$z^\ell_{h-1}z^\ell_h$ and $z^r_{n-h+1}z^r_{n-h}$) and no other edges;
moreover, $K_h$ has planar embedding in $Z$;

{\rm (ii)} any other component $K'$ of $H_h$ contains exactly one terminal
vertex $v$; this $K'$ is a star at $v$ whose edges are the (semi-)white edges
incident to $v$.
   \end{prop}

It follows that the subgraph $G^{fw}=G_T^{fw}$ of $G_T$ induced by the
fully white edges has planar embedding in $Z$. We refer to $K_h$ in~(i)
of the proposition as the \emph{principal tree} in $H_h$. The common
vertices of two neighboring principal trees $K_h,K_{h+1}$ will play an
important role later; we call them \emph{critical vertices} for $T$ in
level $h$. \medskip

\noindent \textbf{E.} Two more useful facts (which look so natural but their
proofs are not straightforward) concern relations between vertices and edges in
$G_T$ and tiles in $T$.
  \begin{prop} \label{pr:4edges}
~{\rm(i)} Any two nonterminal vertices of the form $X,Xi$ in $G_T$ are
connected by edge. {\em(Such an edge need not exist when some of $X,Xi$ is
terminal.)}

{\rm(ii)} If four \emph{nonterminal} vertices are connected by four edges
forming a cycle in $G_T$, then there is a tile in $T$ having these vertices and
edges. \emph{(Cf.~\refeq{4edges}.)}
  \end{prop}

\section{Proof of Theorem~B}  \label{sec:proof}

In this section we explain how to obtain Theorem~B in the assumption of
validity of the following statement, which is interesting by its own right.
Recall that for sets $A,B\subseteq [n]$, we write $A\prec^\ast B$ if $A\lessdot
B$ and $|A|\le|B|$.
  \begin{theorem} \label{tm:prec-lat}
Let $\Cscr$ be a largest ws-collection. Then the partial order on $\Cscr$
determined by $\prec^\ast$ forms a lattice.
  \end{theorem}
Here the fact that $(\Cscr,\prec^\ast)$ is a poset (partially ordered set) is
an immediate consequence of the following simple, but important, property
established in~\cite{LZ}, which describes a situation when the relation
$\lessdot$ becomes transitive:
  \begin{numitem1}
for sets $A,A',A''\subseteq[n]$, if $A\lessdot A'\lessdot A''$, ~$A\wsep A''$
and $|A|\le|A'|\le |A''|$, then $A\lessdot A''$.
  \label{eq:AAA}
  \end{numitem1}

Theorem~\ref{tm:prec-lat} will be proved in Sections~\ref{sec:aux_graph}
and~\ref{sec:2posets}. Relying on this, the method of proving that any
ws-collection $\Cscr\subseteq 2^{[n]}$ is contained in a largest ws-collection
in $2^{[n]}$ (yielding Theorem~B) is roughly as follows. We first reduce
$\Cscr$ to a ws-collection of subsets of $[n-1]$ and then extend the latter to
a maximal ws-collection $\Cscr'$ in $2^{[n-1]}$. One may assume by induction on
$n$ that $\Cscr'$ is a largest ws-collection; so, by Theorem~~\ref{tm:til-ws},
$\Cscr'$ is the spectrum of some g-tiling $T'$ on the zonogon $Z_{n-1}$.
Theorem~\ref{tm:prec-lat} applied to $\Cscr'$ is then used to show the
existence of an appropriate $n$-legal path $P$ for $T'$. Applying the
$n$-expansion operation to $(T',P)$ (as described in
Subsection~\ref{ssec:s-c-e}, part~C), we obtain a g-tiling $T$ on $Z_n$ whose
spectrum contains $\Cscr$, whence the result follows.

We need two lemmas (where we write $X\lessdoteq Y$ if either $X\lessdot Y$ or
$X=Y$).

 \begin{lemma} \label{lm:ABC}
{\rm(i)} Let $A\lessdoteq C$ and $B\lessdoteq C$. Then either $C\subset A\cup
B$ or $A\cup B\lessdoteq C$.

{\rm(ii)} Let $C\lessdoteq A$ and $C\lessdoteq B$. Then $C\lessdoteq A\cup B$.
  \end{lemma}
  \begin{proof}
~(i) If $C\subset A\cup B$ or $A\cup B\subseteq C$, we are done. So assume that
both $C-(A\cup B)$ and $(A\cup B)-C$ are nonempty and consider an element $c$
in the former and an element $x$ in the latter of these sets. One may assume
that $x\in A$. Since $x\not\in C$, ~$c\not\in A$, and $A\lessdoteq C$, we have
$x<c$. This implies $A\cup B\lessdoteq C$.

(ii) If $C\subseteq A\cup B$, we are done. So assume this is not the case, and
let $c\in  C-(A\cup B)$. Let $x$ be an element of the set $(A\cup B)-C$ (which
is, obviously, nonempty). One may assume that $x\in A-C$. Then $c\in C-A$ and
$C\lessdoteq A$ imply $c<x$, as required.
  \end{proof}

 \noindent \textbf{Definition 6}
~Let $\Lscr,\Rscr\subseteq 2^{[n']}$. We call $(\Lscr,\Rscr)$ a
\emph{left-right pair}, or, briefly, an \emph{lr-pair}, if $\Lscr\cup\Rscr$ is
a ws-collection and
  \begin{description}
\item[{\rm(LR):}] ~$L\lessdoteq R$ holds for any $L\in \Lscr$ and
$R\in\Rscr$ with $|L|\le |R|$.
  \end{description}

  \begin{lemma} \label{lm:ijkY}
~Let $(\Lscr,\Rscr)$ be an lr-pair in $2^{[n']}$.

{\rm(i)} ~Suppose that there are $X\subseteq [n']$ and $i,j,k\in X$
such that: $i<j<k$, the sets $X-k,X-j$ belong to $\Lscr$, and the sets
$X-j,X-i$ belong to $\Rscr$. Then $(\Lscr\cup\{X\},\Rscr)$ is an
lr-pair as well.

{\rm(ii)} ~Symmetrically, suppose that there are $X\subseteq [n']$ and
$i,j,k\not\in X$ such that: $i<j<k$, the sets $Xi,Xj$ belong to
$\Lscr$, and the sets $Xj,Xk$ belong to $\Rscr$. Then
$(\Lscr,\Rscr\cup\{X\})$ is an lr-pair as well.
  \end{lemma}
  \begin{proof}
~To see~(i), we first show that $X$ is weakly separated from any member $Y$ of
$\Lscr\cup\Rscr$.

Suppose that $|X|=|Y|+1$. If $Y\in \Lscr$, we argue as follows. Since
$X-i,X-j\in\Rscr$ and $|X-i|=|X-j|=|Y|$, we have $Y\lessdoteq X-i$ and
$Y\lessdoteq X-j$. By~(ii) in Lemma~\ref{lm:ABC}, we obtain that
$Y\lessdoteq (X-i)\cup (X-j)=X$. In case $Y\in \Rscr$, we have
$X-j,X-k\lessdoteq Y$, and now (i) in Lemma~\ref{lm:ABC} implies that
either $Y\lessdot X$ or $X\lessdoteq Y$.

Now suppose that $X$ is not weakly separated from some $Y\in\Lscr\cup\Rscr$
with $|X|\ne |Y|+1$. Three cases are possible.

1) Let $|X|<|Y|$. Then one easily shows that there are $a,c\in Y-X$ and
$b\in X-Y$ such that $a<b<c$; cf.~Lemma~3.8 in~\cite{LZ}. The element
$b$ belongs to some set $X'$ among $X-i,X-j,X-k$. Then $b\in X'-Y$ and
$a,c\in Y-X'$, implying $X'\rhd Y$ (since $X'\wsep Y$). But $|X'|<|Y|$;
a contradiction.

2) Let $|X|=|Y|$. Then there are $a,c\in Y-X$ and $b,d\in X-Y$ such that either
$a<b<c<d$ or $a>b>c>d$, by the same lemma in~\cite{LZ}. But both $b,d$ belong
to at least one set $X'$ among $X-i,X-j,X-k$. So $X',Y$ are not weakly
separated; a contradiction.

3) Let $|X|>|Y|+1$. Then $a<b<c$ for some $a,c\in X-Y$ and $b\in Y-X$. Both
$a,c$ belong to some set $X'$ among $X-i,X-j,X-k$. Then $Y\rhd X'$. But
$|X'|=|X|-1>|Y|$; a contradiction.

Thus, $\Lscr\cup\Rscr\cup\{X\}$ is a ws-collection. It remains to check that
$X\lessdoteq R$ for any $R\in\Rscr$ with $|R|\ge |X|$ (then
$(\Lscr\cup\{X\},\Rscr)$ is an lr-pair). Since $X-j\in\Lscr$ and $|X-j|<|R|$,
we have $X-j\lessdot R$. Similarly, $X-k\lessdot R$. So, by
Lemma~\ref{lm:ABC}(i), $X\lessdoteq R$, as required. (The case $R\subset X$ is
impossible since $|R|\ge |X|$.) This yields~(i).

Validity of~(ii) follows from~(i) applied to the complementary lr-pair
$(\{[n']-R\colon R\in\Rscr\},\{[n']-L\colon L\in\Lscr\})$.
  \end{proof}

Now we start proving Theorem~B. Let $\Cscr\subseteq 2^{[n]}$ be a
ws-collection. The goal is to show that $\Cscr$ is contained in a
largest ws-collection in $2^{[n]}$. We use induction on $n$.

Form the collections $\Lscr:=\{X\subseteq[n-1]\colon X\in \Cscr\}$ and $\Rscr:=
\{X\subseteq[n-1]\colon Xn\in\Cscr\}$. By easy observations
in~\cite[Section~3]{LZ}, ~$\Lscr\cup\Rscr$ is a ws-collection and, furthermore,
$(\Lscr,\Rscr)$ is an lr-pair. Let us extend $(\Lscr,\Rscr)$ to a maximal lr-pair
$(\bar\Lscr,\bar\Rscr)$ in $2^{[n-1]}$, i.e. $\Lscr\subseteq\bar\Lscr$,
$\Rscr\subseteq\bar\Rscr$, and neither $\bar\Lscr$ nor $\bar\Rscr$ can be further
extended. In particular, $\bar\Lscr$ contains the intervals $[i]$ and $\bar\Rscr$
does the intervals $[i..n-1]$ for each $i$ (including the empty interval).

By induction, there exists a largest ws-collection $\Cscr'\subseteq 2^{[n-1]}$
containing $\bar\Lscr\cup\bar\Rscr$. By Theorem~\ref{tm:til-ws}, $\Cscr'$ is
the spectrum $\mathfrak{S}_{T'}$ of some g-tiling $T'$ on the zonogon
$Z_{n-1}$. By Theorem~\ref{tm:prec-lat}, the poset $\Pscr$ determined by
$\Cscr'$ and $\prec^\ast$ is a lattice.

For $h=0,\ldots,n-1$, let $\Cscr'_h,\bar\Lscr_h,\bar\Rscr_h$ consist of the
sets $X$ with $|X|=h$ in $\Cscr',\bar\Lscr,\bar\Rscr$, respectively. Let
$C_h\in\Cscr'$ be a maximal element in the poset $\Pscr$ provided that
  \begin{numitem1}
~$C_h\preceq^\ast R$ ~for all $R\in\bar\Rscr_h\cup\ldots\cup\bar\Rscr_{n-1}$.
  \label{eq:ChR}
  \end{numitem1}
Since $\Pscr$ is a lattice, $C_h$ exists and is unique, and we have:
  \begin{numitem1}
~{\rm(i)} $L\preceq^\ast C_h$ ~for all $L\in\bar\Lscr_0\cup\ldots\cup
\bar\Lscr_h$;

{\rm(ii)} $C_0\preceq^\ast C_1\preceq^\ast \ldots\preceq^\ast C_n$,
  \label{eq:Cchain}
  \end{numitem1}
where (i) follows from condition~(LR) in the definition of lr-pairs. Note that
for each $h$, both sets $\bar\Lscr_h$ and $\bar\Rscr_h$ are nonempty, as the
former contains the interval $[h]$ (viz. the vertex $z^\ell_h$ of $\Gamma'$)
and the latter contains $[n-h..n-1]$ (viz. the vertex $z^r_{n-h}$). Also for
any $L\in\bar\Lscr_h$ and $R\in\bar\Rscr_h$, the facts that $L\prec^\ast
C_h\prec^\ast R$ (cf. \refeq{ChR},\refeq{Cchain}) and $|L|=|R|=h$ imply that
$|C_h|=h$, whence $C_h\in\Cscr'_h$.

If $X\in\Cscr'_h$ and $X\prec^\ast C_h$ (resp. $C_h\prec^\ast X$), then $X$
must belong to $\bar\Lscr$ (resp. $\bar\Rscr$); this follows from the
maximality of $(\bar\Lscr,\bar\Rscr)$ and relations~\refeq{ChR}
and~\refeq{Cchain}(i) (in view of the transitivity of $\prec^\ast$). Also the
maximality of $(\bar\Lscr,\bar\Rscr)$ implies that $C_h$ belongs to both
$\bar\Lscr_h$ and $\bar\Rscr_h$. We assert that $C_h$ is a critical vertex (in
level $h$) of the graph $G_{T'}$, for each $h$; see the definition in
Subsection~\ref{ssec:s-c-e}, part~D.

To see this, take a white edge $e$ leaving the vertex $C_h$ in $G_{T'}$ (unless
$h=n-1$); it exists by Corollary~\ref{cor:ord-mixed}. Suppose that the end
vertex $X$ of $e$ is terminal. Then $X$ is the top vertex $t(\tau)$ of some
black tile $\tau\in T'$; in particular, $X$ is not in
$\mathfrak{S}_{T'}=\Cscr'$. Since $e$ is white, there are white tiles
$\tau',\tau''\in F_{\tau}(X)$ with $r(\tau')=\ell(\tau'')=C_h$. Then: (a) the
vertices $X':=\ell(\tau')$, $C_h$ and $X'':=r(\tau'')$ are of the form
$X-k,X-j,X-i$, respectively, for some $i<j<k$; (b) $X'$ belongs to
$\bar\Lscr_h$ (since $X'$ is nonterminal and, obviously, $X'\prec^\ast C_h$);
and (c) $X''$ belongs to $\bar\Rscr_h$ (since $C_h\prec^\ast X''$). But then,
by Lemma~\ref{lm:ijkY}(i), $\bar\Lscr$ can be increased by adding the new
element $X$, contrary to the maximality of $(\bar\Lscr,\bar\Rscr)$. So the edge
$e$ is fully white. In a similar fashion (using (ii) in Lemma~\ref{lm:ijkY}),
one shows that $C_h$ has an entering fully white edge (unless $h=0$). Thus,
$C_h$ is critical, as required.

Finally, by Proposition~\ref{pr:Hh}, each pair of critical vertices $C_{h-1},C_h$
is connected by a path $P_h$ in the principal tree $K_h$. We assert that $P_h$
either has only one edge or is a zigzag path going from left to right. Indeed, if
this is not so, then $P_h$ is a zigzag path going from right to left, i.e. $P_h$
is of the form $(C_{h-1}=X_1,Y_1,\ldots,X_k,Y_k=C_h)$ with $k\ge 2$, the vertices
$X_p$ (resp. $Y_p$) are in level $h-1$ (resp. $h$), and for labels $i_p$ of
(forward) edges $(X_p,Y_p)$ and labels $j_p$ of (backward) edges $(Y_p,X_{p+1})$,
one holds $\ldots >i_p<j_p>i_{p+1}<\ldots$. Then $X_k\prec^\ast
X_{k-1}\prec^\ast\ldots\prec^\ast X_1$ (in view of $X_p-X_{p+1}=\{j_{p}\}$ and
$X_{p+1}-X_p=\{i_p\}$). This and $C_{h-1}\prec^\ast C_h$ imply $X_{k-1}\prec^\ast
Y_k$, which is impossible since $X_{k-1}-Y_k=\{j_{k-1}\}$,
~$Y_k-X_{k-1}=\{i_{k-1},i_k\}$ and $i_{k-1},i_k<j_{k-1}$.

It follows that the concatenation of the paths $P_1,\ldots,P_{n-1}$ gives an
$n$-legal path $P$ in $G_{T'}$. By Proposition~\ref{pr:legal}, the
$n$-expansion of $T'$ along $P$ is a feasible g-tiling $T$ on the zonogon
$Z_n$, and now it is straightforward to check that the initial collection
$\Cscr$ is contained in the spectrum $\mathfrak{S}_T$ of $T$ (which is a
largest ws-collection, by Theorem~\ref{tm:til-ws}).

This completes the proof of Theorem~B (provided validity of
Theorem~\ref{tm:prec-lat}).

\section{The auxiliary graph}  \label{sec:aux_graph}

The proof of Theorem~\ref{tm:prec-lat} is divided into two stages. In order to
prove the theorem, we are forced to include into consideration an additional
combinatorial object. This is a certain acyclic directed graph $\Gamma_T$
associated to a g-tiling $T$ on the zonogon $Z_n$; it is different from the
graph $G_T$ and its vertex set is the spectrum $\mathfrak{S}_T$ of $T$. In this
section we define $\Gamma_T$, called the \emph{auxiliary graph} for $T$, and
show that the natural partial order on $\mathfrak{S}_T$ determined by
$\Gamma_T$ is a lattice; this is the first, easier, stage of the proof of
Theorem~\ref{tm:prec-lat}. The second, crucial, stage will be given in the next
section; it consists in proving that the partial order of $\Gamma_T$ coincides
with $\prec^\ast$ (within $\mathfrak{S}_T$). Then Theorem~\ref{tm:prec-lat}
will follow from the fact that the largest ws-collections in $2^{[n]}$ are
exactly the spectra of g-tilings on $Z_n$, i.e. from Theorem~\ref{tm:til-ws}.
\medskip

\noindent \textbf{Construction of $\Gamma=\Gamma_T$:}~ Given a g-tiling $T$ on
$Z_n$, the vertex set of $\Gamma$ is $\mathfrak{S}_T$. The edge set of $\Gamma$
consists of two subsets: the set $E^{asc}$ of fully white edges of $G_T$ (the
edge set of $G^{fw}$), and the set $E^{hor}$ of edges corresponding to the
``horizontal diagonals'' of white tiles, namely, for each $\tau\in T^w$, we
form edge $e_\tau$ going from $\ell(\tau)$ to $r(\tau)$. An edge in $E^{asc}$
is called \emph{ascending} (it goes from some level $h$ to the next level
$h+1$, having the form $(X,Xi)$ with $|X|=h$). An edge in $E^{hor}$ is called
\emph{horizontal} (as it connects vertices of the same height).
\medskip

In particular, $\Gamma$ contains $bd(Z_n)$ (since all boundary edges are fully
white). Figure~\ref{fig:GT-Gamma} compares the graph $G_T$ drawn in
Fig.~\ref{fig:GT} and the graph $\Gamma_T$ for the same $T$; here the ascending
edges of $\Gamma$ are indicated by ordinary lines (which should be directed
up), and the horizontal edges by double lines or arcs (which should be directed
from left to right).
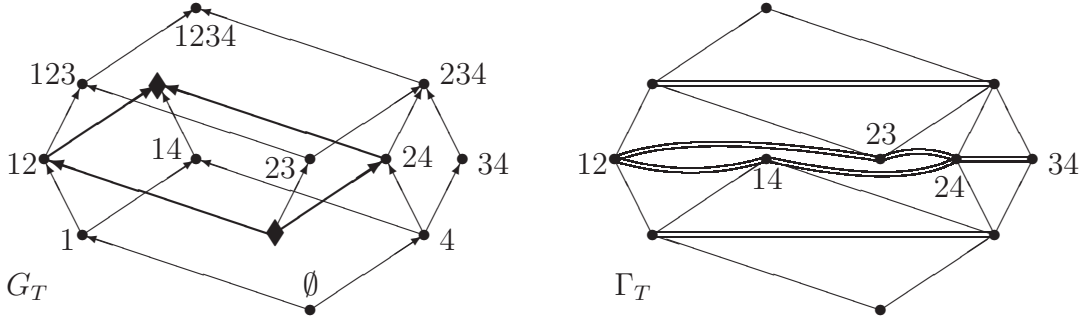
\begin{figure}[htb]
 \begin{center}
  \unitlength=1mm
  \begin{picture}(140,40)
   \put(0,0){\begin{picture}(60,40)
  \put(40,0){\vector(-3,1){29.5}}
  \put(10,10){\vector(-1,2){4.7}}
  \put(5,20){\vector(1,2){4.7}}
  \put(10,30){\vector(3,2){14.7}}
  \put(55,30){\vector(-3,1){29.5}}
  \put(60,20){\vector(-1,2){4.7}}
  \put(55,10){\vector(1,2){4.7}}
  \put(40,0){\vector(3,2){14.7}}
  \put(40,0){\circle*{1.5}}
  \put(10,10){\circle*{1.5}}
  \put(34,9){$\blacklozenge$}
  \put(55,10){\circle*{1.5}}
  \put(5,20){\circle*{1.5}}
  \put(25,20){\circle*{1.5}}
  \put(40,20){\circle*{1.5}}
  \put(50,20){\circle*{1.5}}
  \put(60,20){\circle*{1.5}}
  \put(10,30){\circle*{1.5}}
  \put(18.5,28.5){$\blacklozenge$}
  \put(55,30){\circle*{1.5}}
  \put(25,40){\circle*{1.5}}
  \put(40,20){\vector(-3,1){29.5}}
  \put(55,10){\vector(-3,1){29.5}}
  \put(25,20){\vector(-1,2){4.5}}
  \put(55,10){\vector(-1,2){4.7}}
  \put(35,10){\vector(1,2){4.7}}
  \put(50,20){\vector(1,2){4.7}}
  \put(10,10){\vector(3,2){14.7}}
  \put(40,20){\vector(3,2){14.7}}
\thicklines{
  \put(35,10){\vector(-3,1){29.5}}
  \put(50,20){\vector(-3,1){29.5}}
  \put(5,20){\vector(3,2){14.7}}
  \put(35,10){\vector(3,2){14.7}}
}
  \put(39,2){$\emptyset$}
  \put(7,8){1}
  \put(57,8){4}
  \put(0,18){12}
  \put(19,20){14}
  \put(34,17.5){23}
  \put(52,19){24}
  \put(62,18){34}
  \put(3,30){123}
  \put(57,30){234}
  \put(22,35){1234}
  \put(0,2){$G_T$}
    \end{picture}}
%
   \put(75,0){\begin{picture}(60,40)
  \put(40,0){\circle*{1.5}}
  \put(10,10){\circle*{1.5}}
  \put(55,10){\circle*{1.5}}
  \put(5,20){\circle*{1.5}}
  \put(25,20){\circle*{1.5}}
  \put(40,20){\circle*{1.5}}
  \put(50,20){\circle*{1.5}}
  \put(60,20){\circle*{1.5}}
  \put(10,30){\circle*{1.5}}
  \put(55,30){\circle*{1.5}}
  \put(25,40){\circle*{1.5}}
  \put(40,0){\line(-3,1){30}}
  \put(10,10){\line(-1,2){5}}
  \put(5,20){\line(1,2){5}}
  \put(10,30){\line(3,2){15}}
  \put(55,30){\line(-3,1){30}}
  \put(60,20){\line(-1,2){5}}
  \put(55,10){\line(1,2){5}}
  \put(40,0){\line(3,2){15}}
  \put(40,20){\line(-3,1){30}}
  \put(55,10){\line(-3,1){30}}
  \put(55,10){\line(-1,2){5}}
  \put(50,20){\line(1,2){5}}
  \put(10,10){\line(3,2){15}}
  \put(40,20){\line(3,2){15}}
  \put(10,10.3){\line(1,0){45}}
  \put(10,9.7){\line(1,0){45}}
  \put(10,30.3){\line(1,0){45}}
  \put(10,29.7){\line(1,0){45}}
  \put(50,20.3){\line(1,0){10}}
  \put(50,19.7){\line(1,0){10}}
  \qbezier(5,20.3)(15,17.3)(25,20.3)
  \qbezier(5,19.7)(15,16.7)(25,19.7)
  \qbezier(25,20.3)(43,16.3)(50,20.3)
  \qbezier(25,19.7)(43,15.7)(50,19.7)
  \qbezier(5,20.3)(15,24.3)(40,20.3)
  \qbezier(5,19.7)(15,23.7)(40,19.7)
  \qbezier(40,20.3)(45,22.3)(50,20.3)
  \qbezier(40,19.7)(45,21.7)(50,19.7)
  \put(5,2){$\Gamma_T$}
  \put(0,18){12}
  \put(23,16){14}
  \put(38,22){23}
  \put(47,15){24}
  \put(62,18){34}
    \end{picture}}
  \end{picture}
   \end{center}
\caption{Graphs $G_T$ and $\Gamma_T$} \label{fig:GT-Gamma}
  \end{figure}

In what follows we write $\prec_{G'}$ for the natural partial order on the
vertices of an acyclic directed graph $G'$, i.e. $x\prec_{G'} y$ if vertices
$x,y$ are connected in $G'$ by a directed path from $x$ to $y$.

Since each ascending edge of $\Gamma$ goes from one level to the next one and
each horizontal edge goes from left to right, the graph $\Gamma$ is acyclic; so
$\prec_\Gamma$ is a partial order on $\mathfrak{S}_T$. Moreover, $\Gamma$
possesses the following nice property.
   \begin{prop}  \label{pr:Gamma-lat}
The poset $(\mathfrak{S}_T,\prec_{\Gamma_T})$ is a lattice.
   \end{prop}

This is provided by the following lattice property of acyclic planar graphs; it
is pointed out in~\cite{Bir} (p.~32, Exer.~7(a)), but for the completeness of our
description, we give a proof.

  \begin{lemma} \label{lm:lattice}
Let $G'=(V',E')$ be a planar acyclic directed graph embedded in the plane.
Suppose that the partial order $\Pscr=(V',\prec_{G'})$ has a unique minimal
element $s$ and a unique maximal element $t$ and that both $s,t$ are contained
in (the boundary of) the same face of $G'$. Then $\Pscr$ is a lattice.
  \end{lemma}
  \begin{proof}
Let $\prec$ stand for $\prec_{G'}$. One may assume that both $s,t$ belong to
the outer (unbounded) face of $G'$.

Consider two vertices $x,y\in V'$. Let $L$ be the set of maximal elements in
$\{v\in V'\colon v\preceq x,y\}$, and $U$ the set of minimal elements in
$\{v\in V'\colon v\succeq x,y\}$. Since $\Pscr$ has unique minimal and maximal
elements, both $L,U$ are nonempty. We have to show that $|L|=|U|=1$. Below by a
path we mean a directed path.

Suppose, for a contradiction, that $L$ contains two distinct elements $a,b$.
Take four paths in $G'$ connecting $\{a,b\}$ to $\{x,y\}$: a path $P_x$ from
$a$ to $x$, a path $P_y$ from $a$ to $y$, a path $P'_x$ from $b$ to $x$, and a
path $P'_y$ from $b$ to $y$. Then $P_x$ meets $P_y$ only at the vertex $a$ and
is disjoint from $P'_y$ (otherwise at least one of the lower bounds $a,b$ for
$x,y$ is not maximal). Similarly, $P'_x$ meets $P'_y$ only at $b$ and is
disjoint from $P_y$. Also we may assume that $P_x\cap P'_x=\{x\}$ (otherwise we
could replace $x$ by another common point $x'$ of $P_x,P'_x$; then $a,b\in
L(x',y)$ and $x'$ is ``closer'' to $a,b$ than $x$). Similarly, we may assume
that $P_y\cap P'_y=\{y\}$. Let $R$ be the closed region (homeomorphic to a
disc) surrounded by $P_x,P_y,P'_x,P'_y$. The fact that $s,t$ belong to the
outer face easily implies that they lie outside $R$.

Take a path $Q_a$ from $s$ to $a$ and a path $Q_b$ from $s$ to $b$. For any
vertex $v\ne a$ of $Q_a$, relation $v\prec a$ implies that $v$ belongs to
neither $P_x\cup P_y$ (otherwise $a\prec v$ would take place) nor $P'_x\cup
P'_y$ (otherwise $b\preceq v\prec a$ would take place). Therefore, $Q_a$ meets
$R$ only at $a$. Similarly, $Q_b\cap R=\{b\}$.

Let $v$ be the last common vertex of $Q_a$ and $Q_b$. Take the part $Q'$ of
$Q_a$ from $v$ to $a$, and the part $Q''$ of $Q_b$ from $v$ to $b$. Then
$Q'\cap Q''=\{v\}$. Observe that $R$ is contained either in the closed region
$R_1$ surrounded by $Q',Q'',P_x,P'_x$, or in the closed region $R_2$ surrounded
by $Q',Q'',P_y,P'_y$. One may assume that $R\subset R_1$ (this case is
illustrated in the picture below). Then $y$ is an interior point in $R_1$.
Obviously, $t\not\in R_1$. Now since $y\prec t$, there exists a path from $y$
to $t$ in $G'$. This path meets the boundary of $R_1$ at some vertex $z$. But
if $z$ occurs in $P_x\cup P'_x$, then $y\prec z\preceq x$, and if $z$ occurs in
$Q'$ (in $Q''$), then $y\prec z\preceq a$ (resp. $y\prec z\preceq b$); a
contradiction.

 \begin{center}
  \unitlength=1mm
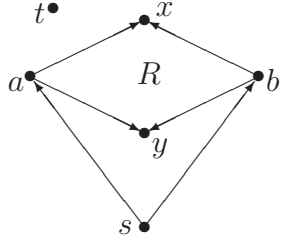
  \begin{picture}(40,29)
  \put(20,0){\circle*{1.5}}
  \put(20,12.5){\circle*{1.5}}
  \put(5,20){\circle*{1.5}}
  \put(35,20){\circle*{1.5}}
  \put(20,27.5){\circle*{1.5}}
  \put(8,29){\circle*{1.5}}
  \put(20,0){\vector(-3,4){14.5}}
  \put(20,0){\vector(3,4){14.5}}
  \put(5,20){\vector(2,-1){14.5}}
  \put(5,20){\vector(2,1){14.5}}
  \put(35,20){\vector(-2,-1){14.5}}
  \put(35,20){\vector(-2,1){14.5}}
  \put(16.5,-1){$s$}
  \put(2,18){$a$}
  \put(36,18){$b$}
  \put(21,10){$y$}
  \put(21.5,28){$x$}
  \put(19,19){$R$}
  \put(5.5,27){$t$}
  \end{picture}
   \end{center}

Thus, $|L|=1$. The equality $|U|=1$ is obtained by reversing the edges of $G'$.
  \end{proof}

Now we prove Proposition~\ref{pr:Gamma-lat} as follows. Consider the image
$\sigma(\Gamma)$ of $\Gamma$ on the disc $D_T$, where the image $\sigma(e)$ of
the horizontal edge $e$ drawn in a white tile $\tau$ is naturally defined to be
the corresponding directed diagonal of the square $\sigma(\tau)$. Since the
embedding of $\sigma(G_T)$ in $D_T$ is planar, so is the embedding of
$\sigma(\Gamma)$. Also: (i) ~$\sigma(bd(Z))$ is the boundary of $D_T$; (ii)
each boundary vertex of $Z$ lies on a directed path from $z_0$ to $z_n$ in
$G_T$, which belongs to $\Gamma$ as well; and (iii) $\Gamma$ is acyclic. Next,
if a nonterminal vertex of $T$ does not belong to the left (right) boundary of
$Z$, then $v$ is the right (resp. left) vertex of some white tile, as is seen
from (iv) (resp.~(v)) in Corollary~\ref{cor:ord-mixed}. This implies that
   \begin{numitem1}
for $v\in\mathfrak{S}_T$, if $v$ is not in $\ell bd(Z)$ (not in $rbd(Z)$), then
there exists a horizontal edge in $\Gamma$ entering (resp. leaving) $v$.
   \label{eq:horedge}
   \end{numitem1}
Thus, $z_0$ and $z_n$ are the unique minimal and maximal vertices in $\Gamma$.
Applying Lemma~\ref{lm:lattice} to $\sigma(\Gamma)$, we conclude that $\Gamma$
determines a lattice, as required. \hfill\qed

\section{Equality of two posets}  \label{sec:2posets}

The goal of this section is to show the following crucial property of the
auxiliary graph $\Gamma=\Gamma_T$ introduced in the previous section.
 \begin{theorem} \label{tm:2posets}
{\rm (Auxiliary Theorem)} ~For a g-tiling $T$ on $Z_n$, the partial orders on
$\mathfrak{S}_T$ given by $\prec^\ast$ and by $\prec_\Gamma$ are equal.
  \end{theorem}

This together with Proposition~\ref{pr:Gamma-lat} will imply
Theorem~\ref{tm:prec-lat}, thus completing the whole proof of the main results
(Theorem~B and, further, Theorem~A) of this paper, as is explained in
Section~\ref{sec:proof}. The proof of Theorem~\ref{tm:2posets} is given
throughout this section. We keep notation from Section~\ref{sec:aux_graph}.

The fact that $\prec_\Gamma$ implies $\prec^\ast$ is easy. Indeed, since each
edge $e=(A,B)$ of $\Gamma$ is of the form either $(X,Xi)$ (when $e$ is
ascending) or $(Xi,Xj)$ with $i<j$ (when $e$ is horizontal), we have $A\lessdot
B$ and $|A|\le|B|$, whence $A\prec^\ast B$. Then for any $C,D\in
\mathfrak{S}_T$ satisfying $C\prec_\Gamma D$, the relation $C\prec^\ast D$ is
obtained by considering a directed path from $C$ to $D$ in $\Gamma$ and using
the transitivity property~\refeq{AAA}.

It remains to show the converse implication, i.e. the following
  \begin{prop}  \label{pr:LZ-Gamma}
~For a g-tiling $T$ on $Z=Z_n$, let two sets (nonterminal vertices)
$A,B\in\mathfrak{S}_T$ satisfy $A\prec^\ast B$. Then $A\prec_\Gamma B$, i.e.
the graph $\Gamma=\Gamma_T$ contains a directed path from $A$ to $B$.
  \end{prop}

The rest of this section is devoted to proving this proposition, which is
rather long and appeals to results on contractions and expansions mentioned in
Subsection~\ref{ssec:s-c-e}(C,D). Moreover, we need to conduct a more
meticulous analysis of structural features of the graphs $G_T$ and $\Gamma$,
and of the action of contraction operations on $\Gamma$; this is given in
Subsections~\ref{ssec:ref_prop}--\ref{ssec:contr_Gamma}. Using these, we then
prove the desired implication in Subsection~\ref{ssec:proof_impl}.


\subsection{Refined properties of the graphs $G^{fw}$ and $\Gamma$.}
\label{ssec:ref_prop}

We start with one fact which immediately follows from the planarity of
principal trees $K_h$ defined in part D of Subsection~\ref{ssec:s-c-e}.
   \begin{numitem1}
~the edges of $K_h$ can be (uniquely) ordered as $e_1,\ldots,e_p$ so that for
$1\le q<q'\le p$, the edge $e_{q'}$ lies on the right from $e_q$ (in
particular, $e_1=z^\ell_{h-1}z^\ell_h$ and $e_p=z^r_{n-h+1}z^r_{n-h}$);
equivalently, consecutive edges $e_q,e_{q+1}$, with labels $i_q,i_{q+1}$,
respectively, either leave a common vertex and satisfy $i_q<i_{q+1}$, or enter
a common vertex and satisfy $i_q>i_{q+1}$.
  \label{eq:Kh}
  \end{numitem1}

We denote the sequence of edges of $K_h$ in this order by $E_h$. Also we denote
the sequence of vertices of $K_h$ occurring in level $h-1$ (level $h$) and
ordered from left to right by $V^{low}_h$ (resp. $V^{up}_h$).

Recall that the common vertices of two neighboring principal trees
$K_h,K_{h+1}$ are called critical vertices in level $h$. Let $U_h$
denote the sequence of these vertices ordered from left to right:
   $$
   U_h:=V^{up}_h\cap V^{low}_{h+1}.
   $$
The picture below illustrates an example of neighboring principal trees
$K_h,K_{h+1}$; here the critical vertices in level $h$ are indicated by
circles.

 \begin{center}
  \unitlength=1mm
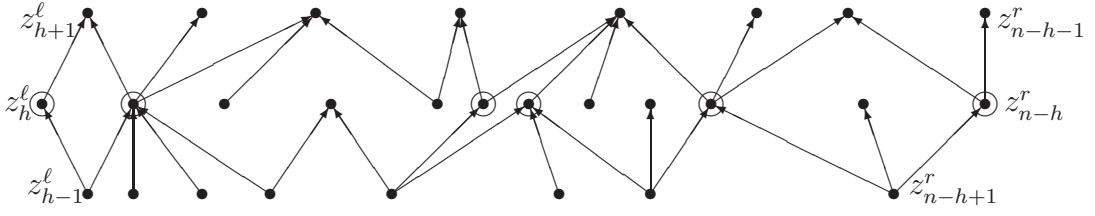
  \begin{picture}(145,25)
  \put(16,0){\circle*{1.5}}
  \put(22,0){\circle*{1.5}}
  \put(31,0){\circle*{1.5}}
  \put(40,0){\circle*{1.5}}
  \put(56,0){\circle*{1.5}}
  \put(78,0){\circle*{1.5}}
  \put(90,0){\circle*{1.5}}
  \put(122,0){\circle*{1.5}}
  \put(10,12){\circle*{1.5}}
  \put(22,12){\circle*{1.5}}
  \put(34,12){\circle*{1.5}}
  \put(48,12){\circle*{1.5}}
  \put(62,12){\circle*{1.5}}
  \put(68,12){\circle*{1.5}}
  \put(74,12){\circle*{1.5}}
  \put(82,12){\circle*{1.5}}
  \put(90,12){\circle*{1.5}}
  \put(98,12){\circle*{1.5}}
  \put(118,12){\circle*{1.5}}
  \put(134,12){\circle*{1.5}}
  \put(10,12){\circle{3}}
  \put(22,12){\circle{3}}
  \put(68,12){\circle{3}}
  \put(74,12){\circle{3}}
  \put(98,12){\circle{3}}
  \put(134,12){\circle{3}}
  \put(16,24){\circle*{1.5}}
  \put(31,24){\circle*{1.5}}
  \put(46,24){\circle*{1.5}}
  \put(65,24){\circle*{1.5}}
  \put(86,24){\circle*{1.5}}
   \put(104,24){\circle*{1.5}}
  \put(116,24){\circle*{1.5}}
  \put(134,24){\circle*{1.5}}
  \put(16,0){\vector(-1,2){5.7}}
  \put(16,0){\vector(1,2){5.7}}
  \put(22,0){\vector(0,1){11.4}}
  \put(31,0){\vector(-3,4){8.7}}
  \put(40,0){\vector(-3,2){17.5}}
  \put(40,0){\vector(2,3){7.7}}
  \put(56,0){\vector(-2,3){7.7}}
  \put(56,0){\vector(1,1){11.5}}
  \put(56,0){\vector(3,2){17.5}}
  \put(78,0){\vector(-1,3){3.8}}
  \put(90,0){\vector(-4,3){15.5}}
  \put(90,0){\vector(0,1){11.4}}
  \put(90,0){\vector(2,3){7.7}}
  \put(122,0){\vector(-2,1){23.5}}
  \put(122,0){\vector(-1,3){3.8}}
  \put(122,0){\vector(1,1){11.5}}
  \put(10,12){\vector(1,2){5.7}}
  \put(22,12){\vector(-1,2){5.7}}
  \put(22,12){\vector(3,4){8.7}}
  \put(22,12){\vector(2,1){23.5}}
  \put(34,12){\vector(1,1){11.5}}
  \put(62,12){\vector(-4,3){15.5}}
  \put(62,12){\vector(1,4){2.9}}
  \put(68,12){\vector(-1,4){2.9}}
  \put(68,12){\vector(3,2){17.5}}
  \put(74,12){\vector(1,1){11.5}}
  \put(82,12){\vector(1,3){3.8}}
  \put(98,12){\vector(-1,1){11.5}}
  \put(98,12){\vector(1,2){5.7}}
  \put(98,12){\vector(3,2){17.5}}
  \put(134,12){\vector(-3,2){17.5}}
  \put(134,12){\vector(0,1){11.4}}
  \put(8,0){$z^\ell_{h-1}$}
  \put(5,11){$z^\ell_h$}
  \put(7,22){$z^\ell_{h+1}$}
  \put(124,0){$z^r_{n-h+1}$}
  \put(136.5,11){$z^r_{n-h}$}
  \put(135.5,22){$z^r_{n-h-1}$}
 \end{picture}
   \end{center}

We need to explore the structure of $G_T$ and $\Gamma$ in a neighborhood of
level $h$ in more details. For vertices $x,y$ of $K_h$, let $P_h(x,y)$ denote
the (unique) path from $x$ to $y$ in $K_h$; in this path the vertices in levels
$h-1$ and $h$ alternate. When two consecutive edges of $P_h(x,y)$ enter their
common vertex, say, $w$ (lying in level $h$), we call $w$ a
$\wedge$-\emph{vertex} in this path; otherwise (when $e,e'$ leave $w$) we call
$w$ a $\vee$-\emph{vertex}. Also we denote by $K_h(x,y)$ the minimal subtree of
$K_h$ containing $x,y$ and all edges incident to \emph{intermediate} vertices
of $P_h(x,y)$.

Consider two consecutive critical vertices $u,v$ in level $h$, where $v$ is the
immediate successor of $u$ in $U_h$. Then the subtrees $K_h(u,v)$ and
$K_{h+1}(u,v)$ intersect exactly at the vertices $u,v$. In particular, the
concatenation of $P_{h+1}(u,v)$ and $P^{rev}_h(u,v)$ forms a simple cycle,
denoted by $C(u,v)=C_h(u,v)$, in the graph $G^{fw}$ induced by the fully white
edges (forming the set $E^{asc}$). Define:

$\Omega(u,v)=\Omega_h(u,v)$ to be the closed region in $Z$ surrounded by
$C(u,v)$;

$\Omega^\ast(u,v)=\Omega^\ast_h(u,v)$ to be the closed region in the disc $D_T$
surrounded by $\sigma(C(u,v))$;

$T(u,v)=T_h(u,v)$ to be the set of tiles $\tau\in T$ such that $\sigma(\tau)$
lies in $\Omega^\ast(u,v)$.

\noindent (For example, in the graph $G_T$ drawn in
Figures~\ref{fig:GT},\ref{fig:GT-Gamma}, the vertices $12$ and $24$ are
consecutive critical vertices in level 2 and the cycle $C(12,24)$
passes $12,123,23,234,24,4,14,1,12$.) Clearly each tile in $T$ belongs
to exactly one set $T_h(u,v)$.

Let $\Cscr$ be the set of cycles $C_h(u,v)$ among all $h,u,v$ as above, and let
$\tilde G$ be the subgraph of $G^{fw}$ that is the union of these cycles; this
subgraph has planar embedding in $Z$ because $G^{fw}$ does so. Observe that
each boundary edge of $Z$ belongs to exactly one cycle in $\Cscr$ and that any
other edge of $\tilde G$ belongs to exactly two such cycles. It follows that
  \begin{numitem1}
the regions $\Omega(\cdot,\cdot)$ give a subdivision of $Z$ and are exactly the
faces of the graph $\tilde G$; similarly, the regions
$\Omega^\ast(\cdot,\cdot)$ give a subdivision of $D_T$ and are the faces of
$\sigma(\tilde G)$; the face structures of the planar graphs $\tilde G$ and
$\sigma(\tilde G)$ are isomorphic (more precisely, the restriction of $\sigma$
to $\tilde G$ can be extended to a homeomorphism of $Z$ to $D_T$ which maps
each $\Omega_h(u,v)$ onto $\Omega^\ast_h(u,v)$).
  \label{eq:Omega_cover}
  \end{numitem1}
Hereinafter, speaking of a face of a planar graph, we mean an inner (bounded)
face.

Any vertex $v$ of a cycle $C(u,v)=C_h(u,v)$ belongs to level $h-1,h$ or $h+1$,
and we call $v$ a \emph{peak} in $C(u,v)$ if it has height $\ne n$, i.e. when
$v$ is either a $\vee$-vertex of $P_h(u,v)$ or a $\wedge$-vertex of
$P_{h+1}(u,v)$. Also we distinguish between two sorts of edges $e$ in
$(K_h(u,v)\cup K_{h+1}(u,v))-C(u,v)$, by saying that $e$ is an \emph{inward
pendant} edge w.r.t. $C(u,v)$ if it lies in $\Omega(u,v)$, and an \emph{outward
pendant} edge otherwise. (In fact, outward pendant edges will not be important
for us later on). See the picture where the peaks are indicated by symbol
$\otimes$, the inward pendant edges by $\iota$, and the outward pendant edges
by $o$.
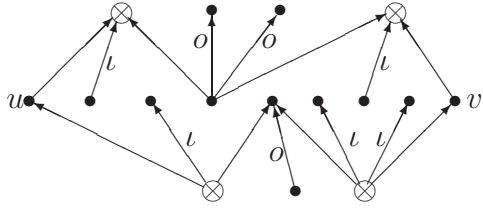
 \begin{center}
  \unitlength=1mm
  \begin{picture}(80,25)
  \put(45,0){\circle*{1.5}}
  \put(10,12){\circle*{1.5}}
  \put(10,12){\circle*{1.5}}
  \put(18,12){\circle*{1.5}}
  \put(26,12){\circle*{1.5}}
  \put(34,12){\circle*{1.5}}
  \put(42,12){\circle*{1.5}}
  \put(48,12){\circle*{1.5}}
  \put(54,12){\circle*{1.5}}
  \put(60,12){\circle*{1.5}}
  \put(66,12){\circle*{1.5}}
  \put(34,24){\circle*{1.5}}
  \put(43,24){\circle*{1.5}}
  \put(33,0.5){\vector(-2,1){22.5}}
  \put(33.2,1.2){\vector(-2,3){6.9}}
  \put(34.8,1.2){\vector(2,3){6.9}}
  \put(45,0){\vector(-1,4){2.8}}
  \put(53,1){\vector(-1,1){10.5}}
  \put(53.5,1){\vector(-1,2){5.3}}
  \put(54.5,1){\vector(1,2){5.3}}
  \put(55,1){\vector(1,1){10.5}}
  \put(10,12){\vector(1,1){11}}
  \put(34,12){\vector(-1,1){11}}
  \put(34,12){\vector(2,1){23}}
  \put(66,12){\vector(-2,3){7.2}}
  \put(18,12){\vector(1,3){3.5}}
  \put(34,12){\vector(0,1){11.3}}
  \put(34,12){\vector(3,4){8.5}}
  \put(54,12){\vector(1,3){3.5}}
  \put(32.5,-1){$\otimes$}
  \put(52.5,-1){$\otimes$}
  \put(20.5,22.5){$\otimes$}
  \put(56.5,22.5){$\otimes$}
  \put(7,11){$u$}
  \put(67.5,11){$v$}
  \put(30.5,6){$\iota$}
  \put(52,6){$\iota$}
  \put(55.5,6){$\iota$}
  \put(20,16){$\iota$}
  \put(56,16){$\iota$}
  \put(41.5,4){$o$}
  \put(31.5,19){$o$}
  \put(40.5,19){$o$}
\end{picture}
   \end{center}

Clearly each edge in $G^{fw}-\tilde G$ is an inward pendant edge of
exactly one cycle in $\Cscr$.

Let $\mathfrak{S}(u,v)$ be the set of nonterminal vertices $x$ such that
$\sigma(x)\in\Omega^\ast(u,v)$ and $x$ \emph{is not a peak in $C(u,v)$}. The
next lemma describes a number of important properties. Hereinafter $\Gamma^h$
denotes the subgraph of $\Gamma$ induced by the (horizontal) edges in level
$h$.
  \begin{lemma} \label{lm:hor_paths}
~For $h,u,v$ as above:

{\rm(i)} the fully white edges $e$ such that $e\not\in C(u,v)$ and
$\sigma(e)\subset \Omega^\ast(u,v)$ are exactly the inward pendant edges for
$C(u,v)$;

{\rm(ii)} all tiles in $T(u,v)$ are of the same height $h$.

{\rm(iii)} $\mathfrak{S}(u,v)$ is exactly the set of vertices that are
contained in directed paths from $u$ to $v$ in $\Gamma^h$.
  \end{lemma}
  \begin{proof}
~Let $Q$ be the graph whose vertices are the tiles in $T$ and whose edges
correspond to the pairs $\tau,\tau'$ of tiles that have a common edge not in
$\tilde G$. One can see that any two tiles in the same set $T_h(u,v)$ are
connected by a path in $Q$. On the other hand, if two tiles occur in different
sets $T_h(u,v)$ and $T_{h'}(u',v')$, then these tiles cannot be connected by a
path in $Q$. Therefore, the connected components of $Q$ correspond to the sets
$T_h(u,v)$. Considering a pair $e,e'$ of consecutive edges in a cycle $C(u,v)$
(which are fully white) and applying Corollary~\ref{cor:ord-mixed} to the
common vertex $w$ of $e,e'$, we observe that the set $F_T(w)$ of tiles at $w$
is partitioned into two subsets $F^1(w),F^2(w)$ such that: (a) the interior of
each tile $\tau$ in $F^1(w)$ meets $\Omega(u,v)$ (in particular, $\tau$ lies
between $e$ and $e'$ when $w$ is a peak in $C(u,v)$), whereas the interior of
each tile in $F^2(w)$ is disjoint from $\Omega(u,v)$; (b) the tiles in $F^1(w)$
belong to a path in $Q$; and (c) each inward pendant edge at $w$ w.r.t.
$C(u,v)$ (if any) belongs to some tile in $F^1(w)$. It follows that all tiles
in the set $\Fscr:=\cup(F^1(w)\colon w\in C(u,v))$ belong to the same component
of $Q$. Then all squares in $\sigma(\Fscr)$ are contained in one face of
$\sigma(\tilde G)$, and at the same time, they cover the cycle
$\sigma(C(u,v))$. This is possible only if $\Fscr\subseteq T(u,v)$. Now (i)
easily follows.

Next, observe that for any vertex $w$ of $C(u,v)$, the set $F^1(w)$ as above
contains a tile of height $h$. Therefore, in order to obtain~(ii), it suffices
to show that any two tiles $\tau,\tau'\in T(u,v)$ sharing an edge $e$ have the
same height. This is obvious when $\tau,\tau'$ have either the same top or the
same bottom (in particular, when one of these tiles is black). Suppose this is
not the case. Then both tiles are white, and the edge $e$ connects the left or
right vertex of one of them to the right or left vertex of the other. So both
ends of $e$ are nonterminal and $e$ is fully white. By~(i), $e$ is an inward
pendant edge for $C(u,v)$ and one end $x$ of $e$ is a peak. Therefore, both
$\tau,\tau'$ belong to the set $F^1(x)$ and lie between the two edges of
$C(u,v)$ incident to $x$. But both edges either enter $x$ or leave $x$ (since
$x$ is a peak), implying that $x$ must be either the top or the bottom of both
$\tau,\tau'$; a contradiction. Thus, (ii) is valid.

Finally, to see~(iii), consider a vertex $x\in\mathfrak{S}(u,v)$ different from
$u$. Then $x$ is not in $\ell bd(Z)$; for otherwise we would have $u=z_h^\ell$
and $x\in\{z_{h-1}^\ell,z_{h+1}^\ell\}$, implying that $x$ is a peak in
$C(u,v)$. Therefore (cf. Corollary~\ref{cor:ord-mixed}(iv)), there exists a
white tile $\tau$ such that $r(\tau)=x$. Both $x,\tau$ have the same height.
Suppose that $\tau\notin T(u,v)$. Then $x\in C(u,v)$ and $x$ is of height $h$
(since $x$ is not a peak in $C(u,v)$). So $\tau$ is of height $h$ as well, and
in view of~(ii), $\tau$ belongs to some collection $T_h(u',v')\ne T_h(u,v)$.
One can see that the latter is possible only if $v'=u$, implying $x=u$; a
contradiction.

Hence, $\tau$ belongs to $T(u,v)$ and has height $h$. Take the vertex
$y:=\ell(\tau)$. Then $y$ is nonterminal, $\sigma(y)\in\Omega^\ast(u,v)$, and
there is a horizontal edge from $y$ to $x$ in $\Gamma$. Also $y$ is not a peak
in $C(u,v)$ (since $y$ is of height $h$). So $y\in\mathfrak{S}(u,v)$. Apply a
similar procedure to $y$, and so on. Eventually, we reach the vertex $u$,
obtaining a directed path in $\Gamma^h$ going from $u$ to the initial vertex
$x$. A directed path from $x$ to $v$ is constructed in a similar way.

Conversely, let $P$ be a directed path from $u$ to $v$ in $\Gamma^h$. The fact
that all vertices of $P$ belong to $\mathfrak{S}(u,v)$ is easily shown by
considering the sequence of white tiles corresponding to the edges of $P$ and
using the fact that all these tiles have height $h$.

Thus,~(iii) is valid and the lemma is proven.
  \end{proof}

Let the sequence $U_h$ consist of (critical) vertices
$u_0=z^\ell_h,u_1,\ldots,u_{r-1},u_r=z^r_{n-h}$. We abbreviate
$T_h(u_{p-1},u_p)$ as $T_h(p)$, and denote by $G^h(p)$ the subgraph of $G_T$
whose image by $\sigma$ lies in $\Omega^\ast_h(u_{p-1},u_p)$. By~(ii) in
Lemma~\ref{lm:hor_paths}, $T_h(1),\ldots,T_h(r)$ give a partition of the set of
tiles of height $h$.

In its turn, (iii) in this lemma shows that the graph $\Gamma^h$ is
represented as the concatenation of $\Gamma^h(1),\ldots,\Gamma^h(r)$,
where each graph $\Gamma^h(p)$ is the union of (horizontal) directed
paths from $u_{p-1}$ to $u_p$ in $\Gamma$. We refer to $\Gamma^h(p)$ as
$p$-th \emph{hammock of $\Gamma^h$ \emph{(or in level $h$)} beginning
at $u_{p-1}$ and ending at $u_p$}, and similarly for the subgraph
$\sigma(\Gamma^h(p))$ of $\sigma(\Gamma^h)$. The fact that
$\sigma(\Gamma^h(p))$ is the union of directed paths from
$\sigma(u_{p-1})$ to $\sigma(u_p)$ easily implies that
  \begin{numitem1}
the boundary of each face $F$ of the hammock $\sigma(\Gamma^h(p))$ is
formed by two directed paths with the same beginning $x$ and the same
end $y$;
   \label{eq:cell}
   \end{numitem1}
we say that the face $F$ \emph{begins at $x$ and ends at $y$}. The
\emph{extended hammock} $\bar\Gamma^h(p)$ is constructed by adding to
$\Gamma^h$ the cycle $C(u_{p-1},u_p)$ and the inward pendant edges for it (all
added edges are ascending in $\Gamma$); this is just the subgraph of $\Gamma$
whose image by $\sigma$ is contained in $\Omega^\ast(u_{p-1},u_p)$.

Applying Lemma~\ref{lm:lattice} to the planar graph $\sigma(\Gamma^h)$, we
obtain that
   \begin{numitem1}
the partial order $(\mathfrak{S}^h,\prec_{\Gamma^h})$, where
$\mathfrak{S}^h:=\{X\in\mathfrak{S}_T\colon |X|=h\}$, is a lattice with
the minimal element $z^\ell_h$ and the maximal element $z^r_{n-h}$;
similarly, for each $p=1,\ldots,r$,
~$(\mathfrak{S}(u_{p-1},u_p),\prec_{\Gamma^h(p)})$ is a lattice with
the minimal element $u_{p-1}$ and the maximal element $u_p$.
   \label{eq:p-o-h}
   \end{numitem1}

An example of $\sigma(\Gamma^h)$ with $r=5$ is drawn in the picture; here all
edges are directed from left to right.
 \begin{center}
  \unitlength=1mm
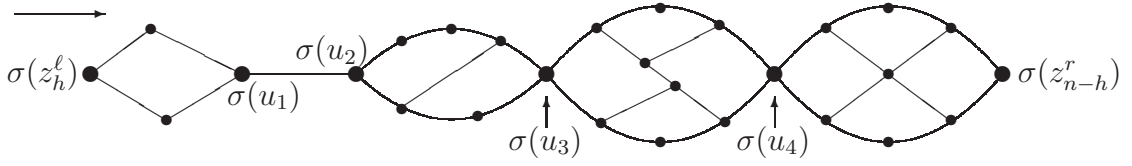
  \begin{picture}(140,22)
  \put(10,10){\circle*{2}}
  \put(30,10){\circle*{2}}
  \put(45,10){\circle*{2}}
  \put(70,10){\circle*{2}}
  \put(100,10){\circle*{2}}
  \put(130,10){\circle*{2}}
  \put(20,4){\circle*{1.5}}
  \put(18,16){\circle*{1.5}}
  \put(10,10){\line(4,3){8}}
  \put(10,10){\line(5,-3){10}}
  \put(20,4){\line(5,3){10}}
  \put(18,16){\line(2,-1){12}}
  \put(30,10){\line(1,0){15}}
  \put(-1,9){$\sigma(z^\ell_h)$}
  \put(132,9){$\sigma(z^r_{n-h})$}
  \put(28,6){$\sigma(u_1)$}
  \put(37,12){$\sigma(u_2)$}
  \put(65,0){$\sigma(u_3)$}
  \put(95,0){$\sigma(u_4)$}
  \put(0,18){\vector(1,0){12}}
  \put(70,3){\vector(0,1){4}}
  \put(100,3){\vector(0,1){4}}
  \qbezier(45,10)(57.5,22)(70,10)
  \qbezier(45,10)(57.5,-2)(70,10)
  \put(51.2,5.8){\line(3,2){12.6}}
  \put(51,5.4){\circle*{1.5}}
  \put(64,14.4){\circle*{1.5}}
  \put(51,14.4){\circle*{1.5}}
  \put(57.5,16){\circle*{1.5}}
  \put(61,4.5){\circle*{1.5}}
  \qbezier(70,10)(85,28)(100,10)
  \qbezier(70,10)(85,-8)(100,10)
  \put(77,16){\line(4,-3){16}}
  \put(76.6,16.2){\circle*{1.5}}
  \put(93.5,3.6){\circle*{1.5}}
  \put(83,11.5){\circle*{1.5}}
  \put(87,8.5){\circle*{1.5}}
  \put(83,11.5){\line(2,1){10}}
  \put(87,8.5){\line(-2,-1){10}}
  \put(92.8,16.5){\circle*{1.5}}
  \put(77.2,3.5){\circle*{1.5}}
  \put(85,18.8){\circle*{1.5}}
  \put(85,1){\circle*{1.5}}
  \qbezier(100,10)(115,28)(130,10)
  \qbezier(100,10)(115,-8)(130,10)
  \put(107,16){\line(4,-3){16}}
  \put(107,4){\line(4,3){16}}
  \put(106.8,16.1){\circle*{1.5}}
  \put(115,10){\circle*{1.5}}
  \put(123.2,16.1){\circle*{1.5}}
  \put(123.2,3.9){\circle*{1.5}}
  \put(106.8,3.9){\circle*{1.5}}
  \put(115,18.8){\circle*{1.5}}
  \put(115,1){\circle*{1.5}}
\end{picture}
   \end{center}

We call a hammock $\Gamma^h(p)$ \emph{trivial} if it has only one edge (which
goes from $u_{p-1}$ to $u_p$). In this case $T_h(p)$ consists of a single white
tile $\tau$ such that both $b(\tau),t(\tau)$ are nonterminal,
$\ell(\tau)=u_{p-1}$ and $r(\tau)=u_p$ (then $\bar \Gamma^h(p)$ is formed by
the four edges of $\tau$ and the horizontal edge from $\ell(\tau)$ to
$r(\tau)$).\medskip


\subsection{Nontrivial hammocks}
\label{ssec:hamm}

Next we describe the structure of a \emph{nontrivial} hammock $\Gamma^h(p)$.
For a white tile in $T_h(p)$ (which, obviously, exists), at least one of its
bottom and top vertices is terminal (for otherwise all edges of this tile are
fully white, implying that its left and right vertices are critical). Then
$|T_h(p)|\ge 2$ and the set $T^b_h(p)$ of black tiles in $T_h(p)$ is nonempty.
We are going to show a one-to-one correspondence between the black tiles and
the faces of $\sigma(\Gamma^h(p))$.

Given a black tile $\tau\in T^b_h(p)$, consider the sequence $x_0,\ldots,x_k$
of the end vertices of the edges $e_0,\ldots,e_k$ leaving $b(\tau)$ and ordered
from left to right (i.e. by increasing their labels), and the sequence
$y_0,\ldots,y_{k'}$ of the beginning vertices of the edges
$e'_0,\ldots,e'_{k'}$ entering $t(\tau)$ and ordered from left to right. Then
$e_0,e_k,e'_0,e'_{k'}$ are the edges of $\tau$, the other edges $e_q,e'_{q'}$
are semi-white, $x_0=y_0=\ell(\tau)$ and $x_k=y_{k'}=r(\tau)$. Also each pair
$e_{q-1},e_q$ belongs to a white tile $\tau_q$, each pair $e'_{q'-1},e'_{q'}$
belongs to a white tile $\tau'_{q'}$, and there are no other tiles having a
vertex at $b(\tau)$ or $t(\tau)$, except for $\tau$. For two consecutive tiles
$\tau_q,\tau_{q+1}$ (resp. $\tau'_{q'},\tau'_{q'+1}$), we have
$r(\tau_q)=\ell(\tau_{q+1})=x_q$ (resp.
$r(\tau'_{q'})=\ell(\tau'_{q'+1})=y_{q'}$). Therefore, the sequence
$(x_0,\ldots,x_k)$ gives a directed path in $\Gamma$, denoted by $\gamma_\tau$,
and similarly,  $(y_0,\ldots,y_{k'})$ gives a directed path in $\Gamma$,
denoted by $\beta_\tau$. Both paths go from $\ell(\tau)$ to $r(\tau)$ and have
no other common vertices (since $x_q=y_{q'}$ for some $0<q<k$ and $0<q'<k'$
would lead to a contradiction with~\refeq{4edges}).

We denote $\beta_\tau\cup\gamma_\tau$ by $\zeta_\tau$, regarding it
both as a graph and as the simple cycle in which the edges of
$\gamma_\tau$ are forward. The closed region in $D_T$ surrounded by
$\sigma(\zeta_\tau)$ (which is a disc) is denoted by $\rho_\tau$. We
call $\beta_\tau$ and $\gamma_\tau$ the \emph{lower} and \emph{upper}
paths in $\zeta_\tau$, respectively, and similarly for the paths
$\sigma(\gamma_\tau)$ and $\sigma(\beta_\tau)$ in $\sigma(\zeta_\tau)$
(a motivation will be clearer later).

Any white tile in $T_h(p)$ has the bottom or top in common with some black
tile. This implies that the graph $\Gamma^h(p)$ is exactly the union of cycles
$\zeta_\tau$ over $\tau\in T^b_h(p)$. Moreover, each edge $e$ of $\Gamma^h(p)$
with $\sigma(e)$ not in the boundary of $\sigma(\Gamma^h(p))$ belongs to two
cycles as above (since such an $e$ is the diagonal of a white tile in which
both the bottom and top vertices are terminal). These facts are strengthened as
follows.
  \begin{lemma} \label{lm:rho-face}
The regions $\rho_\tau$, $\tau\in T^b_h(p)$, are exactly the faces of the graph
$\sigma(\Gamma^h(p))$.
   \end{lemma}
\begin{proof} ~For such a $\tau$, form the region $R_\tau$ in $D_T$ to be the
union of the square $\sigma(\tau)$, the triangles (half-squares) with the
vertices $\sigma(\ell(\tau')),\sigma(r(\tau')),\sigma(b(\tau'))$ over all white
tiles $\tau'\in T$ having the common bottom with $\tau$, and the triangles
(half-squares) with the vertices
$\sigma(\ell(\tau')),\sigma(r(\tau')),\sigma(t(\tau'))$ over all white tiles
$\tau'\in T$ having the common top with $\tau$. One can see that $R_\tau$ is a
disc and its boundary is just $\sigma(\zeta_\tau)$. So $R_\tau=\rho_\tau$.
Obviously, the regions $R_\tau$, $\tau\in T^b_h(p)$, have pairwise disjoint
interiors.
 \end{proof}

Thus, the faces of $\sigma(\Gamma^h(p))$ are generated by the black tiles in
$T_h(p)$; each face $\rho_\tau$ contains $\sigma(\tau)$, begins at $\ell(\tau)$
and ends at $r(\tau)$. Figure~\ref{fig:hammock} illustrates an example with two
black tiles: the subgraph $G^h(p)$ of $G_T$ and the extended hammock for it.
\begin{figure}[htb]
  \begin{center}
  \unitlength=1mm
  \begin{picture}(145,32)
     \put(0,-3){\begin{picture}(70,35)
{\thicklines
  \put(10,18){\line(1,1){12}}
  \put(10,18){\line(4,-3){16}}
  \put(38,18){\line(-1,-1){12}}
  \put(38,18){\line(-4,3){16}}
  \put(26,18){\line(3,2){18}}
  \put(26,18){\line(4,-3){16}}
  \put(60,18){\line(-4,3){16}}
  \put(60,18){\line(-3,-2){18}}
}
  \put(10,18){\line(-1,2){6}}
  \put(26,6){\line(-1,2){6}}
  \put(4,30){\line(4,-3){16}}
  \put(26,18){\line(-1,3){4}}
  \put(10,18){\line(1,-3){4}}
  \put(26,18){\line(-1,-1){12}}
  \put(38,18){\line(-1,2){6}}
  \put(20,18){\line(1,1){12}}
  \put(38,18){\line(1,-3){4}}
  \put(60,18){\line(-1,3){4}}
  \put(38,18){\line(3,2){18}}
  \put(26,18){\line(2,-3){8}}
  \put(44,30){\line(2,-3){8}}
  \put(60,18){\line(2,-3){8}}
  \put(52,18){\line(-3,-2){18}}
  \put(52,18){\line(4,-3){16}}
  \put(10,18){\circle*{1.8}}
  \put(38,18){\circle*{1.2}}
  \put(26,18){\circle*{1.2}}
  \put(60,18){\circle*{1.8}}
  \put(4,30){\circle*{1.2}}
  \put(32,30){\circle*{1.2}}
  \put(56,30){\circle*{1.2}}
  \put(20,18){\circle*{1.2}}
  \put(52,18){\circle*{1.2}}
  \put(14,6){\circle*{1.2}}
  \put(34,6){\circle*{1.2}}
  \put(68,6){\circle*{1.2}}
  \put(24.5,5){$\blacklozenge$}
  \put(20.5,28.5){$\blacklozenge$}
  \put(42.8,28.5){$\blacklozenge$}
  \put(40.5,5){$\blacklozenge$}
  \put(1,17){$u_{p-1}$}
  \put(61.5,17.5){$u_p$}
  \put(13,2){$a$}
  \put(22,2){$b(\tau)$}
  \put(33,2){$a'$}
  \put(40,2){$b(\tau')$}
  \put(67,2){$a''$}
  \put(3,32){$c$}
  \put(18,32){$t(\tau)$}
  \put(31,32){$c'$}
  \put(40,32){$t(\tau')$}
  \put(54,32){$c''$}
  \put(17,16){$x$}
  \put(23,18){$y$}
  \put(39,16.5){$x'$}
  \put(47.5,17){$y'$}
  \put(-2,0){in $G_T$:}
  \put(73,25){\line(1,0){10}}
  \put(73,23){\line(1,0){10}}
  \put(85,24){\line(-2,1){4}}
  \put(85,24){\line(-2,-1){4}}
    \end{picture}}
     \put(100,-5){\begin{picture}(45,40)
  \put(0,22){\circle*{1.5}}
  \put(12,16){\circle*{1.5}}
  \put(12,28){\circle*{1.5}}
  \put(24,10){\circle*{1.5}}
  \put(24,22){\circle*{1.5}}
  \put(36,16){\circle*{1.5}}
  \put(0,22.4){\line(2,-1){24}}
  \put(12,28.4){\line(2,-1){24}}
  \put(0,22.4){\line(2,1){12}}
  \put(12,16.4){\line(2,1){12}}
  \put(24,10.4){\line(2,1){12}}
  \put(0,21.6){\line(2,-1){24}}
  \put(12,27.6){\line(2,-1){24}}
  \put(0,21.6){\line(2,1){12}}
  \put(12,15.6){\line(2,1){12}}
  \put(24,9.6){\line(2,1){12}}
  \put(3.6,36){\circle*{1.5}}
  \put(0,22){\vector(1,4){3.4}}
  \put(12,28){\vector(-1,1){8}}
  \put(20.4,36){\circle*{1.5}}
  \put(24,22){\vector(-1,4){3.4}}
  \put(12,28){\vector(1,1){8}}
  \put(32.4,30){\circle*{1.5}}
  \put(36,16){\vector(-1,4){3.4}}
  \put(24,22){\vector(1,1){8}}
  \put(3.6,8){\circle*{1.5}}
  \put(3.6,8){\vector(-1,4){3.4}}
  \put(3.6,8){\vector(1,1){8}}
  \put(15.6,2){\circle*{1.5}}
  \put(15.6,2){\vector(-1,4){3.4}}
  \put(15.6,2){\vector(1,1){8}}
  \put(32.4,2){\circle*{1.5}}
  \put(32.4,2){\vector(1,4){3.4}}
  \put(32.4,2){\vector(-1,1){8}}
  \put(-8,21){$\hat u_{p-1}$}
  \put(37,15){$\hat u_p$}
  \put(0.5,6){$\hat a$}
  \put(12,0){$\hat a'$}
  \put(33.5,0){$\hat a''$}
  \put(1,36){$\hat c$}
  \put(21.5,36){$\hat c'$}
  \put(33.5,30){$\hat c''$}
  \put(11,30){$\hat x$}
  \put(10,11){$\hat y$}
  \put(24,25){$\hat x'$}
  \put(22.5,5.5){$\hat y'$}
  \put(10.5,21.5){$\rho_\tau$}
  \put(22,15){$\rho_{\tau'}$}
  \put(-15,3){in $\sigma(\Gamma)$:}
    \end{picture}}
\end{picture}
   \end{center}
\caption{\small On the left: an instance of $G^h(p)$. Here $T_h(p)$
consists of two black tiles $\tau,\tau'$ and seven white tiles. The
corresponding cycle $C(u_{p-1},u_p)$ contains all nonterminal vertices
(there is no inward pendant edge). On the right: the extended hammock
$\sigma(\bar\Gamma^h(p)$. Here the hammock $\sigma(\Gamma^h(p)$) is
indicated by double lines (the edges should be directed to the right),
and for a nonterminal vertex $\ast$ of $G_T$, we write $\hat\ast$ for
$\sigma(\ast)$.}
 \label{fig:hammock}
  \end{figure}
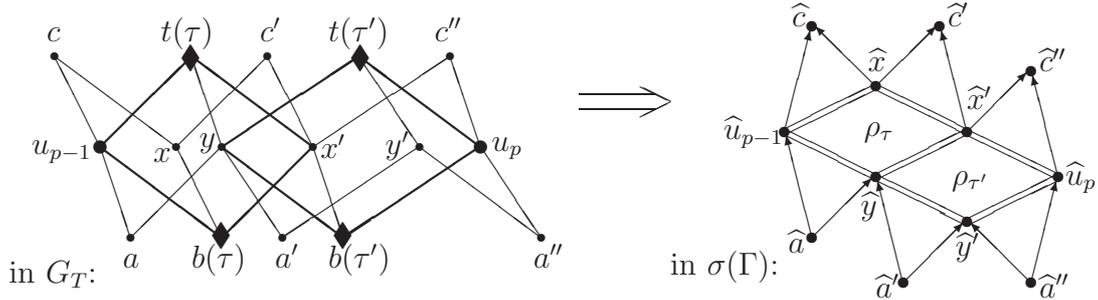

A further refinement shows that the pairwise intersections of cycles
$\zeta_\tau$ are poor.
  \begin{lemma} \label{lm:intersect}
For distinct $\tau,\tau'\in T^b_h(p)$, let
$\zeta_\tau\cap\zeta_{\tau'}\ne\emptyset$. Then the intersection of these
cycles is contained in the upper path of one of them and in the lower path of
the other, and it consists of a single vertex or a single edge.
  \end{lemma}
  \begin{proof}
~Suppose that $\gamma_\tau$ and $\gamma_{\tau'}$ have a common vertex $w$ and
this vertex is intermediate in both paths. Assuming that the label of the edge
$(b(\tau),w)$ is smaller than the label of the edge $(b(\tau'),w)$, take the
tile $\tilde\tau\in F_T(b(\tau))$ with $r(\tilde\tau)=w$ and the tile
$\tilde\tau'\in F_T(b(\tau'))$ with $\ell(\tilde\tau')=w$ (which exist since
$w\ne\ell(\tau)$ and $w\ne r(\tau')$). Then $\tilde\tau$ and $\tilde\tau'$
overlap, contrary to Corollary~\ref{cor:ord-mixed}(c). Similarly, $\beta_\tau$
and $\beta_{\tau'}$ cannot intersect at an intermediate vertex of both paths.

Now suppose that $\gamma_\tau\cap\beta_{\tau'}$ contains two vertices $w,w'$.
Then $w,w',b(\tau),t(\tau')$ are connected by the four edges $e:=(b(\tau),w)$,
$e':=(b(\tau),w')$, $f:=(w,t(\tau'))$, and $f':=(w',t(\tau'))$.
By~\refeq{4edges}, these edges are spanned by a tile $\tilde\tau$. We assert
that $\tilde\tau$ is a white tile in $T$ (whence $w,w'$ are connected by a
horizontal edge in $\Gamma$).

Indeed, if this is not so, then the edges $e,e'$ are not consecutive in
$E_T(b(\tau))$ and $f,f'$ are not consecutive in $E_T(t(\tau'))$. Note that at
least one of the edges $e,e'$, say, $e$, is white (for otherwise
$\tilde\tau=\tau$, implying that the black tiles $\tau,\tau'$ have a common
terminal vertex, namely, $t(\tau)=t(\tau')$). Since $f,f$ are not consecutive,
there is a (unique) white tile $\hat\tau\in F_T(t(\tau'))$ lying between $f$
and $f'$ and containing $f$ but not $f'$. Since $e,f'$ are parallel, the white
edge $e$ lies between $f$ and the edge of $\hat\tau$ connecting $b(\hat\tau)$
and $w$ (these three edges are incident to $w$). This leads to a contradiction
with Corollary~\ref{cor:ord-mixed}.

Thus, $\tilde\tau$ is a white tile in $T$, and now the lemma easily follows.
  \end{proof}


\subsection{Contractions on $\Gamma$}
\label{ssec:contr_Gamma}

Our final step in preparation to proving Proposition~\ref{pr:LZ-Gamma} is to
explain how the graph $\Gamma$ changes under the $n$- and 1-contraction
operations. We use terminology and notation from Subsection~\ref{ssec:s-c-e}. A
majority of our analysis is devoted to the $n$-contraction operation that
reduces a g-tiling $T$ on $Z=Z_n$ to the g-tiling $T':=T/n$ on $Z_{n-1}$.

It is more convenient to consider the reversed $n$-strip
$Q=(e_0,\tau_1,e_1,\ldots, \tau_r,e_r)$, i.e. $e_0$ is the $n$-edge
$z^r_nz^r_{n-1}$ on $rbd(Z)$ and $e_r$ is the $n$-edge $z^\ell_{n-1}z^\ell_n$
on $\ell bd(Z)$. Let $R_Q=(v_0,a_1,v_1,\ldots,a_r,v_r)$ be the right boundary,
and $L_Q=(v'_0,a'_1,v'_1,\ldots,a'_r,v'_r)$ the left boundary of $Q$, i.e.
$v'_0=z_0$, ~$v_r=z_n$, and $e_q=(v'_q,v_q)$ for each $q$.

Since $n$ is the maximal label, if an $n$-edge $e$ belongs to a tile $\tau$,
then $e$ is either $br(\tau)$ or $\ell t(\tau)$. This implies (in view of
$e_0=br(\tau_1)$) that for consecutive tiles $\tau_q,\tau_{q+1}$ in $Q$, one
holds: if both tiles are white then $e_q=\ell t(\tau_q)=br(\tau_{q+1})$; if
$\tau_q$ is black then $e_q=br(\tau_q)=br(\tau_{q+1})$; and if $\tau_{q+1}$ is
black then $e_q=\ell t(\tau_q)=\ell t(\tau_{q+1})$. So the height of
$\tau_{q+1}$ is greater by one than the height of $\tau_q$ if both tiles are
white, and the heights are equal otherwise. In particular, the tile height is
weakly increasing along $Q$ and grows from 1 to $n-1$. For $h=1,\ldots,n-1$,
let $Q^h=(e_{d(h)-1},\tau_{d(h)},e_{d(h)}, \ldots,\tau_{f(h)},e_{f(h)})$ be the
maximal part of $Q$ in which all tiles are of height $h$; we call it $h$-th
\emph{fragment} of $Q$.

Recall that from the viewpoint of $D_T$, the $n$-contraction operation acts as
follows. The interior of $\sigma(Q)$ is removed from $D_T$, forming two closed
simply connected regions $D^r,D^\ell$, where $D^r$ (the ``right'' region)
contains $\sigma(R_Q)$ and the rest of $\sigma(rbd(Z))$, and $D^\ell$ (the
``left'' region) contains $\sigma(L_Q)$ and the rest of $\sigma(\ell bd(Z))$.
The region $D^r$ is shifted by $-\eps_n$ (where $\eps_n$ is $n$-th unit base
vector in $\Rset^{[n]}$) and the path $\sigma(R_Q)-\eps_n$ merges with
$\sigma(L_Q)$. From the viewpoint of $Z$, the tiles occurring in $Q$ vanish and
the tiles $\tau\in T$ with $\sigma(\tau)\subset D^r$ are shifted by $-\xi_n$;
in other words, each vertex $X$ (regarded as a set) containing the element $n$
turns into the vertex $X-n$ of the resulting tiling $T'$ on $Z_{n-1}$. (Recall
that $X$ contains $n$ if and only if $\sigma(X)$ is in $D^r$.) Each vertex
$v_q$ of $R_Q$ merges with the vertex $v'_q$ of $L_Q$. The path $L_Q$ no longer
contains terminal vertices (so all edges in it are now fully white) and becomes
an $n$-legal path for $T'$. Any zigzag subpath in this path goes from left to
right, and
   \begin{numitem1}
for $h=1,\ldots,n-2$, ~$v'_{f(h)=d(h+1)-1}$ is the critical vertex of
$L_Q$ in level $h$ for $T'$.
   \label{eq:critLQ}
   \end{numitem1}

Consider $h$-th fragment $Q^h=(e_{d-1},\tau_d,e_d,\ldots,\tau_f,e_f)$ (letting
$d:=d(h)$ and $f:=f(h)$). It produces $(f-d)/2+1$ horizontal and four ascending
edges in $\Gamma=\Gamma_T$ (note that $f-d$ is even). More precisely, each tile
$\tau_q$ with $q-d$ even is white and its diagonal makes the horizontal edge
$g_q:=(v'_q,v_{q-1})$ in $\Gamma$. Also $\tau_d$ contributes the ascending
edges $e_{d-1}=br(\tau_d)=(v'_{d-1},v_{d-1})$ and
$a'_d=b\ell(\tau_q)=(v'_{d-1},v'_d)$. In its turn, $\tau_f$ contributes the
ascending edges $e_f=\ell t(\tau_f)=(v'_f,v_f)$ and
$a_f=rt(\tau_f)=(v_{f-1},v_f)$ to $\Gamma$. Let $\Escr^h$ be the set of edges
in $\Gamma$ produced by $Q^h$. Then $\Escr^h\cap \Escr^{h+1}=\{e_f\}$
~($=\{e_{d(h+1)-1}\}$).

Under the $n$-contraction operation, $\Gamma$ is transformed into the graph
$\Gamma':=\Gamma_{T'}$. The transformation concerns only the sets $\Escr^h$ and
is obvious: all horizontal edges of $\Escr^h$ disappear (as all tiles in $Q^h$
vanish) and the four ascending edges are replaced by (the edges of) the subpath
$L_{Q^h}$ of $L_Q$ from $v'_{d-1}$ to $v'_f$, in which all edges connect levels
$h-1$ and $h$ (using indices as above). In particular, when $d=f$ (i.e. when
$Q^h$ has only one (white) tile), the five edges of $\Escr^h$ shrink into one
edge $a'_d=(v'_{d-1},v'_d)$.

When $\Delta:=f-d>0$, the transformation needs to be examined more carefully.
The $\Delta/2+1$ white tiles and the $\Delta/2$ black tiles in $Q^h$ alternate.
The horizontal edges $g_d,g_{d+2},\ldots,g_f$ in $\Escr^h$ belong to the same
nontrivial hammock in $\Gamma^h$, say, $\Gamma^h(p)$. More precisely, for
$q=d+1,d+3,\ldots,f-1$, the edges $g_{q-1}=(v'_{q-1},v_{q-2})$ and
$g_{q+1}=(v'_{q+1},v_q)$ are contained in the cycle $\zeta_{\tau_q}$ related to
the black tile $\tau_q$. Also $\tau_{q-1}$ is the \emph{leftmost} white tile in
$F_T(t(\tau_q))$ and $\tau_{q+1}$ is the \emph{rightmost} white tile in
$F_T(b(\tau_q))$. Therefore, $g_{q-1}$ is the \emph{first} edge in the lower
path $\beta_{\tau_q}$ and $g_{q+1}$ is the \emph{last} edge in the upper path
$\gamma_{\tau_q}$ in $\zeta_{\tau_q}$. For a similar reason, unless $q=f-1$,
~$g_{q+1}$ is simultaneously the first edge in $\beta_{\tau_{q+2}}$.

This and Lemma~\ref{lm:intersect} imply that if we take the union of cycles
$\zeta_{\tau_q}$ for $q=d+1,d+3,\ldots,f-1$ and delete from it the horizontal
edges in $\Escr^h$, then we obtain two directed horizontal paths in level $h$
of $\Gamma$: path $\Pscr^h_1$ from $v'_d$ to $v'_f$ which passes the vertices
$v'_d,v'_{d+2},\ldots,v'_f$ in this order, and path $\Pscr^h_2$ from $v_{d-1}$
to $v_{f-1}$ which passes the vertices $v_{d-1},v_{d+1},\ldots,v_{f-1}$ in this
order (these paths may contain other vertices as well). When $f=d$, each of
these paths consists of a single vertex.

In the new graph $\Gamma'$, the path $\Pscr^h_1$ preserves and continues to be
a horizontal path in level $h$, whereas $\Pscr^h_2$ is shifted by $-\xi_n$ and
turns into the directed horizontal path, denoted by $\Pscr'^h_2$, that passes
the vertices $v'_{d-1},v'_{d+1},\ldots,v'_{f-1}$ in level $h-1$. These paths
are connected in $\Gamma'$ by the (zigzag) path
$\Zscr^h:=(v'_{d-1},v'_d,v'_{d+1},\ldots,v'_f)$ whose edges connect levels
$h-1$ and $h$. (Under the transformation, the hammock $\Gamma^h(p)$ becomes
split into two hammocks, one (in level $h$) containing the path $\Pscr^h_1$,
and the other (in level $h-1$) containing the image $\Pscr'^h_2$ of
$\Pscr^h_2$.) In view of~\refeq{critLQ},
  \begin{numitem1}
the last vertex $v'_f$ of $\Pscr^h_1$ (which is simultaneously the
first vertex of $\Pscr'^{(n+1)}_2$ when $h<n-1$) and the first vertex
$v'_{d-1}$ of $\Pscr'^h_2$ (which is simultaneously the last vertex of
$\Pscr^{n-1}_1$ when $h>1$) are critical for $T'$ in levels $h$ and
$h-1$, respectively.
   \label{eq:new_critic_n}
   \end{numitem1}

The transformation of $\sigma(\Gamma)$ into $\sigma(\Gamma')$ within the
fragment $Q^h$ with $d<f$ is illustrated in the picture; here for brevity we
write $\rho_q$ instead of $\rho_{\tau_q}$, and omit $\sigma$ in notation for
vertices, edges and paths.

 \begin{center}
  \unitlength=1mm
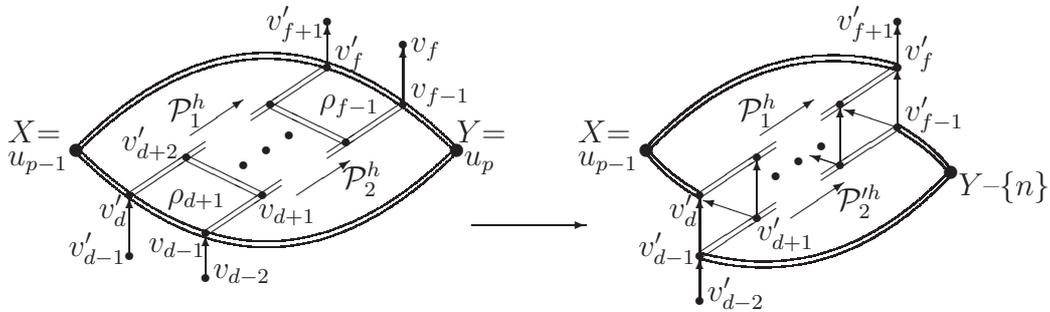
  \begin{picture}(145,38)
     \put(0,0){\begin{picture}(70,38)
  \put(10,20){\circle*{1.8}}
  \put(60,20){\circle*{1.8}}
  \qbezier(10,20.4)(35,-4.6)(60,20.4)
  \qbezier(10,19.6)(35,-5.4)(60,19.6)
  \qbezier(10,20.4)(35,45.4)(60,20.4)
  \qbezier(10,19.6)(35,44.6)(60,19.6)
  \put(53,26){\circle*{1.2}}
  \put(43,31){\circle*{1.2}}
  \put(17,14){\circle*{1.2}}
  \put(27,9){\circle*{1.2}}
  \put(45.5,21){\circle*{1.2}}
  \put(35.5,26){\circle*{1.2}}
  \put(24.5,19){\circle*{1.2}}
  \put(34.5,14){\circle*{1.2}}
  \put(35,20){\circle*{1}}
  \put(38,22){\circle*{1}}
  \put(32,18){\circle*{1}}
  \put(53,26.4){\line(-3,-2){10}}
  \put(43,31.4){\line(-3,-2){10}}
  \put(17,14.4){\line(3,2){10}}
  \put(27,9.4){\line(3,2){10}}
  \put(45.5,21.4){\line(-2,1){10}}
  \put(24.5,19.4){\line(2,-1){10}}
  \put(53,25.6){\line(-3,-2){10}}
  \put(43,30.6){\line(-3,-2){10}}
  \put(17,13.6){\line(3,2){10}}
  \put(27,8.6){\line(3,2){10}}
  \put(45.5,20.6){\line(-2,1){10}}
  \put(24.5,18.6){\line(2,-1){10}}
  \put(45,15){$\Pscr^h_2$}
  \put(25,22){\vector(3,2){7}}
  \put(22,24){$\Pscr^h_1$}
  \put(39,14){\vector(3,2){7}}
  \put(16,20){$v'_{d+2}$}
  \put(53,26){\vector(0,1){7.7}}
  \put(43,31){\vector(0,1){5.7}}
  \put(53,34){\circle*{1.2}}
  \put(43,37){\circle*{1.2}}
  \put(17,6){\vector(0,1){7.7}}
  \put(27,3){\vector(0,1){5.7}}
  \put(17,6){\circle*{1.2}}
  \put(27,3){\circle*{1.2}}
  \put(42,25.5){$\rho_{f-1}$}
  \put(22,13){$\rho_{d+1}$}
  \put(1,21){$X$=}
  \put(1,18){$u_{p-1}$}
  \put(60,21){$Y$=}
  \put(61,18){$u_p$}
  \put(54,27){$v_{f-1}$}
  \put(44,32){$v'_{f}$}
  \put(13,11){$v'_{d}$}
  \put(19.5,6.5){$v_{d-1}$}
  \put(34,11){$v_{d+1}$}
  \put(54,33){$v_{f}$}
  \put(35,36){$v'_{f+1}$}
  \put(9,6){$v'_{d-1}$}
  \put(28,3){$v_{d-2}$}
  \put(62,10){\vector(1,0){15}}
    \end{picture}}
     \put(75,0){\begin{picture}(70,38)

  \put(10,20){\circle*{1.8}}
  \qbezier(10,20.4)(12,17.4)(17,14.4)
  \qbezier(10,19.6)(12,16.6)(17,13.6)

  \qbezier(10,20.4)(25,36.4)(43,31.4)
  \qbezier(10,19.6)(25,35.6)(43,30.6)
  \put(17,14){\circle*{1.2}}
  \put(35.5,26){\circle*{1.2}}
  \put(24.5,19){\circle*{1.2}}
  \put(43,31){\circle*{1.2}}
  \put(43,31.4){\line(-3,-2){10}}
  \put(17,14.4){\line(3,2){10}}
  \put(43,30.6){\line(-3,-2){10}}
  \put(17,13.6){\line(3,2){10}}
  \put(25,22){\vector(3,2){7}}
  \put(22,24){$\Pscr^h_1$}
  \put(30,18.5){\circle*{1}}
  \put(33,20.5){\circle*{1}}
  \put(27,16.5){\circle*{1}}
  \put(43,31){\vector(0,1){5.7}}
  \put(43,37){\circle*{1.2}}
  \put(17,6){\vector(0,1){7.7}}
  \put(17,6){\circle*{1.2}}
  \put(1,21){$X$=}
  \put(1,18){$u_{p-1}$}
  \put(44,32){$v'_{f}$}
  \put(13,11){$v'_{d}$}
  \put(35,36){$v'_{f+1}$}
  \put(9,6){$v'_{d-1}$}
  \put(50,17){\circle*{1.8}}
  \put(43,23.4){\line(-3,-2){10}}
  \qbezier(17,6.4)(35,1.4)(50,17.4)
 \qbezier(17,5.6)(35,0.6)(50,16.6)
  \qbezier(43,23.4)(48,20.4)(50,17.4)
 \qbezier(43,22.6)(48,19.6)(50,16.6)
  \put(43,23){\circle*{1.2}}
  \put(35.5,18){\circle*{1.2}}
  \put(24.5,11){\circle*{1.2}}
  \put(17,6.4){\line(3,2){10}}
  \put(43,22.6){\line(-3,-2){10}}
  \put(17,5.6){\line(3,2){10}}
  \put(35,12){$\Pscr'^h_2$}
  \put(29,11){\vector(3,2){7}}
  \put(43,23){\vector(0,1){7.7}}
  \put(43,31){\circle*{1.2}}
  \put(17,0){\vector(0,1){5.7}}
  \put(17,0){\circle*{1.2}}
  \put(51,14){$Y$--$\{n\}$}
  \put(44,24){$v'_{f-1}$}
  \put(18,0){$v'_{d-2}$}
  \put(24.5,8){$v'_{d+1}$}
  \put(24.5,11){\vector(0,1){7.7}}
  \put(35.5,18){\vector(0,1){7.7}}
  \put(24.5,11){\vector(-3,1){7}}
  \put(43,23){\vector(-3,1){7}}
  \put(35.5,18){\vector(-3,1){4}}
    \end{picture}}
\end{picture}
   \end{center}

Notice that the paths $\Pscr^{h-1}_1$ and $\Pscr^h_1$ are connected in $\Gamma$
by one ascending edge (namely, $a'_{d(h)}$) going from the end
$v'_{f(h-1)=d(h)-1}$ of the former path to the beginning $v'_{d(h)}$ of the
latter one; we call it the \emph{bridge} between these paths and denote by
$b'_h$. Similarly, there is an ascending edge (namely, $a_{d(h)-1}$) going from
the end $v_{d(h)-2}$ of $\Pscr^{h-1}_2$ to the beginning $v_{d(h)-1}$ of
$\Pscr^h_2$, the bridge between these paths, denoted by $b_h$. Under the
transformation, $b_h$ is shifted and becomes the last edge
$(v'_{d(h)-2},v'_{d(h)})$ of the (zigzag) path $\Zscr^{h-1}$ and the beginning
of $\Pscr'^h_2$ merges with the end of $\Pscr^{h-1}_1$.
(Cf.~\refeq{new_critic_n}.)

Thus, concatenating the paths $\Pscr_1^1,\ldots,\Pscr_1^{n-1}$ and the bridges
$b'_1,b'_2,\ldots,b'_{n-1}$ (where $b'_1:=a'_1$), we obtain a directed path
from $z_0$ to $z^\ell_{n-1}$ in both $\Gamma$ and $\Gamma'$, denoted by
$\Pscr_1$. Accordingly, we construct directed paths $\Pscr_2$ (in $\Gamma$) and
$\Pscr'_2$ (in $\Gamma'$) by concatenating, in a due way, the paths $\Pscr^h_2$
with the bridges $b_h$, and the paths $\Pscr'^h_2$ with the shifts of these
bridges, respectively.

Let $\Gamma_1$ and $\Gamma_2$ be the subgraphs of $\Gamma$ whose images
by $\sigma$ lie in the regions $D^\ell$ and $D^r$ of $D_T$,
respectively. Let $\Gamma'_2$ be the subgraph of $\Gamma'$ whose image
by $\sigma$ lies in $D^r-\eps_n$. We observe that:
   \begin{numitem1}
 \begin{itemize}
\item[(i)] the common vertices of $\Pscr_1$ and $\Pscr'_2$ are exactly the
critical vertices (indicated in~\refeq{critLQ}) of the path $L_Q$ in $G_{T'}$;
\item[(ii)] if a directed path $P$ in $\Gamma'$ goes from a vertex $x$ of
$\Gamma_1$ to a vertex $y$ of $\Gamma'_2$, then $P$ contains a critical vertex
$v$ of $L_Q$; moreover, there exist a directed path $P'$ from $x$ to $v$ in
$\Gamma_1$ and a directed path $P''$ from $v$ to $y$ in $\Gamma'_2$.
  \end{itemize}
  \label{eq:separ_paths}
  \end{numitem1}
Here (ii) follows from the facts that both $\Pscr_1,\Pscr'_2$ are directed
paths and that all edges not in $\Pscr_1\cap \Pscr'_2$ that connect $\Gamma_1$
and $\Gamma'_2$ go from vertices of $\Pscr'_2$ to vertices of $\Pscr_1$ (as
they are ascending edges occurring in zigzag paths $\Zscr^h$). \medskip

The 1-contraction operation acts symmetrically, in a sense; below we mainly
describe the moments where there are differences from the $n$-contraction case.

Let $Q=(e_0,\tau_1,e_1,\ldots,\tau_r,e_r)$ be the 1-strip in $T$,
~$R_Q=(v_0,a_1,v_1,\ldots,a_r,v_r)$ the right boundary of $Q$, and
$L_Q=(v'_0,a'_1,v'_1,\ldots,a'_r,v'_r)$ the left boundary of $Q$. So
$e_0=(v_0,v'_0)=z_0 z^\ell_1$ and $e_r=(v_r,v'_r)=z^r_{n-1}z_n$.

Since label 1 is minimal, if a 1-edge $e$ belongs to a tile $\tau\in T$, then
either $e=b\ell(\tau)$ or $e=rt(\tau)$. For consecutive tiles
$\tau_q,\tau_{q+1}$ in $Q$: the height of $\tau_{q+1}$ is greater by one than
the height of $\tau_q$ if both tiles are white, and the heights are equal if
one of these tiles is black. Like the previous case, for $h=1,\ldots,n$, define
$h$-th fragment of $Q$ to be the maximal part
$Q^h=(e_{d(h)-1},\tau_{d(h)},e_{d(h)},\ldots,\tau_{f(h)},e_{f(h)})$ of $Q$ in
which all tiles are of height $h$.

The fragment $Q^h$ produces $(f-d)/2$ horizontal and four ascending
edges in $\Gamma$, where $d:=d(h)$ and $f:=f(h)$. Each tile $\tau_q$
with $q-d$ even is white and produces the horizontal edge
$g_q:=(v'_{q-1},v_q)$ in $\Gamma^h$. Also $\tau_d$ contributes the
ascending edges $e_{d-1}=b\ell(\tau_d)=(v_{d-1},v'_{d-1})$ and
$a_d=br(\tau_q)=(v_{d-1},v_d)$, and $\tau_f$ contributes the ascending
edges $e_f=rt(\tau_f)=(v_f,v'_f)$ and $a'_f=\ell
t(\tau_f)=(v'_{f-1},v'_f)$. Let $\Escr^h$ be the set of edges in
$\Gamma$ produced by $Q^h$.

Under the 1-contraction operation, $\Gamma$ is transformed into
$\tilde\Gamma':=\Gamma_{T/1}$ as follows: for each $h$, the horizontal edges of
$\Escr^h$ disappear and the four ascending edges are replaced by the subpath
$R^h_Q$ of $R_Q$ from $v_{d-1}$ to $v_f$ (using indices as above) in which all
edges connect levels $h-1$ and $h$. When $d=f$, ~$\Escr^h$ shrinks into one
edge $a_d=(v_{d-1},v_d)$.

When $\Delta:=f-d>0$, the horizontal edges $g_d,g_{d+2},\ldots,g_f$ in
$\Escr^h$ belong to the same nontrivial hammock, $\Gamma^h(p)$ say. On the
other hand, for $q=d+1,d+3,\ldots,f-1$, the edges $g_{q-1}$ and $g_{q+1}$
belong to the cycle $\zeta_{\tau_q}$ related to the black tile $\tau_q$. Also
$\tau_{q-1}$ is the \emph{rightmost} white tile in $F_T(t(\tau_q))$ and
$\tau_{q+1}$ is the \emph{leftmost} white tile in $F_T(b(\tau_q))$. Therefore,
$g_{q-1}$ is the \emph{last} edge in the lower path $\beta_{\tau_q}$ and
$g_{q+1}$ is the \emph{first} edge in the upper path $\gamma_{\tau_q}$ in
$\zeta_{\tau_q}$. Unless $q=f-1$, ~$g_{q+1}$ is simultaneously the last edge in
$\beta_{\tau_{q+2}}$. In view of Lemma~\ref{lm:intersect}, this implies that if
we take the union of cycles $\zeta_q$ and then delete from it the horizontal
edges of $\Escr^h$, then we obtain two horizontal paths in $\Gamma^h(p)$: path
$\tilde\Pscr^h_1$ from $v'_{f-1}$ to $v'_{d-1}$ which passes the vertices
$v'_{f-1},v'_{f-3},\ldots,v'_{d-1}$ in this order, and path $\tilde\Pscr^h_2$
from $v_f$ to $v_d$ which passes the vertices $v_f,v_{f-2},\ldots,v_d$ in this
order. (So, both paths are directed by decreasing the vertex indices, in
contrast to the direction of the corresponding paths in the $n$-contraction
case.)

For our further purposes, it will be sufficient to examine the transformation
of $\Gamma$ only within its part related to a single fragment $Q^h$. In the new
graph $\tilde\Gamma'$, the path $\tilde\Pscr^h_2$ preserves and continues to be
a horizontal path in level $h$, whereas $\tilde\Pscr^h_1$ is shifted by
--$\xi_1$ and turns into a horizontal path in level $h-1$, denoted by
$\tilde\Pscr'^h_1$, which passes the vertices $v_{f-1},v_{f-3},\ldots,v_{d-1}$.
These paths are connected in $\tilde\Gamma'$ by the zigzag path
$\tilde\Zscr^h=(v_f,v_{f-1},\ldots, v_d,v_{d-1})$ whose edges are ascending and
connect levels $h-1$ and $h$. The picture illustrates the transformation
$\sigma(\Gamma)\mapsto\sigma(\tilde\Gamma')$ within the fragment $Q^h$.

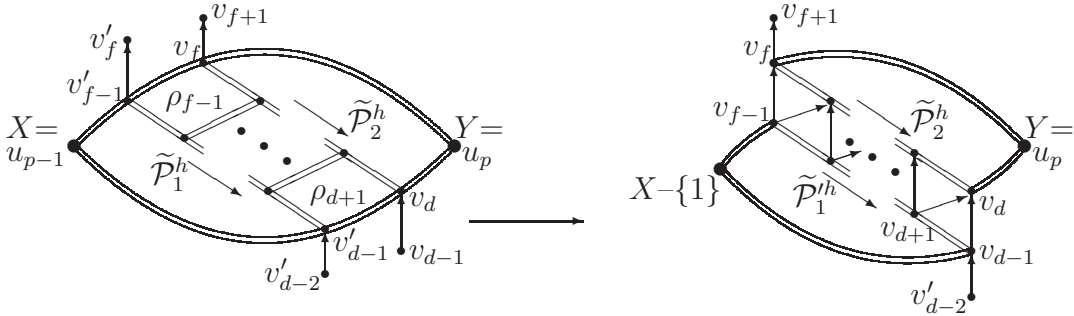
 \begin{center}
  \unitlength=1mm
  \begin{picture}(145,38)
     \put(0,0){\begin{picture}(70,38)
  \put(10,20){\circle*{1.8}}
  \put(60,20){\circle*{1.8}}
  \qbezier(10,20.4)(35,-4.6)(60,20.4)
  \qbezier(10,19.6)(35,-5.4)(60,19.6)
  \qbezier(10,20.4)(35,45.4)(60,20.4)
  \qbezier(10,19.6)(35,44.6)(60,19.6)
  \put(17,26){\circle*{1.2}}
  \put(27,31){\circle*{1.2}}
  \put(53,14){\circle*{1.2}}
  \put(43,9){\circle*{1.2}}
  \put(24.5,21){\circle*{1.2}}
  \put(34.5,26){\circle*{1.2}}
  \put(45.5,19){\circle*{1.2}}
  \put(35.5,14){\circle*{1.2}}
  \put(35,20){\circle*{1}}
  \put(32,22){\circle*{1}}
  \put(38,18){\circle*{1}}
  \put(17,26.4){\line(3,-2){10}}
  \put(27,31.4){\line(3,-2){10}}
  \put(53,14.4){\line(-3,2){10}}
  \put(43,9.4){\line(-3,2){10}}
  \put(24.5,21.4){\line(2,1){10}}
  \put(45.5,19.4){\line(-2,-1){10}}
  \put(17,25.6){\line(3,-2){10}}
  \put(27,30.6){\line(3,-2){10}}
  \put(53,13.6){\line(-3,2){10}}
  \put(43,8.6){\line(-3,2){10}}
  \put(24.5,20.6){\line(2,1){10}}
  \put(45.5,18.6){\line(-2,-1){10}}
  \put(20,15){$\tilde\Pscr^h_1$}
  \put(25,18){\vector(3,-2){7}}
  \put(46,22){$\tilde\Pscr^h_2$}
  \put(39,26){\vector(3,-2){7}}
  \put(17,26){\vector(0,1){7.7}}
  \put(27,31){\vector(0,1){5.7}}
  \put(17,34){\circle*{1.2}}
  \put(27,37){\circle*{1.2}}
  \put(53,6){\vector(0,1){7.7}}
  \put(43,3){\vector(0,1){5.7}}
  \put(53,6){\circle*{1.2}}
  \put(43,3){\circle*{1.2}}
  \put(22,25.5){$\rho_{f-1}$}
  \put(41,13){$\rho_{d+1}$}
  \put(1,21){$X$=}
  \put(1,18){$u_{p-1}$}
  \put(60,21){$Y$=}
  \put(61,18){$u_p$}
  \put(9,27){$v'_{f-1}$}
  \put(23,32){$v_{f}$}
  \put(54,12){$v_{d}$}
  \put(44,6){$v'_{d-1}$}
  \put(12,33){$v'_{f}$}
  \put(28,37){$v_{f+1}$}
  \put(54,5){$v_{d-1}$}
  \put(35,2){$v'_{d-2}$}
  \put(62,10){\vector(1,0){15}}
    \end{picture}}
     \put(75,0){\begin{picture}(45,38)

  \put(60,20){\circle*{1.8}}
  \qbezier(60,20.4)(58,17.4)(53,14.4)
  \qbezier(60,19.6)(58,16.6)(53,13.6)

  \qbezier(60,20.4)(45,36.4)(27,31.4)
  \qbezier(60,19.6)(45,35.6)(27,30.6)
  \put(53,14){\circle*{1.2}}
  \put(34.5,26){\circle*{1.2}}
  \put(45.5,19){\circle*{1.2}}
  \put(27,31){\circle*{1.2}}
  \put(27,31.4){\line(3,-2){10}}
  \put(53,14.4){\line(-3,2){10}}
  \put(27,30.6){\line(3,-2){10}}
  \put(53,13.6){\line(-3,2){10}}
  \put(38,26){\vector(3,-2){7}}
  \put(45,22){$\tilde\Pscr^h_2$}
  \put(40,18.5){\circle*{1}}
  \put(37,20.5){\circle*{1}}
  \put(43,16.5){\circle*{1}}
  \put(27,31){\vector(0,1){5.7}}
  \put(27,37){\circle*{1.2}}
  \put(53,6){\vector(0,1){7.7}}
  \put(53,6){\circle*{1.2}}
  \put(8,13){$X$--$\{1\}$}
  \put(23,32){$v_{f}$}
  \put(54,11.5){$v_{d}$}
  \put(28,37){$v_{f+1}$}
  \put(54,5){$v_{d-1}$}
  \put(20,17){\circle*{1.8}}
  \put(27,23.4){\line(3,-2){10}}
  \qbezier(53,6.4)(35,1.4)(20,17.4)
 \qbezier(53,5.6)(35,0.6)(20,16.6)
  \qbezier(27,23.4)(22,20.4)(20,17.4)
 \qbezier(27,22.6)(22,19.6)(20,16.6)
  \put(27,23){\circle*{1.2}}
  \put(34.5,18){\circle*{1.2}}
  \put(45.5,11){\circle*{1.2}}
  \put(53,6.4){\line(-3,2){10}}
  \put(27,22.6){\line(3,-2){10}}
  \put(53,5.6){\line(-3,2){10}}
  \put(29.5,12){$\tilde\Pscr'^h_1$}
  \put(33.5,16.5){\vector(3,-2){7}}
  \put(27,23){\vector(0,1){7.7}}
  \put(27,31){\circle*{1.2}}
  \put(53,0){\vector(0,1){5.7}}
  \put(53,0){\circle*{1.2}}
  \put(60,21){$Y$=}
  \put(61,18){$u_p$}
  \put(19,24){$v_{f-1}$}
  \put(45,-1){$v'_{d-2}$}
  \put(41,8){$v_{d+1}$}
  \put(45.5,11){\vector(0,1){7.7}}
  \put(34.5,18){\vector(0,1){7.7}}
  \put(45.5,11){\vector(3,1){7}}
  \put(27,23){\vector(3,1){7}}
  \put(34.5,18){\vector(3,1){4}}
    \end{picture}}
\end{picture}
   \end{center}

One can see that the first vertex $v_f$ of $\tilde\Pscr^h_2$ and the
last vertex $v_{d-1}$ of $\tilde\Pscr'^{h}_1$ are critical vertices for
$\tilde T':=T/1$ in levels $h$ and $h-1$, respectively (and they are
the only critical vertices for $\tilde T'$ occurring in these paths).
Under the transformation, the hammock $\Gamma^h(p)$ (concerning $Q^h$)
becomes split into two hammocks in $\tilde\Gamma'$: hammock $H_1$ in
level $h-1$ that contains the path $\tilde\Pscr'^h_1$, and hammock
$H_2$ in level $h$ that contains $\tilde\Pscr^h_2$. Comparing the part
of $\Gamma^h(p)$ between $\tilde\Pscr^h_1$ and $\tilde\Pscr^h_2$ with
the zigzag path $\Zscr^h$, one can conclude that
  \begin{numitem1}
if a vertex $x$ of $H_1$ and a vertex $y$ of $H_2$ are connected in
$\tilde\Gamma'$ by a directed path from $x$ to $y$, then there exists a
directed path from $x+\xi_1$ to $y$ in $\Gamma$.
  \label{eq:H1H2}
  \end{numitem1}


\subsection{Proof of ``$\prec^\ast$ implies $\prec_\Gamma$''}
\label{ssec:proof_impl}

Based on the above explanations, we are now ready to prove
Proposition~\ref{pr:LZ-Gamma}. We use induction on the number of edge labels of
a g-tiling and apply the $n$- and 1-contraction operations.

Let $A,B\in\mathfrak{S}_T$, $A\ne B$ and $A\prec^\ast B$. We have to
show the existence of a directed path from the vertex $A$ to the vertex
$B$ in the graph $\Gamma=\Gamma_T$.

First we consider the case $|A|<|B|$. Let $A':=A-\{n\}$ and $B':=B-\{n\}$. Then
$A',B'\in\mathfrak{S}_{T'}$, where $T'$ is the $n$-contraction $T/n$ of $T$.
Also $|A'|\le|B'|$. Therefore, $A\prec^\ast B$ implies $A'\prec^\ast B'$, and
by induction the graph $\Gamma':=\Gamma_{T'}$ contains a directed path $P$ from
$A'$ to $B'$. Note that $n\in A$ would imply $n\in B$. So either $n\not\in
A,B$, or $n\in A,B$, or $n\not\in A$ and $n\in B$. Consider these cases,
keeping notation from part III.
\smallskip

\emph{Case 1}: $n\not\in A,B$. Then $A'=A$, $B'=B$, and both $A',B'$
belong to the graph $\Gamma_1$. Since the path $\Pscr_1$ in the
boundary of $\Gamma_1$ is directed, $P$ as above can be chosen so as to
be entirely contained in $\Gamma_1$. (For if $P$ meets $\Pscr_1$, take
the first and last vertices of $P$ that occur in $\Pscr_1$, say, $x,y$
(respectively), and replace in $P$ its subpath from $x$ to $y$ by the
subpath of $\Pscr_1$ connecting $x$ and $y$, which must be directed
from $x$ to $y$ since $\Gamma'$ is acyclic.) Then $P$ is a directed
path from $A$ to $B$ in $\Gamma$, as required.
\smallskip

\emph{Case 2}: $n\in A,B$. Then both $A',B'$ belong to the graph
$\Gamma'_2$. Like the previous case, one may assume that $P$ is
entirely contained in $\Gamma'_2$. Since $\Gamma'_2+\xi_n$ is a
subgraph of $\Gamma$, ~$P+\xi_n$ is the desired path from $A$ to $B$ in
$\Gamma$.
\smallskip

\emph{Case 3}: $n\not\in A$ and $n\in B$. Then $A'=A$ is in $\Gamma_1$
and $B'$ is in $\Gamma'_2$. By~\refeq{separ_paths}, there exist a
directed path $P'$ from $A$ to $v$ in $\Gamma_1$ and a directed path
$P''$ from $v$ to $B'$ in $\Gamma'_2$, where $v$ is a critical vertex
$v'_{f(h)}$ in $\Pscr_1\cap\Pscr'_2$. Concatenating $P'$, the (fully
white) ascending edge $e_{f(h)}=(v'_{f(h)},v_{f(h)})$, and the path
$P''+\xi_n$ (going from $v_{f(h)}$ to $B'n=B$), we obtain a directed
path from $A$ to $B$ in $\Gamma$, as required. \medskip

Now consider the case $|A|=|B|=:h$. Note that the reduction by label $n$ as
above does not work when the element $n$ is contained in $B$ but not in $A$
(since in this case $|B-\{n\}|$ becomes less than $|A-\{n\}|$ and we cannot
apply induction). Nevertheless, we can use the 1-contraction operation (which,
in its turn, would not fit to treat the case $|A|<|B|$, since the concatenation
of the paths $\tilde\Pscr^h_2$ and the corresponding bridges is not a directed
path). Let $A':=A-\{1\}$ and $B':=B-\{1\}$; then $|A'|\le |B'|$ (since $1\in B$
would imply $1\in A$, in view of $A\lessdot B$). The vertices $A,B$ belong to
the horizontal subgraph $\Gamma^h$ of $\Gamma$; suppose they occur in $p$-th
and $p'$-th hammocks of $\Gamma^h$, respectively. The existence of a directed
path from $A$ to $B$ is not seen immediately only when $p=p'$ and, moreover,
when the 1-strip for $T$ ``splits'' the hammock $\Gamma^h(p)$. (Note that
$p>p'$ is impossible; otherwise there would exist a directed path from $B$ to
$A$, implying $B\lessdot A$.) Let $H_1,H_2$ be the hammocks in $\tilde\Gamma'$
created from $\Gamma^h(p)$ by the 1-contraction operation as described above.
If both $A',B'$ belong to the same $H_i$, then the existence (by induction) of
a directed path from $A'$ to $B'$ in $\tilde\Gamma'$ (and therefore, in $H_i$)
immediately yields the result. Let $A'\in H_1$ and $B'\in H_2$ (the case $A'\in
H_2$ and $B'\in H_1$ is impossible). Then the existence of a directed path from
$A$ to $B$ in $\Gamma$ follows from~\refeq{H1H2}.
\smallskip

Thus, $A\prec_\Gamma B$ is valid in all cases, as required. This completes the
proof of Proposition~\ref{pr:LZ-Gamma}, yielding Theorems~\ref{tm:2posets}
and~\ref{tm:prec-lat} and completing the proof of Theorem~B.

\section{Additional results and a generalization}  \label{sec:concl}

In this concluding section we gather in an additional harvest from results and
methods described in previous sections; in particular, we generalize Theorem~A
to the case of two permutations. Sometimes the description below will be given
in a sketched form and we leave the details to the reader.

We start with associating to a permutation $\omega$ on $[n]$ the directed path
$P_\omega$ on the zonogon $Z_n$ in which the vertices are the points
$v^i_\omega:=\sum(\xi_j\colon j\in\omega^{-1}[i])$, $i=0,\ldots,n$, and the
edges are the directed line segments $e^i_\omega$ from $v^{i-1}_\omega$ to
$v^i_\omega$. So $P_\omega$ begins at $v^0_\omega=z_0$, ends at
$v^n_\omega=z_n$, and each edge $e^i_\omega$ is a parallel translation of the
vector $\xi_{\omega^{-1}(i)}$. Also a vertex $v^i_\omega$ represents $i$-th
ideal $I_\omega^i=\omega^{-1}[i]$ for $\omega$ (cf. Section~\ref{sec:omega}).
Note that if the spectrum of a g-tiling $T$ on $Z_n$ contains all sets
$I_\omega^i$, then the graph $G_T$ contains the path $P_\omega$, in view of
Proposition~\ref{pr:4edges}(i). When $\omega$ is the longest permutation
$\omega_0$ on $[n]$, ~$P_\omega$ becomes the right boundary $rbd(Z_n)$ of
$Z_n$. When $\omega$ is the identical permutation, denoted as $id$, ~$P_\omega$
becomes the left boundary $\ell bd(Z_n)$.

Consider two permutations $\omega',\omega$ on $[n]$ and assume that the pair
$(\omega',\omega)$ satisfies the condition:
   \begin{numitem1}
for any $i,j\in[n]$, either $I_{\omega'}^i\lessdot I_\omega^j$ or
$I_{\omega'}^i\supseteq I_\omega^j$.
  \label{eq:omega-omega}
  \end{numitem1}
In particular, this implies that $P_\omega$ lies on the right from
$P_{\omega'}$, i.e. each point $v^i_\omega$ lies on the right from
$v^i_{\omega'}$ in the corresponding horizontal line, with possibly
$v^i_\omega=v^i_{\omega'}$.

For the closed region $Z(\omega',\omega)$ bounded by the paths $P_{\omega'}$
(the left boundary) and $P_{\omega}$ (the right boundary), we can consider a
g-tiling $T$. It is defined by axioms (T2),(T3) as before and slightly modified
axioms (T1),(T4), where (cf. Subsection~\ref{ssec:tiling}): in~(T1), the first
condition is replaced by the requirement that each edge in $(P_{\omega'}\cup
P_\omega)-(P_{\omega'}\cap P_\omega)$ belong to exactly one tile; and in~(T4),
it is now required that $D_T\cup \sigma(P_{\omega'}\cap P_\omega)$ be simply
connected, where, as before, $D_T$ denotes $\cup(\sigma(\tau)\colon \tau\in
T)$. Also one should include in the graph $G_T=(V_T,E_T)$ all common vertices
and edges of $P_{\omega'},P_\omega$. Note that such a $T$ possesses the
following properties:
   \begin{numitem1}
(i) the union of tiles in $T$ and the edges in $P_{\omega'}\cap P_\omega$ is
exactly $Z(\omega',\omega)$; and (ii) all vertices in $bd(Z(\omega',\omega))
=P_{\omega'}\cup P_\omega$ are nonterminal.
   \label{eq:bd_oo}
   \end{numitem1}

This is seen as follows. Let an edge $e$ of height $h$ belong to two tiles
$\tau,\tau'\in T$ (where the height of an edge is the half-sum of the heights
of its ends). Suppose $\tau\cup\tau'$ contains no edge of height $h$ lying on
the right from $e$. Then one of these tiles, say, $\tau$, is black and either
$e=br(\tau)$ or $e=rt(\tau)$. Assuming $e=br(\tau)$ (the other case is
similar), take the white tile $\tau''$ with $rt(\tau'')=rt(\tau)$. Then the
edge $br(\tau'')$ has height $h$ and lies on the right from $e$. So $e$ cannot
belong to the ``right boundary'' of $\cup(\tau\in T)$. By similar reasonings,
$e$ cannot belong to the ``left boundary'' of $\cup(\tau\in T)$. This
yields~(i). Property~(ii) for the vertices $z_0,z_n$ easily follows from~(i),
and is trivial for the other vertices.

In view of~\refeq{bd_oo}(i), we may speak of $T$ as a g-tiling \emph{on}
$Z(\omega',\omega)$. We proceed with several observations.\smallskip

(O1) When $\omega',\omega$ obey~\refeq{omega-omega}, at least one g-tiling,
even a pure tiling, on $Z(\omega',\omega)$ does exist (assuming
$\omega\ne\omega'$). (This generalizes a result in~\cite{El} where $\omega'=id$
and $\omega$ is arbitrary; in this case~\refeq{omega-omega} is obvious.) It can
be constructed by the following procedure, that we call \emph{stripping
$Z(\omega',\omega)$ along $P_\omega$ from below}. At the first iteration of
this procedure, we take the minimum $i$ such that the edges
$e^i_{\omega'},e^i_\omega$ are different, and take the edge $e^k_\omega$ such
that $\omega'^{-1}(i)=\omega^{-1}(k)=:c$. Then $k>i$. Let $P'$ be the part of
$P_\omega$ from $v^{i-1}_{\omega'}=v^{i-1}_{\omega}$ to $v^{k-1}_\omega$.
Using~\refeq{omega-omega} for this $i$ and $j=i,\ldots,k-1$, one can see that
the label $\omega^{-1}(j)=:c_j$ of each edge $e^j_\omega$ of $P'$ is greater
than $c$. So we can form the $cc_j$-tiles $\tau_j$ with
$br(\tau_j)=e^j_\omega$. Then $b\ell(\tau_i)=e^i_{\omega'}$ and
$rt(\tau_{k-1})=e^k_{\omega}$. Therefore, these tiles determine $c$-strip $Q$
connecting the edge $e^i_{\omega'}$ to the edge $e^k_\omega$. Replace in
$P_\omega$ the subpath $P'$ followed by the edge $e^k_\omega$ by the edge
$e^i_{\omega'}$ followed by the left boundary of $Q$ (beginning at
$v^i_{\omega'}$ and ending at $v^k_\omega$). The obtained path $\tilde P$
determines the permutation $\omega''$ (i.e. $\tilde P=P_{\omega''}$) for which
the set $I_{\omega''}^j$ is expressed as $I_\omega^{j-1}\cup\{c\}$ for
$j=i,\ldots,k$, and is equal to $I_\omega^j$ otherwise. Using this, one can
check that \refeq{omega-omega} continues to hold when $\omega$ is replaced by
$\omega''$. At the second iteration, we handle the pair $(\omega',\omega'')$
(satisfying $|P_{\omega'}\cap P_{\omega''}|>|P_{\omega'}\cap P_\omega|$) in a
similar way, and so on until the current ``right'' path turns into
$P_{\omega'}$. The tiles constructed during the procedure give a pure tiling
$T$ on $Z(\omega',\omega)$, as required. \smallskip

(O2) Conversely, let $\omega',\omega$ be two permutations on $[n]$ such that
$P_\omega$ lies on the right from $P_{\omega'}$. Suppose that there exists a
pure tiling $T$ on $Z(\omega',\omega)$. Let $i,j\in[n]$. Then $I_{\omega'}^i$
is the set of edge labels in the subpath $P'$ of $P_{\omega'}$ from its
beginning to $v^i_{\omega'}$, and $I_\omega^j$ is the set of edge labels in the
subpath $P$ of $P_{\omega}$ from its beginning to $v^j_{\omega}$. Take
arbitrary elements $a\in I_{\omega'}^i$ and $b\in I_\omega^j$, and consider in
$T$ the $a$-strip $Q_a$ and the $b$-strip $Q_b$, each having the first edge on
$P_{\omega'}$ and the last edge on $P_\omega$. So $Q_a$ begins with an edge in
$P'$; if it ends with an edge in $P$, then $a$ is a common element of
$I_{\omega'}^i,I_\omega^j$. In its turn, $Q_b$ ends in $P$; if it begins in
$P'$, then $b\in I_{\omega'}^i\cap I_\omega^j$. On the other hand, if $Q_a$
ends in $P_\omega-P$ and $Q_b$ begins in $P_{\omega'}-P'$, then these strips
must cross at some tile $\tau\in T$. Moreover, since $T$ is a pure tiling, such
a $\tau$ is unique (which is easy to show). It is clear that $Q_a$ contains the
edge $b\ell(\tau)$, and $Q_b$ contains $br(\tau)$. Hence, $a<b$,
implying~\refeq{omega-omega}.
\smallskip

(O3) One more useful observation is that, for permutations $\omega'\ne\omega$,
the existence of a pure tiling on $Z(\omega',\omega)$ (subject to the
requirement that $P_\omega$ lie on the right from $P_{\omega'}$) is equivalent
to satisfying the \emph{weak Bruhat relation} $\omega'\prec\omega$. The latter
means that $Inv(\omega')\subset Inv(w)$, where $Inv(w'')$ denotes the set of
inversions for a permutation $\omega''$. This can be seen as follows. Let
$P_\omega$ lie on the right from $P_{\omega'}$ and let $T$ be a pure tiling on
$Z(\omega',\omega)$. It is easy to see that there exists a tile $\tau\in T$
that contains two consecutive edges $e^i_{\omega'},e^{i+1}_{\omega'}$ in
$P_{\omega'}$. Then $e^i_{\omega'}=b\ell(\tau)$ and $e^{i+1}_{\omega'}=\ell
t(\tau)$, and the label $c:=\omega'^{-1}(i)$ is less than the label
$c':=\omega'^{-1}(i+1)$. Therefore, $(c,c')\not\in Inv(\omega')$. Remove $\tau$
from $T$, obtaining a pure tiling on $Z(\omega'',\omega)$, where $\omega''$ is
formed from $\omega'$ by swapping $c$ and $c'$, i.e. $\omega''(c)=i+1$ and
$\omega''(c')=i$. Then $Inv(\omega'') =Inv(\omega')\cup\{(c,c')\}$. Repeat the
procedure for $\omega''$. Eventually, when the current left path turns into
$P_\omega$, we reach $\omega$. This yields $Inv(\omega')\subset Inv(\omega)$.
Note that the number $|T|$ of steps in the procedure is equal to
$|Inv(\omega)-Inv(\omega')|$, implying
$|\mathfrak{S}_T|=\ell(\omega)-\ell(\omega')+n+1$. \smallskip

(O4) Conversely, let $\omega'\prec \omega$. Take $i\in[n]$ such that
$e^i_{\omega'}\ne e^i_\omega$ and $e^j_{\omega'}= e^j_\omega$ for
$j=1,\ldots,i-1$. Let $c:=\omega^{-1}(i)$ and $k:=\omega'(c)$ (i.e. the edges
$e^i_\omega$ and $e^k_{\omega'}$ have the same label $c$). Clearly $k>i$. Let
$d$ be the label of $e^{k-1}_{\omega'}$. Then the inequality $k-1\ge i$ and the
choice of $i$ imply $\omega(d)>i$. So we have $\omega'(d)<\omega'(c)$ and
$\omega(d)>\omega(c)$. This is possible only if $d<c$ and $(d,c)\in
Inv(\omega)-Inv(\omega')$ (in view of $Inv(\omega')\subset Inv(\omega)$). Using
these facts, one can form the tile $\tau$ with $b\ell(\tau)=e^{k-1}_{\omega'}$
and $\ell t(\tau)=e^{k}_{\omega'}$ (taking into account that $d<c$). Replacing
in $P_{\omega'}$ the edges $e^{k-1}_{\omega'},e^{k}_{\omega'}$ by the other two
edges of $\tau$, we obtain the path corresponding to the permutation $\omega''$
satisfying $Inv(\omega'')=Inv(\omega')\cup \{(\tilde c,\tilde c')\}$. Repeating
the procedure step by step, we eventually reach $\omega$, and the tiles
constructed during the process give the desired pure tiling on
$Z(\omega',\omega)$.
\smallskip

(O5) Using flip techniques elaborated in~\cite{DKK-09}, one can show the
existence of a pure tiling on $Z(\omega',\omega)$ provided that a g-tiling $T$
on it exists. More precisely, extend $T$ to a g-tiling $\tilde T$ on $Z_n$ by
adding a pure tiling on $Z(id,\omega')$ and a pure tiling on
$Z(\omega,\omega_0)$. If $\tilde T$ has a black tile (contained in $T$), then,
as is shown in~\cite{DKK-09}~(Proposition~5.1), one can choose a black tile
$\tau$ with the following properties: (a) there are five nonterminal vertices
$Xi,Xk,Xij,Xik,Xjk$ (in set notation) such that $i<j<k$, the vertices
$Xij,Xik,Xjk$ are connected by edges to $t(\tau)$ (lying in the cone of $\tau$
at $t(\tau)$); and (b) replacing $Xik$ by $Xj$ (the \emph{lowering flip} w.r.t.
the above quintuple) makes the spectrum of some other g-tiling $\tilde T'$ on
$Z_n$. Moreover, the transformation $\tilde T\mapsto \tilde T'$ is local and
involves only tiles having a vertex at $Xik$ (and the new tiles have a vertex
at $Xj$). This implies that the tiles of $\tilde T'$ contained in
$Z(\omega',\omega)$ form a g-tiling $T'$ on it, whereas the other tiles (lying
within $Z(id,\omega')\cup Z(\omega,\omega_0)$) are exactly the same as those in
$\tilde T$. Since the flip decreases (by one) the total size of sets in the
spectrum, we can conclude that a pure tiling for $Z(\omega',\omega)$ does
exist.
\smallskip

(O6) Suppose that the path $P_\omega$ lies on the right from $P_{\omega'}$ and
that the sets $I_{\omega'}^i,I_\omega^j$ (over all $i,j=1,\ldots,n$) form a
ws-collection. By Theorem~B, ~$\Cscr$ is extendable to a largest ws-collection
$\Cscr'$, and by Theorem~\ref{tm:til-ws}, there exists a g-tiling $T$ on $Z_n$
with $\mathfrak{S}_T=\Cscr'$. All sets in $\Cscr$ are nonterminal vertices of
$T$, and by Proposition~\ref{pr:4edges}(i), all edges in $P_{\omega'}$ and
$P_\omega$ are edges of $T$. Let $T'$ be the set of tiles $\tau\in T$ such that
$\sigma(\tau)$ lies in the simply connected region in $D_T$ bounded by
$\sigma(P_{\omega'})$ and $\sigma(P_\omega)$. Then $T'$ is a g-tiling on
$Z(\omega',\omega)$.
\smallskip

Summing up the above observations, we obtain the following
  \begin{theorem} \label{tm:equiv}
~For distinct permutations $\omega',\omega$ on $[n]$, the following are
equivalent:

{\rm(i)} $\omega',\omega$ satisfy~\refeq{omega-omega};

{\rm(ii)} $\omega',\omega$ satisfy the weak Bruhat relation
$\omega'\prec\omega$;

{\rm(iii)} $P_\omega$ lies on the right from $P_{\omega'}$ and
$Z(\omega',\omega)$ admits a pure tiling;

{\rm(iv)} $P_\omega$ lies on the right from $P_{\omega'}$ and
$Z(\omega',\omega)$ admits a generalized tiling;

{\rm(v)} $P_\omega$ lies on the right from $P_{\omega'}$ and
$\{I_\omega^i,I_{\omega'}^i\colon i=1\ldots,n\}$ is a ws-collection.
  \end{theorem}

Now return to the case of one permutation $\omega$. Let us apply the procedure
of stripping $Z(\omega,\omega_0)$ along $rbd(Z_n)$ \emph{from above} (cf. the
procedure in~(O1)). One can check that the pure tiling $T''$ on
$Z(\omega,\omega_0)$ obtained in this way has the spectrum $\mathfrak{S}_{T''}$
to be exactly the $\omega$-checker $\Cscr^0_\omega$ defined in~\refeq{C0}. Such
a $T''$ for $n=5$ and $\omega=31524$ is illustrated in the picture.
   \begin{center}
  \unitlength=1mm
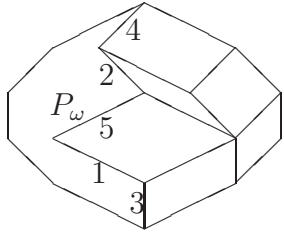
  \begin{picture}(40,30)
 %
  \put(20,0){\line(-2,1){12}}
  \put(20,6){\line(-2,1){12}}
  \put(32,12){\line(-2,1){12}}
  \put(26,18){\line(-2,1){12}}
  \put(32,24){\line(-2,1){12}}
  \put(8,6){\line(-1,1){6}}
  \put(38,18){\line(-1,1){6}}
  \put(20,18){\line(-1,1){6}}
  \put(32,12){\line(-1,1){6}}
  \put(2,12){\line(0,1){6}}
  \put(20,0){\line(0,1){6}}
  \put(32,6){\line(0,1){6}}
  \put(38,12){\line(0,1){6}}
  \put(20,0){\line(2,1){12}}
  \put(20,6){\line(2,1){12}}
  \put(8,12){\line(2,1){12}}
  \put(8,24){\line(2,1){12}}
  \put(2,18){\line(1,1){6}}
  \put(14,24){\line(1,1){6}}
  \put(26,18){\line(1,1){6}}
  \put(32,12){\line(1,1){6}}
  \put(32,6){\line(1,1){6}}
 \put(7.5,15){$P_\omega$}
 \put(18,2){3}
 \put(13,6){1}
 \put(14,12){5}
 \put(14,19){2}
 \put(17.5,25){4}
  \end{picture}
   \end{center}

\noindent \textbf{Remark.} We refer to $T''$ as above as the \emph{standard
tiling} on $Z(\omega,\omega_0)$. This adopts, to the $\omega$ case, terminology
from~\cite{DKK-09} where a similar tiling for $\omega=id$ is called the
standard tiling on the zonogon $Z_n$. The spectrum of the latter consists of
all intervals in $[n]$; this is just the collection of $X\cap Y$ over all
vertices $X$ in $\ell bd(Z_n)$ (i.e. the ideals for $id$) and all vertices $Y$
in $rbd(Z_n)$ (i.e. the ideals for $\omega_0$). The spectrum of $T''$ possesses
a similar property: it is the collection $\{I_\omega^i\cap I_{\omega_0}^j\colon
i,j\in[n]\}$ (with repeated sets ignored). (Cf.~\refeq{C0} where the term
$[j..n]$ is just $j$-th ideal for $\omega_0$. One can see that withdrawal of
the condition $j\le\omega^{-1}(k)$ results in the same collection
$\Cscr_\omega^0$.) Also one can check that the same tiling $T''$ is obtained if
we make stripping $Z(\omega,\omega_0)$ \emph{along} $P_\omega$ from above. It
turns out that a similar phenomenon takes place for any permutations
$\omega',\omega$ obeying~\refeq{omega-omega}: one can show that stripping
$Z(\omega',\omega)$ along $P_\omega$ (or along $P_{\omega'}$) from above
results in a pure tiling on $Z(\omega',\omega)$ whose spectrum consists of all
(different) sets of the form $I_{\omega'}^i\cap I_{\omega}^j$, $i,j\in[n]$; we
may refer to it as the \emph{standard tiling for} $(\omega',\omega)$.\medskip

By reasonings in Section~\ref{sec:omega}, for any maximal $\omega$-chamber
ws-collection $\Cscr$, the collection $\Dscr:=\Cscr\cup \Cscr^0_\omega$ is a
largest ws-collection. So, by Theorem~\ref{tm:til-ws}, $\Dscr=\mathfrak{S}_T$
for some g-tiling $T$ on $Z_n$. In view of Proposition~\ref{pr:4edges}, $T$
must include the subtiling $T''$ as above. Then each edge in $(P_{id}\cup
P_\omega)-(P_{id}\cap P_\omega)$ belongs to exactly one tile in $T':=T-T''$;
this implies that $T'$ is a g-tiling on $Z(id,\omega)$, and we can conclude
that $\Cscr=\mathfrak{S}_{T'}$. Conversely (in view of
Theorem~\ref{tm:checker}), for any g-tiling $T'$ on $Z(id,\omega)$,
~$\mathfrak{S}_{T'}$ is a maximal $\omega$-chamber collection. One can see that
the role of $\omega$-checker can be played, in essence, by the spectrum of
\emph{any} pure, or even generalized, tiling $T''$ on $Z(\omega,\omega_0)$
(i.e. Theorem~\ref{tm:checker} remains valid if we take $\mathfrak{S}_{T''}$ in
place of $\Cscr^0_\omega$). Thus, we obtain the following
   \begin{corollary} \label{cor:chamb-til}
{\rm(a)} Any maximal $\omega$-chamber ws-collection in $2^{[n]}$ is the
spectrum of some g-tiling on $Z(id,\omega)$, and vice versa. In particular, any
$\omega$-chamber set $X$ lies on the left from the path $P_\omega$ (regarding
$X$ as a point).

{\rm(b)} For any fixed g-tiling $T''$ on $Z(\omega,\omega_0)$, ~$X\subseteq[n]$
is an $\omega$-chamber set if and only if
$X\not\in\mathfrak{S}_{T''}-\Iscr_\omega$ and $X\wsep \mathfrak{S}_{T''}$,
where $\Iscr_\omega:=\{I_\omega^0,\ldots,I_\omega^n\}$.

{\rm(c)} $X\subseteq[n]$ is an $\omega$-chamber set if and only if $X\wsep
\Iscr_\omega$ and $X\lessdot I_\omega^{|X|}$.
   \end{corollary}

(Note that $X\wsep \Iscr_\omega$ and $X\lessdot I_\omega^{|X|}$ easily imply
that either $X\lessdot I_\omega^i$ or $X\supseteq I_\omega^i$ holds for each
$i$.) In light of~(a) in this corollary, it is not confusing to refer to an
$\omega$-chamber set $X$ as a \emph{left set} for $\omega$. A reasonable
question is how to characterize, in terms of $\omega$, the corresponding sets
$X$ lying on the right from $P_\omega$ (i.e. when $X$ belongs to the spectrum
of some g-tiling on the right region $Z(\omega,\omega_0)$ for $\omega$). To do
this, suppose we turn the zonogon at $180^\circ$ and reverse the edges. Then
the path $P_{\omega}$ reverses and becomes the path $P_{\bar \omega}$ for
$\bar\omega:=\omega_0\omega$, each set $X\subseteq[n]$ is replaced by $[n]-X$,
and $Z(\omega,\omega_0)$ turns into the region $Z(id,\bar\omega)$ lying on the
left from $P_{\bar\omega}$. The spectrum of any g-tiling on $Z(id,\bar\omega)$
is formed by left sets $[n]-X$ for $\bar\omega$, and when going back to the
original $X$ (thus lying in the region $Z(\omega,\omega_0)$), we observe that
$X$ is characterized by the condition:
  \begin{numitem1}
for each $i\in X$, ~$X$ contains all $j\in[n]$ such that $j>i$ and
$\bar\omega(j)>\bar\omega(i)$; equivalently: $X$ contains all $j$ such that
$j>i$ but $\omega(j)<\omega(i)$.
  \label{eq:co-chamb}
  \end{numitem1}
We refer to such an $X$ as a \emph{right set} for $\omega$.
\smallskip

Finally, consider again two permutations $\omega',\omega$ and let
$\omega'\prec\omega$.  Representing the middle region $Z(\omega',\omega)$ as
the intersection of $Z(id,\omega)$ and $Z(\omega',\omega_0)$ and relying on the
analysis above, we can conclude with the following generalization of Theorem~A.
\medskip

 \noindent\textbf{Theorem A$'$}
\emph{~Let $\omega,\omega'$ be two permutations on $[n]$ satisfying
$\omega'\prec\omega$. Then all maximal ws-collections $\Cscr\subseteq
2^{[n]}$ whose members $X$ are simultaneously left sets for $\omega$
and right sets for $\omega'$ have the same cardinality; namely,
$|\Cscr|=\ell(\omega)-\ell(\omega')+n+1$. These collections $\Cscr$ are
precisely the spectra of g-tilings on $Z(\omega',\omega)$.}
 \medskip

The sets $X$ figured in this theorem can be alternatively characterized by the
condition: for $i,j\in[n]$, if $\omega'(i)\prec\omega'(j)$,
~$\omega(i)\prec\omega(j)$ and $j\in X$, then $i\in X$, i.e. $X$ is an ideal of
the partial order on $[n]$ that is the intersection of two linear orders, one
being generated by the path $P_{\omega'}$, and the other by $P_\omega$.


\begin{thebibliography}{99}

\bibitem{BFZ} A.~Berenstein, S.~Fomin, and  A.~Zelevinsky,
Parametrizations of canonical bases and totally positive matrices,
\textsl{Adv.~Math.} \textbf{122} (1996) 49--149.
 %
\bibitem{Bir} G.~Birkhoff, \textsl{Lattice Theory}, 3rd ed., Amer. Math. Soc.,
Providence, Rh.~Island, 1967.
  %
\bibitem{DKK-08} V.~Danilov, A.~Karzanov and G.~Koshevoy, Tropical
Pl\"ucker functions and their bases, in: \textsl{Tropical and Idempotent
Mathematics} (eds. G.L.~Litvinov and S.N.~Sergeev), \textsl{Contemporary Math.}
{\bf 495} (2009) 127--158.
 %
\bibitem{DKK-09} V.~Danilov, A.~Karzanov and G.~Koshevoy, Pl\"ucker environments,
wiring and tiling diagrams, and weakly separated set-systems,
\textsl{Adv.~Math.} \textbf{224} (2010) 1--44.
 %
\bibitem{El} S.~Elnitsky, Rhombic tilings of polygons and classes of reduced
words in Coxeter groups, \textsl{J.~Comb. Theory}, Ser.~A, {\bf 77} (1997)
193--221.
 %
\bibitem{Fan} C.K.~Fan, A Hecke algebra quotient and some combinatorial
applications, \textsl{J.~Alg.~Combin.} {\bf 5} (1996) 175--189.
 %
\bibitem{HS} A.~Henriques and D.E.~Speyer, The multidimensional
cube recurrence, \textsl{ArXiv}:0708.2478v1[math.CO], 2007.
 %
\bibitem{Kn} D.E.~Knuth, \textsl{Axioms and hulls}, Lecture Notes in Computer
Science, vol.~{\bf 606}, 1992.
 %
\bibitem{LZ} B.~Leclerc and A.~Zelevinsky: Quasicommuting families of
quantum Pl\"ucker coordinates, \textsl{Amer. Math. Soc. Trans., Ser.~2}
~\textbf{181} (1998) 85--108.
 %
\bibitem{Stem} J.~Stembridge, On the fully commutative elements of Coxeter
groups, \textsl{J.~Alg.~Combin.} {\bf 5} (1996) 353--385.
 %
 \end{thebibliography}
\end{document}